\documentclass[a4paper, 12pt]{amsart}
\usepackage{fullpage}
\usepackage[utf8]{inputenc}
\usepackage{amsmath,amsthm,amssymb,amscd}
\usepackage{mathtools}
\usepackage{mathrsfs}
\usepackage{tikz-cd}
\usepackage{multicol}
\usepackage{stmaryrd}
\usepackage{dsfont}
\usepackage{cancel}
\usepackage{amsfonts}
\usepackage{yfonts}
\usepackage{esint}
\usepackage{graphicx}
\graphicspath{ {./images/} }
\newsavebox\MBox
\usepackage{hyperref}
\usepackage{nameref}
\usepackage{etoolbox}
\usepackage{comment}
\usepackage{ulem}

\usepackage{qtree}
\usepackage{forest}

\usepackage{enumitem}
\setlist[enumerate,1]{label={(\alph*)}}

\usepackage{appendix}

\newtheorem{thm}{Theorem}
\newtheorem{deff}{Definition}

\newtheorem{theorem}{Theorem}[section]

\newtheorem{proposition}[theorem]{Proposition}
\newtheorem{lemma}[theorem]{Lemma}
\newtheorem{corollary}[theorem]{Corollary}

\theoremstyle{definition}
\newtheorem{definition}[theorem]{Definition}

\newtheorem{notation}[theorem]{Notation}
\newtheorem{terminology}[theorem]{Terminology}

\theoremstyle{remark}
\newtheorem{remark}[theorem]{Remark}
\newtheorem{example}[theorem]{Example}

\newcommand*{\Q}{\mathbb{Q}}
\newcommand*{\R}{\mathbb{R}}
\newcommand*{\Z}{\mathbb{Z}}
\newcommand*{\N}{\mathbb{N}}
\newcommand{\C}{\mathbb{C}}

\newcommand*{\F}{\mathbb{F}}

\newcommand*{\A}{\mathbb{A}}
\renewcommand*{\P}{\mathbb{P}}
\newcommand*{\T}{\mathcal{T}}
\newcommand*{\fq}{\mathfrak q}
\newcommand*{\fm}{\mathfrak m}
\newcommand*{\fp}{\mathfrak p}
\newcommand*{\fu}{\mathfrak u}

\newcommand*{\mC}{\mathcal{C}}

\newcommand*{\mM}{\mathcal{M}}

\newcommand*{\mI}{\mathcal{I}}
\newcommand*{\mO}{\mathcal{O}}

\newcommand\pprec{\preccurlyeq}
\newcommand\ssucc{\succcurlyeq}

\renewcommand*{\a}{\alpha}
\renewcommand*{\b}{\beta}
\newcommand*{\g}{\gamma}
\renewcommand*{\d}{\delta}

\newcommand*{\e}{\varepsilon}

\newcommand*{\La}{\Lambda}
\newcommand*{\s}{\sigma}
\renewcommand*{\t}{\tau}

\newcommand{\w}{\omega}
\newcommand{\W}{\Omega}

\newcommand*{\emp}{\varnothing}

\newcommand*{\pa}{\partial}

\renewcommand*{\S}{\Sigma}
\renewcommand*{\phi}{\varphi}

\newcommand*{\we}{\wedge}
\renewcommand{\tilde}{\widetilde}

\DeclarePairedDelimiterX\braket[2]{\langle}{\rangle}{#1 \delimsize\vert #2}
\newcommand{\pull}[1]{{#1}^\star}
\newcommand{\expinv}[1]{{#1}^\diamond}

\newcommand{\mdl}[1]{
    \ifnum #1=1
        {\mM_{k_1+1,I}({\b_1})}
        \else
    \ifnum #1=2
        {\mM_{k_2+1,J}({\b_2})}
        \else
    {\mM_{k+1,l}(\b)}\fi\fi
}
\newcommand{\qor}{Q}
\newcommand{\oqb}[1]{
    \ifnum #1=-3
        {\fq^{\b}_{-1,l}}
        \else
    \ifnum #1=-2
        {\fq^{\b_2}_{-1,|J|}}
        \else
    \ifnum #1=-1
        {\fq^{\b_1}_{-1,|I|}}
        \else
    \ifnum #1=1
        {\fq^{\b_1}_{k_1,|I|}}
        \else
    \ifnum #1=0
        {\fq^{\b_0}_{1,0}}
        \else
    \ifnum #1=2
        {\fq^{\b_2}_{k_2,|J|}}
    \else
    {\fq^{\b}_{k,l}}\fi\fi\fi\fi\fi\fi
}

\newcommand{\oq}[1]{
    \ifnum #1=-1
        {\fq_{-1,l}}
        \else
    \ifnum #1=1
        {\fq_{(1:3)+1+(3:3),|I|}}
        \else
    \ifnum #1=0
        {\fq^{\b_0}_{1,0}}
        \else
    \ifnum #1=2
        {\fq_{(2:3),|J|}}
    \else
    {\fq_{k,l}}\fi\fi\fi\fi
}

\newcommand{\om}[1]{
    \ifnum #1=-1
        {\fm^\g_{-1}}
        \else
    \ifnum #1=1
        {\fm^\g_{|(1:3)|+1+|(3:3)|}}
        \else
    \ifnum #1=0
        {\fm^{\g}_{1}}
        \else
    \ifnum #1=2
        {\fm^\g_{|(2:3)|}}
    \else
    {\fm^\g_{k}}\fi\fi\fi\fi
}
\newcommand{\omb}[1]{
    \ifnum #1=-1
        {\fm^{\b,\g}_{-1}}
        \else
    \ifnum #1=1
        {\fm^{\b,\g}_{(1:3)+1+(3:3)}}
        \else
    \ifnum #1=0
        {\fm^{\b,\g}_{1}}
        \else
    \ifnum #1=2
        {\fm^{\b,\g}_{(2:3)}}
    \else
    {\fm^{\b,\g}_{k}}\fi\fi\fi\fi
}
\newcommand{\orientor}[1]{\ensuremath{#1}-orientor}
\newcommand{\eorientor}[1]{\ensuremath{#1}-endo-orientor}
\newcommand{\Otm}{O}
\newcommand{\efield}{\mathbb E}

\newcommand{\target}{{\mathcal T}}

\DeclareMathOperator{\Ima}{Im}
\newcommand{\Id}{\text{Id}}
\newcommand{\bu}{\bullet}

\newcommand{\cort}[1]{\mathcal K_{#1}}
\newcommand{\zcort}[1]{\mathbb K_{#1}}
\newcommand{\ort}[1]{\mathcal L_{#1}}
\newcommand{\zort}[1]{\mathbb L_{#1}}
\newcommand{\eort}{\mathcal E}

\newcommand{\rort}{\mathcal{R}}
\newcommand{\lort}{\mathcal{L}}
\newcommand{\oddpairing}[1]{\left\langle#1\right\rangle_{\text{odd}}}
\newcommand{\evenpairing}[1]{\left\langle#1\right\rangle_{\text{even}}}

\tikzset{
  symbol/.style={
    draw=none,
    every to/.append style={
      edge node={node [sloped, allow upside down, auto=false]{$#1$}}}
  }
}

\title{The Fukaya $A_\infty$ algebra of a non-orientable Lagrangian}

\author[O. Kedar]{Or Kedar}
\address{Institute of Mathematics\\ Hebrew University, Givat Ram\\Jerusalem, 91904, Israel } \email{or.kedar@mail.huji.ac.il}

\author[J. Solomon]{Jake P. Solomon}
\address{Institute of Mathematics\\ Hebrew University, Givat Ram\\Jerusalem, 91904, Israel } \email{jake@math.huji.ac.il}

\begin{document}

\date{Nov. 2022}
\begin{abstract}
Let $L\subset X$ be a not necessarily orientable relatively $Pin$ Lagrangian submanifold in a symplectic manifold $X$. 
We construct a family of cyclic unital curved $A_\infty$ structures on differential forms on $L$ with values in the local system of graded non-commutative rings given by the tensor algebra of the orientation local system of $L$.
The family of $A_\infty$ structures is parameterized by the cohomology of $X$ relative to $L$ and satisfies properties analogous to the axioms of Gromov-Witten theory. On account of the non-orientability of $L,$ the evaluation maps of moduli spaces of $J$-holomorphic disks with boundary in $L$ may not be relatively orientable. To deal with this problem, we use recent results on orientor calculus.
\end{abstract}
\maketitle
\pagestyle{plain}
\tableofcontents

\section{Introduction}
\label{introduction section}
\subsection{Overview}
\label{overview section}
Let $X$ be a symplectic manifold and let $L\subset X$ be a not necessarily orientable relatively $Pin$ Lagrangian submanifold. Let $J$ be an $\omega$-tame almost complex structure.
We present a construction of the Fukaya $A_\infty$ algebra of $L$ including cyclic symmetry, which extends the constructions given in~\cite{Fukaya,Sara1} to the non-orientable case. This algebra encodes the geometry of the moduli spaces of $J$-holomorphic stable disk maps with boundary in $L.$  Previous work in the non-orientable case is limited, and in particular, it applies only over fields of characteristic $2$ or when $L$ is orientable relative to a local system on $X$. See Section~\ref{ssec:context}. Following~\cite{Sara1}, our construction includes bulk deformations to obtain a family of cyclic $A_\infty$ algebras parameterized by the cohomology of $X$ relative to $L,$ and we show this family satisfies analogs of the axioms of Gromov-Witten theory.

The non-orientability of $L$ generates a number of phenomena unfamiliar from the orientable case. To obtain an $A_\infty$ algebra from $L$, it is necessary to allow these phenomena to interact naturally so that they counterbalance each other. A brief explanation follows.

Unlike the orientable case, the evaluation maps of moduli spaces of $J$-holomorphic disks with boundary on $L$ need not be relatively orientable, and thus can only be used to push-forward differential forms with appropriate local coefficients. However, such local coefficients undergo monodromy under parallel transport around the boundary of a $J$-holomorphic disk. Consequently, apparently spurious signs arise in expressions of the form 
\[
\fm_{k_1}(\alpha_1,\ldots,\alpha_{i-1},\fm_{k_2}(\alpha_i,\ldots,\alpha_{i+k_2-1}),\alpha_{i+k_2},\ldots)
\]
from the local coefficients of the inputs $\alpha_{i+k_2},\ldots$, which need to be transported around the boundary of the $J$-holomorphic disks giving rise to the operation~$\fm_{k_2}.$ 

Furthermore, the Maslov class of $L$ can be odd when $L$ is not orientable. Consequently, for the $A_\infty$ operations to be graded correctly, it is necessary to work over a Novikov ring that includes a formal variable of odd degree. For the $A_\infty$ relations to faithfully encode the structure of the boundary strata of moduli spaces of $J$-holomorphic disks, the odd degree formal variable should not square to zero. That is, the Novikov ring should not be graded commutative.

To allow the above phenomena to interact naturally, we endow the orientation local system of $L$ with degree $-1$ and give it the role of the odd degree formal variable in the Novikov ring. The graded non-commutativity of this ``formal variable'' precisely compensates for the signs arising from parallel transport of local coefficients and also plays an important role in the proof of cyclic symmetry. Relative orientation local systems of evaluation maps of moduli spaces of $J$-holomorphic disks inherit a grading from the orientation local system of $L.$ This degree enters the push-foward of differential forms with local coefficients and is eventually responsible for the grading of the $A_\infty$ operations.

To prove $A_\infty$ relations, we must systematically keep track of the interactions between local coefficients, gradings, fiber products, boundaries, Stokes' theorem and moduli spaces of $J$-holomorphic disks. This is accomplished using the notion of an orientor and the associated orientor calculus introduced in~\cite{orientors} and summarized in Section~\ref{orientors section}.

Building on the work of~\cite{Elad1,Sara1,Sara3,Sara2}, we plan to use the $A_\infty$ algebra of $L$ to define open Gromov-Witten invariants for $L$ and to study the structure of these invariants. When $L$ is fixed by an anti-symplectic involution and $\dim L = 2,$ we expect the open Gromov-Witten invariants of $L$ to recover Welschinger's real enumerative invariants~\cite{Welschinger-invariants-lower-bounds}. When $\dim L > 2$ or when $L$ is not fixed by an anti-symplectic involution, it appears that $A_\infty$ algebra of $L$ plays an essential role in the definition of invariants.

Lagrangian submanifolds arise naturally as the real points of smooth complex projective varieties that are invariant under complex conjugation. Natural constructions in algebraic geometry, such as blowups and quotients, give rise to non-orientable Lagrangians.  Examples of computations of open Gromov-Witten-Welschinger invariants for non-orientable Lagrangian submanifolds of dimension $2$ appear in~\cite{Horev,Kharlamov-new-logarithmic-equivalence,Kharlamov-logarithmic-equivalence,Kharlamov-enumeration-rational-curves,Kharlamov-logarithmic-asymptotics,Kharlamov-toric-surfaces,Kharlamov-degree-3,Kharlamov-small-non-toric,Kharlamov-degree-2,Kharlamov-degree-2-erratum}.
\subsection{Context}\label{ssec:context}
In \cite{Fukaya}, a construction of the Fukaya $A_\infty$ algebra structure on a version of singular chains of $L$ with $\Q$ coefficients is provided when $L$ is orientable. In~\cite{Fukaya-cyclic-symmetry}, a construction of the Fukaya $A_\infty$ algebra structure on the differential forms of $L$ is given. The differential form construction is significantly simpler and also makes it possible to incorporate cylic symmetry in the construction. Cyclic symmetry plays a crucial role in open Gromov-Witten theory as developed in~\cite{Elad1,Fukaya-3-folds,Sara2,Sara3}.

In~\cite{Fukaya-Spherically-positive} the construction of the Fukaya $A_\infty$ algebra structure on singular chains is extended to the non-orientable case when $X$ is spherically positive, using coefficients in $\Z/2$. The spherically positive assumption is used to force stable maps with automorphisms into sufficiently high codimension that they do not lead to denominators when pushing-forward chains by the evaluation maps of moduli spaces. Since the order of automorphism groups can be even, such denominators are not allowable when working with $\Z/2$ coefficients.

Given a local system of $1$-dimensional vector spaces $\T$ on $X$, it should be possible to construct a version of the Fukaya category in which objects arise from Lagrangian submanifolds $L \subset X$ that are relatively oriented with respect to $\T.$ By relative orientation, we mean an isomorphism from $\T|_L$ to the orientation local system of $L.$ In~\cite{Seidel-Lefschetz}, such a construction is carried out when the first Chern class $c_1(X)$ is $2$-torsion, for a local system $\T$ that arises naturally in the context of gradings. The relative orientation of $L$ with respect to $\T$ forces the Maslov index $\mu : H_2(X,L) \to \Z$ to take on only even values. It follows that the evaluation maps of moduli spaces of $J$-holomorphic disks are relatively orientable~\cite{JakePhD}, so the main difficulties in the construction of the present work do not arise.

In the construction of Floer homology for two orientable Lagrangians $L_1,L_2,$ that intersect cleanly given in~\cite[Section 3.7.5]{Fukaya}, a local system arises when the intersection $L_1 \cap L_2$ is not orientable. The symplectic topology of non-orientable Lagrangians has been studied extensively in~\cite{Audin, Dai-Lag-surfaces,Evans-Klein-nonsqueeze,Givental-Lag-embedding, Nemirovski-Klein-in-space,Nemirovski-Klein-homology-class,Polterovich-surgery, Rezchikov,Shevchishin-Klein,Shevchishin-Smirnov-projective-plane}.

\subsection{Construction}
\label{introduction setting section}
Consider a symplectic manifold $(X,\w)$ with $\dim_\R X=2n$, and a connected Lagrangian submanifold $L \subset X$ with a relative $Pin^{\pm}$ structure $\fp$. Let $J$ be an $\w$-tame almost complex structure on $X$. Denote by ${\mu:H_2(X,L)\to \Z}$ the Maslov index \cite{Maslov}. Let $\Pi$ be a quotient of $H_2(X,L)$ by a possibly trivial subgroup contained in the kernel of the homomorphism ${\w\oplus \mu:H_2(X,L)\to \R\oplus \Z}$. Thus the homomorphisms $\w,\mu$ descend to $\Pi$. Denote by $\b_0$ the zero element of $\Pi$. Let $T^{\b}$ for $\b\in H_2(X,L)$ be formal variables of degree zero. Let $\F$ be a field extension of $\R$. Unless otherwise stated, tensor products are taken to be the usual graded tensor product with base field $\F$. Let $\ort{L}$ denote the local system with fiber $\F$ associated to the $\Z/2$-local system orientations of $L$, concentrated in degree $-1.$
Let ${\rort_L}$ be the local system of graded rings  
\[
{\rort_L}:=\bigoplus_{k\in\Z}\ort{L}^{\otimes k},
\] 
where negative tensor powers correspond to positive powers of the dual local system. The multiplication $m:{\rort_L}\otimes {\rort_L}\to {\rort_L}$ is given by tensor product. Note that ${\rort_L}$ is \textbf{not} graded-commutative.
Define
\[
\efield:=H^0(L;\rort_L),\qquad\tilde\La:=\Big\{\sum_{i=0}^\infty a_i T^{\b_i}\Big| a_i\in\F,\b_i\in\Pi,\w(\b_i)\geq 0,\lim_{i\to \infty}w(\b_i)=\infty \Big\}.
\]
The Novikov ring is defined by
\[
\La:=\efield\otimes\tilde\La.
\]
Observe that $\rort_L$ is a local system of $\efield$ algebras. If $L$ is orientable, $\rort_L$ is the constant sheaf with fiber $\efield.$ Otherwise, the fibers of $\rort_L$ have dimension two over $\efield.$

For any manifold $M$, possibly with corners, and a local system of graded rings $Q\to M$, denote by $A^*(M;Q)$ the ring of smooth differential forms on $M$ with values in $Q$. For $m>0$ denote by $A^m(X,L)$ the ring of differential forms that pullback to zero on $L$, and denote by $A^0(X,L)$ the functions on $X$ that are constant on $L$. The exterior derivative $d$ makes $A^*(X,L)$ into a complex.

Let $t_0,...,t_N$ be graded formal variables with degrees in $\Z$. Define graded rings
\[
R:=\La[[t_0,...,t_N]],\qquad Q:=\F[t_0,...,t_N],
\] thought of as differental graded algebras with trivial differential. Set
\[
C:=A^*(L;{\rort_L})\otimes \tilde\La[[t_0,...,t_N]],\qquad D:=A^*(X,L;Q).
\]
As ${\rort_L}$ is a local system of $\efield$ algebras, it follows that $C$ is an $R$ algebra.
Write \[{\hat H^*(X,L;Q):=H^*(D).}\]

Define a valuation 
\[
\nu:\tilde \La[[t_0,\ldots,t_N]]\to \R
\]
by
\begin{equation}\label{valuation equation}
\nu\left(\sum_{j=0}^\infty a_jT^{\b_j}\prod_{i=0}^Nt_i^{l_{ij}}
\right)=
\inf_{\begin{smallmatrix}j\\a_j\neq 0\end{smallmatrix}}\left(\w(\b_j)+\sum_{i=0}^N l_{ij}\right).    
\end{equation}
This valuation extends to a valuation on $R,C,Q,D$ and their tensor products, which we also denote by $\nu$. Define ideals 
\[
{\mI_R:=\{\a\in R\mid\nu(\a)>0\}},\qquad (\text{resp.}\quad\mI_Q:=\{\a\in Q\mid \nu(\a)>0\})
\]
of $R$ (resp. $Q$). Let 
\[
{\overline R:=R/\mI_R R}\qquad \text{and}\quad\overline C:=C/(\mI_R \cdot C)= A^*(L;{\rort_L}).
\]

For $k\geq -1,l\geq 0$ write $\mdl{3}$ for the moduli space of genus zero $J$-holomorphic open stable maps to $(X,L)$ of degree $\b\in\Pi$ with one boundary component, $k+1$ boundary marked points and $l$ interior marked points. The boundary points are labeled according to their cyclic order.  Let $evb_i^\b:\mdl{3}\to L$ and ${evi_j^\b:\mdl{3}\to X}$ denote the boundary and interior evaluation maps, where $i=0,...,k$ and $j=1,...,l$.
To streamline the exposition, we will assume that $\mdl{3}$ is a smooth orbifold with corners and $evb_0^\b$ is a proper submersion. These assumptions hold in a range of important examples~\cite[Example 1.5]{Sara1}. Our construction of cyclic unital $A_\infty$ algebras applies to arbitrary symplectic manifolds and Lagrangian submanifolds by the theory of the virtual fundamental class being developed by several authors~\cite{Fu09a,FO19,FO20,HW10,HWZ21} as explained in Section~\ref{ssec:moduli spaces}. The analogs of the unit and divisor axioms of Gromov-Witten theory given in Theorem~\ref{algebra deformation theorem}\ref{algebra deformation theorem: fundamental class},\ref{algebra deformation theorem: divisor} require compatibility of the virtual fundamental class with the forgetful map of interior marked points. This has not yet been worked out in the Kuranishi structure formalism in the context of differential forms.

Let $\cort{evb_0}$ denote the local system of relative orientations of $evb_0$.
In~\cite{orientors} we construct a family of morphisms of local systems
\[
\qor_{k,l}^\b:=\qor_{k,l}^{\left(X,L,J;\b\right)}:\bigotimes_{j=1}^k\left(evb_i^*\rort_L\right)\to \cort{evb_0}\otimes(evb_0)^*\rort_L
\]
indexed by
\[(k,l,\b)\in\Big(\Z_{\geq0}\times\Z_{\geq0}\times \Pi\Big)\setminus\Big\{ (0,0,\b_0),(1,0,\b_0),(2,0,\b_0),(0,1,\b_0)\Big\}.\]
The family $\{Q_{k,l}^\b\}$ satisfies relations that resemble $A_\infty$ relations. We recall these results in Section~\ref{Orientors over the moduli spaces section}. 



Equip $R$ with the trivial differential $d_R=0$. Consider the $R-$module $C$. For $\g\in\mathcal I_QD$ with $d\g=0,|\g|=2$ and $\b\in \Pi$, define maps
\[
\omb{3}:C^{\otimes k}\to C
\]
by \[\fm_1^{\b_0,\g}(\a)=d\a,\]
and for $k\geq 0$ when $(k,\b)\neq (1,\b_0)$, by
\[
\omb{3}(\a_1,\ldots,\a_k):=(-1)^{1+\sum_{j=1}^k(k-j)(\a_j+1)}\sum_{l\geq0}\frac{1}{l!}(evb_0^\b)_*\circ Q_{k,l}^\b\left(\bigwedge_{j=1}^l evi_j^*\g\wedge \bigwedge_{i=1}^k evb_i^*\a_i\right).
\]
The pushforward of differential forms with values in local systems is defined in Section~\ref{pushforward section}.
Define also
\[
\om{3}:C^{\otimes k}\to C
\]
by
\[
\om{3}:=\sum_{\b\in\Pi}T^\b\omb{3}.
\]


Define an integration operator $\int_{\text{odd}}:C\to R$ as follows.
On the part of $C$ of homogeneous degree with parity equal to $n$ it is set to be zero. On the part of $C$ of homogeneous degree with parity equal to $n-1$ it is set to be the unique $R$-linear extension of the standard integration operator \[\int :A^{*}(L;\ort{L})\to \F.\]
Define a pairing $\oddpairing{,}:C\otimes C\to R$ of degree $1-n$ by \[\oddpairing{\xi,\eta}:=(-1)^{\eta}\int_{\text{odd}}(\xi\we\eta).\]

\subsection{Statement of results}
Let $\mathcal R$ be a differential graded algebra over $\F$ with a valuation $\zeta_{\mathcal R}$ and let $\mC$ be a graded module over $\mathcal R$ with valuation $\zeta_\mC$. We implicitly assume that elements are of homogeneous degree and denote the degree by $|\cdot|$. Let $\d_{ij}$ be the Kronecker delta. Recall the following definition from~\cite[Definition 1.1]{Sara1}.

\begin{deff}[Cyclic unital $A_\infty$ algebra]
\label{A infinity algebra definition}
An $n$-dimentional (curved) \textbf{cyclic unital $A_\infty$ structure} on $\mC$ is a triple 
${(\{\fm_k\}_{k\geq0},\prec,\succ,e)   }$ of maps $\fm_k:\mC^{\otimes k}\to \mC[2-k]$, a pairing $\prec,\succ:\mC\otimes \mC\to \mathcal R[-n]$ and an element $e\in\mC$ satisfying the following properties. We denote by $\a$ (possibly with subscript) an element in $\mC$ and with $a$ an element in $\mathcal R$.
\begin{enumerate}
    \item\label{A infinity algebra definition: multilinearity} The $\fm_k$ are $\mathcal R$-multilinear, in the sense that 
    \[
    \fm_k(\a_1,...,\a_{i-1},a\cdot\a_i,a_k)=(-1)^{|a|\cdot|(1+\sum_{j=1}^{i-1}(|\a_j|+1)}a\cdot\fm_k(\a_1,...,\a_k)+\d_{1k}da\cdot \a_1.
    \]
    \item\label{A infinity algebra definition: pairing bilinearity} The pairing $\prec,\succ$ is $\mathcal R$-bilinear, in the sense that
    \[
    a\cdot \prec\a_1,\a_2\succ=\prec a\cdot \a_1,\a_2\succ=(-1)^{|a|\cdot (|\a_1|+1)}\prec\a_1,a\cdot\a_2\succ.
    \]
    \item\label{A infinity algebra definition: relations} The $A_\infty$ relations hold
    \[\sum_{\begin{smallmatrix}k_1+k_2=k+1\\1\leq i\leq k_1\end{smallmatrix}}(-1)^{\sum_{j=1}^{i-1}(|\a_j|+1)}\fm_{k_1}(\a_1,...,\a_{i-1},\fm_{k_2}(\a_i,...,\a_{i+k_2-1}),\a_{i+k_2},...,\a_{k})=0.\]
    \item\label{A infinity algebra definition: C valuation} $\zeta_\mC(\fm_k(\a_1,...,\a_k))\geq \sum_{j=1}^k\zeta_\mC(\a_j)$ and $\zeta_\mC(\fm_0)>0$. 
    \item\label{A infinity algebra definition: R valuation} $\zeta_\mathcal R(\prec\a_1,\a_2\succ)\geq\zeta_\mC(\a_1)+\zeta_\mC(\a_2).$
    \item \label{A infinity algebra definition: pairing antisymmetry} \[\prec\a_2,\a_1\succ=(-1)^{(|\a_1|+1)(|\a_2|+1)+1}\prec\a_1,\a_2\succ\]
    \item\label{A infinity algebra definition: cyclic} The pairing is cyclic
    \begin{align*}
    \prec \fm_k(\a_1,&...,\a_k),\a_{k+1}\succ=\\&(-1)^{(|\a_{k+1}|+1)\sum_{j=1}^k(|\a_{j}|+1) }\prec\fm_k(\a_{k+1},\a_1,...,\a_{k-1}),\a_k\succ+\d_{1k}\cdot d\prec\a_1,\a_2\succ
    \end{align*}
    \item\label{A infinity algebra definition: unit k neq 0,2} \[\fm_k(\a_1,...,\a_{i-1},e,\a_{i+1}...,\a_k)=0\quad\forall k\neq0,2\]
    \item\label{A infinity algebra definition: unit k=0} $\prec \fm_0,e\succ=0$
    \item\label{A infinity algebra definition: unit k=2} $\fm_2(e,\a)=\a=(-1)^{|\a|}\fm_2(\a,e).$
\end{enumerate}
\end{deff}

The main results of this paper are the following theorems. In fact, in the body of the paper we work with families of symplectic manifolds and Lagrangian submanifolds as explained in Section~\ref{families target definition section}. In Section~\ref{Conclusions section}, we state and prove family versions of Theorems~\ref{A infinity algebra theorem},~\ref{pseudoisotopy thm introduction} and~\ref{algebra deformation theorem}. Let $1\in A^0(L)$ denote the constant function.
\begin{thm}[$A_\infty$ structure on $C$]
\label{A infinity algebra theorem}
The triple $\left(\{\om{3}\}_{k\geq 0},\oddpairing{,},1\right)$ is a cyclic unital $n-1$ dimensional $A_\infty$-algebra structure on $C$. 
\end{thm}

Set
\[
\mathfrak R:=A^*([0,1];R),\quad \mathfrak C:=A^*(L\times [0,1];R),\quad \text{and}\quad \mathfrak D:=A^*(X\times [0,1],L\times[0,1];Q).
\]
The valuation $\nu$ induces valuations on $\mathfrak{R,C}$ and $\mathfrak D$, which we still denote by $\nu$. For $t\in [0,1]$ and $M\in\{*,L\}$, denote by
\[
j_t:M\to M\times [0,1]
\]
the inclusion $j_t(p)=(p,t).$
\begin{deff}
Let $S_1=(\fm,\prec,\succ,\mathbf{e})$ and $S_2=(\fm',\prec,\succ',\mathbf{e}')$ be cyclic unital $A_\infty$ structures on $C$. A cyclic unital \textbf{pseudoisotopy} from $S_1$ to $S_2$ is a cyclic unital $A_\infty$ structure $(\tilde\fm, {\pprec,\ssucc},\tilde{\mathbf{e}})$ on the $\mathfrak R$-module $\mathfrak C$ such that for all $\tilde \a_j\in \mathfrak C$ and all $k\geq 0$,
\begin{align*}
j_0^*\tilde\fm_k(\tilde \a_1,\ldots,\tilde \a_k)=&\fm_k(j_0^*\tilde \a_1,\ldots,j_0^*\tilde \a_k),
\\
j_1^*\tilde\fm_k(\tilde \a_1,\ldots,\tilde \a_k)=&\fm'_k(j_1^*\tilde \a_1,\ldots,j_0^*\tilde \a_k),
\end{align*}
and
\begin{align*}
j_0^* {\pprec\tilde \a_1,\tilde \a_2\ssucc}=\prec j_0^*\tilde\a_1,j_0^*\tilde\a_2\succ,\qquad j_0^*\tilde{ \mathbf e}=\mathbf e,\\
j_1^* {\pprec\tilde \a_1,\tilde \a_2\ssucc}=\prec j_1^*\tilde\a_1,j_1^*\tilde\a_2\succ',\qquad j_1^*\tilde{ \mathbf e}=\mathbf e'.
\end{align*}
\end{deff}

Let $\g,\g'\in \mathcal I_QD$ be closed with $|\g|=|\g'|=2$ and let $J,J'$ be two almost complex $\w$-tame structures on $X$. Let $\mathcal S,\mathcal S'$ be the cyclic unital $n-1$ dimensional $A_\infty$-algebra structure on $C$ from Theorem~\ref{A infinity algebra theorem}, for the pairs $(J,\g)$ and $(J',\g')$.

\begin{thm}\label{pseudoisotopy thm introduction}
If $[\g]=[\g']\in \hat H^*(X,L;Q)$, then there exists a cyclic unital pseudoisotopy from $\mathcal S$ to $\mathcal S'$.
\end{thm}

By Property (4), the maps $\fm_k$ descend to maps on the quotient
\[
\bar\fm_k:\overline{C}^{\otimes k}\to \overline{C}.
\]
\begin{thm}
\label{algebra deformation theorem}
Suppose $\pa_{t_0}\g=1\in A^0(X,L)\otimes Q$ and $\pa_{t_1}\g=\g_1\in A^2(X,L)\otimes Q$. Assume the map $H_2(X,L;\Z)\to Q$ given by $\b\mapsto \int_\b\g_1$ descends to $\Pi$. Then the operations $\fm_k^\g$ satisfy the following properties.
\begin{enumerate}
    \item \label{algebra deformation theorem: fundamental class}(Fundamental class) $\pa_{t_0}\om{3}=-1\cdot\d_{0,k}.$
    \item \label{algebra deformation theorem: divisor}(Divisor) $\pa_{t_1}\omb{3}=\int_\b\g_1\cdot\fm_k^{\g,\b}.$
    \item \label{algebra deformation theorem: energy zero}(Energy zero) The operations $\om{3}$ are deformations of the usual differential graded algebra structure on differential forms. That is,
    \[
    \bar\fm_1^\g(\a)=d\a,\qquad \bar\fm_2^\g(\a_1,\a_2)=(-1)^{|\a_1|}\a_1\we\a_2,\qquad \bar \fm_k^\g=0,\quad k\neq 1,2.
    \]
\end{enumerate}
\end{thm}
Following \cite{Fukaya-3-folds,Sara1}, in Section~\ref{operators subsection}, using the family $Q_{-1,l}^\b$, we construct a distinguished element $\fm_{-1}^\g\in R$. In the subsequent sections, we prove its properties along with the properties of $\fm_k^\g$ for $k\geq 0$.


\subsection{Outline}
In Sections~\ref{notations section}-\ref{orientations section} we review algebraic notations, orbifolds background and orientation conventions. Sections~\ref{orientors section} and~\ref{orientor calculus section} recall orientors and orientor calculus. Section~\ref{moduli spaces section} is devoted to the discussion of families of Lagrangian submanifolds in symplectic manifolds and related moduli spaces of stable maps. In particular, Section~\ref{Orientors over the moduli spaces section} reviews results in orientor calculus of these moduli spaces. Section~\ref{pushforward section} extends the notion of pushforward along a relatively oriented submersion to that of pushforward along orientors covering submersions. Section~\ref{currents section} recalls vertical currents along submersions of orbifolds with corners. Vertical currents are of importance in the proof of Proposition~\ref{divisors}. In Section~\ref{Structure section} we construct the operators $\om{3}$, and the Poincar\'e pairing $\langle,\rangle$ and prove the $A_\infty$ relations for them. Section~\ref{Properties section} states and proves properties of $\om{3}$ and $\langle,\rangle,$ and in particular, the properties in Definition~\ref{A infinity algebra definition}. Section~\ref{Conclusions section} concludes the paper with statements that generalize Theorems~\ref{A infinity algebra theorem},~\ref{pseudoisotopy thm introduction} and~\ref{algebra deformation theorem} to families of Lagrangian submanifolds, along with their proofs.

\subsection{Acknowledgements}
The authors are grateful to M.~Abouzaid, E.~Kosloff, P.~Seidel and S.~Tukachinsky, for helpful conversations. The authors were partially funded by ERC starting grant 337560 as well as ISF grants 569/18 and 1127/22.

\section{Conventions}
\label{Conventions section}
\subsection{Notations}
\label{notations section}
We follow the notations and conventions of \cite{orientors}. The notations and conventions follow. Proofs of all statements appear in \cite{orientors}.
In the following sections we work in the category of orbifolds with corners, indicated by the Latin capital letters $M,N,P,X,Y$, and smooth amps between them, indicated by $f,g, h$ etc. For a comprehensive guide for the category of orbifolds with corners, we recommend \cite{Sara-corners}.
Throughout this paper, we fix a commutative ring $\A$.

\begin{notation}[Abuse of notation in equations of natural numbers]
Let $M,N$ be manifolds and $f:M\to N$ be a smooth map.
Let $Q,S$ be graded local systems over $M$ and let $F:Q\to S$ be a morphism of degree $\deg F$ and let $q\in Q$ be of degree $\deg q$. Let $\a\in A(M;Q)$ be a differential form. Let $\b$ be a homology class of a symplectic manifold $X$ relative to a Lagrangian $L$.
\\
In integral expressions (mostly used as exponents of the number $-1$):
\begin{enumerate}
    \item As stated in the introduction, a local system of graded $\A$-modules will be referred to as a local system. A morphism of local systems might be referred to as a map.
    \item we write $m$ (or $M$) for the dimension of the corresponding capital-letter orbifold $M$;
    \item we write $f$ for  ${\text{rdim}\,f=\dim M-\dim N}$, the relative dimension of $f$;
    \item we write $q$ for $\deg q$ and we write $F$ for $\deg F$;
    \item we write $\a$ for $|\a|$ which is the degree of $\a$;
    \item we write $\b$ for the Maslov Index $\mu(\b)$.
\end{enumerate}
\end{notation}
\subsection{Graded algebra}
Throughout the paper we write $x=_2y$ to denote $x\equiv y\mod 2.$
\begin{definition}[Tensor product]
\label{tensor product}
Let $\A$ be a ring.
Let $A,B,C,D$ be graded $\A$-modules with valuations (or local systems of graded $\A$-modules over an orbifold with corners). Let $F:A\to C, G:B\to D$ be linear maps of degrees $|F|,|G|$. Let $a,b$ be homogeneous elements in $A,B$, respectively.
\begin{enumerate}
\item The sign $\otimes$ means the completed tensor product with respect to the valuations.
\item The tensor product of differential graded algebras with valuations is again a differential graded algebra with valuation in the standard way.
For \[a\in A,b\in B\] the differential is
\[
d_{A\otimes B}(a_0\otimes b_0)=(d_A a_0)\otimes b_0+(-1)^{a_0}a_0\otimes d_Bb_0,
\]
and for \[\{a_i\}_{i=1}^\infty\in A,\{b_j\}_{j=1}^\infty\in B,\]
  the valuation is defined as follows\begin{align*}
\\\nu_{A\otimes B}\left(\sum_{i,j}a_i\otimes b_j\right)=\inf_{a_i\otimes b_j\neq 0}\left(\nu(a_i)+\nu(b_j)\right).
\end{align*}

\item The symmetry isomorphism $\tau_{A,B}$ is given by
\[
A\otimes B\overset{\t_{A,B}}\to B\otimes A,\quad a\otimes b\mapsto (-1)^{ab}b\otimes a.
\]
\item \textit{Tensor product of $\A$-algebras:}\\If $A,B$ are graded $\A$-algebras (or local systems of graded $\A$-algebras) with multiplication $(\cdot_A,\cdot_B)$ then the graded $\A$-algebra $A\otimes B:=A\otimes_\A B$ is defined as the graded tensor product, with multiplication
    \[
    (a_1\otimes b_1)\cdot (a_2\otimes b_2)=(-1)^{b_1a_2}(a_1\cdot_A a_2)\otimes (b_1\cdot_B b_2).
    \]
    \item \textit{Functoriality of tensor product:}\\
    The tensor product of two maps is given by
    \[
    F\otimes G:A\otimes B\to C\otimes D,\qquad F\otimes G(a\otimes b)=(-1)^{|G|a}Fa\otimes Gb.
    \]

\end{enumerate}
\end{definition}

\begin{lemma}
\label{symmetry isomorphism distributivity lemma}
Let $A,B,C$ be graded $\A$-modules.
Then as maps $A\otimes B\otimes C\to B\otimes C\otimes A$ there is the equation
\[
\tau_{A,B\otimes C}=(\Id_B\otimes\tau_{A,C})\circ (\tau_{A,B}\otimes \Id).
\]
\end{lemma}
\begin{proposition}[Koszul signs]\label{Koszul signs}
With the previous notation, if $F':C\to C',G':D\to D'$ are maps leaving $C,D$ respectively, of degrees $|F'|,|G'|$, then
\begin{align}
(G\otimes F)\circ\tau_{A,B}=(-1)^{|F||G|}\tau_{C,D}\circ(F\otimes G).\\
(F'\otimes G')\circ(F\otimes G)=(-1)^{|F||G'|}(F'\circ F)\otimes (G'\circ G).
\end{align}
\end{proposition}
\begin{definition}
\label{dual graded vector space definition}
Let $A=\bigoplus_{i\in \Z}A_i$ be a graded $\A$-module. \textbf{The dual space $A^\vee$ of $A$} is given by
\[
A^\vee :=\bigoplus_{i\in \Z}A_i^\vee,
\]where $A_i^\vee$ is the space of linear maps from $A_i$ to $\A$. Denote by $\nu_A:A\otimes A^\vee\to \A$ the pairing $a\otimes a^\vee\to a^\vee(a)$.
\end{definition}
\begin{definition}
Let $T,K$ be graded $\A$-modules and let $S$ be a graded $\A$-algebra. Assume that $\mu:S\otimes K\to K$ is a module-structure. We define \textbf{the left (resp. right) $T$-extension of $\mu$} to be a module-structure of $T\otimes K$ (resp. $K\otimes T$) as follows.
\[
{}^T\mu(s\otimes t\otimes k)=(-1)^{st}t\otimes \mu(s\otimes k),
\]
\[
\mu^T(s\otimes k\otimes t)=\mu(s\otimes k)\otimes t.
\]
We further define an $S$-module structure on $K^\vee$ by
\[
\left(\mu^\vee(s\otimes v^\vee)\right)\otimes v=(-1)^{sv^\vee}v^\vee\left(\mu(s\otimes v)\right).
\]
\end{definition}
\begin{definition}
\label{dual graded linear map definition}
Let $F:X\to Y$ be a graded linear map. \textbf{The dual map $F^\vee$ of $F$} is the graded linear map
\begin{align*}
    F^\vee:Y^\vee&\to X^\vee\\
    \left(F^\vee y^\vee\right)(x)&=(-1)^{|F||x|}y^\vee(Fx).
\end{align*}
\end{definition}
\begin{remark}
\label{dual map diagram definition remark}
Let $F:X\to Y$ be a graded linear map. Then the following diagram is commutative.
\[
\begin{tikzcd}
X\otimes Y^\vee\ar[r,"F\otimes \Id"]\ar[d,"\Id\otimes F^\vee"]&Y\otimes Y^\vee\ar[d,"\nu_Y"]\\X\otimes X^\vee\ar[r,"\nu_X"]&\A
\end{tikzcd}
\]
\end{remark}

For a set $A$, denote the constant map by $\pi^A:A\to *$.
For two sets $A,B$, we denote their product and corresponding projections as follows.
\[
\begin{tikzcd}
A\times B\ar[r,"\pi^{A\times B}_B"]\ar[d,swap,"\pi^{A\times B}_A"]&B\ar[d,"\pi^B"]\\A\ar[r,swap,"\pi^A"]&*
\end{tikzcd}
\]
When it causes no confusion, we might write $\pi_A,\pi_B$ for the projections.

For two lists $B_1=(v_1,\ldots,v_n),B_2=(w_1,\ldots,w_m)$, denote by $B_1\circ B_2$ the concatenation $(v_1,\ldots,v_n,w_1,\ldots,w_m).$

\subsection{Orbifolds with corners}
\label{orbifolds}
We use the definition of orbifolds with corners from~\cite{Sara-corners,neat-embeddings-of-orbifolds-with-corners}. We also use the definitions of smooth maps, strongly smooth maps, boundary and fiber products of orbifolds with corners given there. In particular, for an orbifold with corners $M$, the boundary $\pa M$ is again an orbifold with corners, and it comes with a natural map $\iota_M:\pa M\to M$. In the special case of manifolds with corners, our definition of boundary coincides with~\cite[Definition 2.6]{Joyce}, our smooth maps coincide with weakly smooth maps in~\cite[Definition 2.1(a)]{Joyce-generalization}, and our strongly smooth maps are as in~\cite[Definition 2.1(e)]{Joyce-generalization}, which coincides with smooth maps in~\cite[Definition 3.1]{Joyce}. We say a map of orbifolds is a submersion if it is a strongly smooth submersion in the sense of~\cite{Sara-corners}. In the special case of manifolds with corners, our submersions coincide with submersions in~\cite[Definition 3.2(iv)]{Joyce} and with strongly smooth horizontal submersions in~\cite[Definition 19(a)]{amitai-moduli-homogeneous}. We use the definition of neat immersions and embeddings from~\cite{neat-embeddings-of-orbifolds-with-corners}. In the case of manifolds with corners, the definitions agree with~\cite{Hajek-neat-immersions}. For a strongly smooth map of orbifolds $f:M\to N$, we use the notion of vertical corners $C_f^r(M)\subset C_r(M)$ as explained in~\cite{neat-embeddings-of-orbifolds-with-corners}. In the special case $r=1$, the vertical boundary $\pa_f M\subset\pa M$ is defined in~\cite[Section 2.1.1]{Sara-corners}, which extends the definition of~\cite[Section 4]{Joyce} to orbifolds with corners. We often write $\pa^vM$ for $\pa_fM$ when $f$ is clear from the context, where $v$ stands for `vertical'. We write $\iota_f:\pa_fM\to M$ for the restriction of $\iota$ to $\pa_fM$. When $f$ is a submersion, the vertical boundary is the fiberwise boundary, that is, $\pa_fM=\coprod_{y\in N}\pa(f^{-1}(y))$. If $\pa N=\emp$, then $\pa_fM=\pa M$. A strongly smooth map of orbifolds $f:M\to N$ induces a strongly smooth map $f|_{\pa_fM}=f\circ \iota_f:\pa_fM\to N$, called the restriction to the vertical boundary. If $f$ is a submersion, then the restriction $f|_{\pa_fM}$ is also a submersion. As usual, diffeomorphisms are smooth maps with a smooth inverse. We use the notion of transversality from~\cite[Section 3]{Sara-corners}, which is induced from transversality of maps of manifolds with corners as defined in~\cite[Definition 6.1]{Joyce}. In particular, any smooth map is transverse to a submersion. Weak fiber products of strongly smooth transverse maps exist by~\cite[Lemma 5.3]{Sara-corners}. Below, we omit the adjective `weak' for brevity. 
For the theory of differential forms on orbifolds, we refer to~\cite{Sara-corners}. We use the definition of vertical currents along a submersion of orbifolds from~\cite{neat-embeddings-of-orbifolds-with-corners}. 

\begin{definition}\label{boundary factorization}
Let $M\overset{f}\to P\overset{g}\to N$ be such that $g\circ f$ is a proper submersion. In particular, $g$ is a proper submersion. we say \textbf{$f$ factorizes through the boundary of~$g$} if there exists a map $\iota_g^*f$ such that the following diagram is a fiber-product.
\[
\begin{tikzcd}
\pa_{g\circ f}M\ar[r,"\iota_{g\circ f}"]\ar[d,swap,dashed, "\iota_g^* f"]&M\ar[d,"f"]\\
\pa_g P\ar[r,swap,"\iota_g"]&P
\end{tikzcd}
\]
\end{definition}
\begin{remark}
If $f$ is a proper submersion, and $f$ factorizes through the boundary of $g$, then $\iota_g^*f$ is also a proper submersion.
\end{remark}
\begin{notation}
generally, for a set $X$ and a topological space $M$, we write $\underline X$ for the trivial local system over $M$ with fiber $X$.
\end{notation}

\subsection{Orientation conventions}\label{orientations section}
We follow the conventions of \cite{Sara-corners} concerning manifolds with corners. In particular, we relatively orient boundary and fiber products as detailed in the following. For an orbifold with corners $M$, we consider the orientation double cover $\tilde M$ as a graded $\Z/2$-bundle, concentrated in degree $\deg \tilde M=\dim M$.
\begin{definition}
 Let $M\overset f\to N$ be a map. We define the \textbf{relative orientation bundle} of $f$ to be the $\Z/2$-bundle over $M$ given by
\[
\zcort{f}:=  Hom_{\Z/2}(\tilde M,f^*\tilde N).
\]

A \textbf{local relative orientation} is a section $\mO:U\to \zcort{f}|_U$ over an open subset $U\subset M$.
A \textbf{relative orientation} is a global section $\mO:M\to\zcort{f}$.
\end{definition}
Note that it is concentrated in degree $-\text{rdim}\, f=-m+n$.

\begin{definition}
 The \textbf{orientation bundle} of an orbifold $M$ is defined to be the relative orientation bundle of the constant map $M\to pt$,
\[
\zcort{M}:=Hom_{\Z/2}(\tilde M,\underline{\Z/2})=\tilde M^\vee,
\]
A \textbf{(local) orientation} for $M$ is (local) orientation relative to the constant map $M\to pt$.
\end{definition}
Note that it is concentrated in degree $-\dim M$.

We now relatively-orient chosen operations on orbifolds.
\subsubsection{Local diffeomorphism}
    \label{local diff orientation}
    \begin{definition}
    Let $f:M\to N$ be a local diffeomorphism. The differential $df$ is regarded as a bundle map $df:TM\to f^*TN$. Its exterior power induces a $\Z/2$-bundle map $[\La^{top}df]:\tilde M\to f^*\tilde N$. It can be thought of as a section $\mO^f_c\in Hom(\tilde M,f^*\tilde N)$ called \textbf{the canonical relative orientation of} $f$. In particular, $\zcort{f}$ is canonically trivial.

    Moreover, given a map $g:N\to P$, there is a \textbf{pullback map}
    \[
        \pull{f}:f^*\zcort{g}\to \zcort{g\circ f}
    \]
    given by composition on the right with $\mO^f_c$.
\end{definition}

\subsubsection{Composition}
\label{Composition orientation convention}
\begin{definition}
\label{composition isomorphism}
Let $M\overset{f}\longrightarrow P\overset{g}\longrightarrow N$ be two maps. There is a canonical isomorphism
    \[
    \zcort{f}\otimes f^*\zcort{g}\simeq \zcort{g\circ f},
    \] called the \textbf{composition isomorphism,}
    given by
    \begin{align*}
    Hom_{\Z/2}(\tilde M,f^*\tilde P)\otimes f^*Hom_{\Z/2}(\tilde P, g^*\tilde N)&\to Hom_{\Z/2}(\tilde M,(g\circ f)^*\tilde N),\\ \mO^f\otimes f^*\mO^g&\mapsto f^*\mO^g\circ\mO^f.
    \end{align*}
\end{definition}

\begin{notation}
By abuse of notation, we may omit the pullback notation $f^*$ if it causes no confusion, such as
\[
\mO^f\otimes \mO^g=\mO^f\otimes f^*\mO^g,\qquad \mO^g\circ\mO^f=f^*\mO^g\circ \mO^f.
\]
Moreover, we may notate this isomorphism as equality. This is justified by Fubini's theorem of Proposition \ref{properties of pushforward}.
\end{notation}

\subsubsection{Relative orientation of boundary}
\label{relative orientation of boundary}
    Let $M\overset{f}\longrightarrow N$ be a proper submersion. As explained in Section \ref{orbifolds}, the boundary of $M$ can be divided into horizontal and vertical components with respect to $f$. Let $p\in \pa_fM$ be a point in the vertical boundary and $x_1,...,x_{m-1}\in T_p\pa_fM$ be a basis, such that
    \[
    df_{\iota_f(p)}\circ {(d\iota_f)}_p(x_{i})=0,  \quad i=n+1,...,m-1.
    \]Let $x_1^\vee,...,x_{m-1}^\vee$ be the dual basis. Let $\nu_{out}$ be an outwards-pointing vector in $T_pM$.
    We define the \textbf{canonical relative orientation of the boundary} to be
    \begin{align}
    \label{Canonical relative orientation of boundary}\mO_c^{\iota_f}|_p:=\Big[x_1^\vee\we...&\we x_{m-1}^\vee\bigotimes {(d\iota_f)}_p(x_1)\we...\we {(d\iota_f)}_p(x_n)\we\\ \notag&\we\nu_{out}\we {(d\iota_f)}_p(x_{n+1})\we...\we {(d\iota_f)}_p(x_{m-1})\Big].
\end{align}
\subsubsection{Fiber product}
\begin{definition}\label{pullback fiber product orientation convention}
Let $M\overset{f}\rightarrow N\overset{g}\leftarrow P$ be transversal smooth maps of orbifolds with corners. Consider the following fiber product diagram.
\begin{equation}
\label{generic pullback diagram}
\begin{tikzcd}
M\times_N P\ar[r,"r"]\ar[d,"q"]&P\ar[d,"g"]\\M\ar[r,"f"]&N
\end{tikzcd}
\end{equation}

There is a canonical isomorphism from the relative orientation bundle of $q$ to the pullback of the relative orientation bundle of $g$. It is called \textbf{the pullback by r over f} and denoted
\[
\pull{(r/f)}:r^*\zcort{g}\simeq \zcort{q}.
\]
It is given as follows.
Let $(m,p)\in M\times P$ be such that $f(m)=g(p)$.
Let $\mO^N,\mO^M,\mO^g$ be local orientations of $N,M,g$ in neighborhoods of $f(m),m,p$, respectively.
Define $\mO^P:=\mO^N\circ \mO^g$.
By the transversality assumption,
\[
F:=df_m\oplus -dg_p:T_mM\oplus T_pP\to T_{f(m)}N
\]
is surjective, and by definition of fiber product, there is a canonical isomorphism
\[
\psi:=dq_{(m,p)}\oplus dr_{(m,p)}:T_{(m,p)}(M\times_N P)\to \ker(F).
\]
Therefore, there exists a short exact sequence
\[
\begin{tikzcd}
0\ar[r]&T_{(m,p)}(M\times_NP)\ar[r,"\psi"]&T_mM\oplus T_pP\ar[r,"F"]&T_{f(m)}N\ar[r]&0.
\end{tikzcd}
\]
Splitting the short exact sequence, we get an isomorphism
\[
T_mM\oplus T_pP\xrightarrow{\Psi} T_{(m,p)}(M\times_NP)\oplus T_{f(m)}N.
\]
We define a local orientation $\mO^{M}\times \mO^g$ of $M\times_NP$ at $(m,p)$ to be the orientation for which $\Psi$ has sign $(-1)^{NP}$, and subsequently we define a local orientation $\pull{(r/f)}(\mO^g)$ of $q$ to satisfy the following equation.
\[
\mO^{M}\times\mO^g=\mO^M\circ \pull{(r/f)}(\mO^g).
\]
\end{definition}

\subsection{Orientors}
\label{orientors section}
In this paper, we will concentrate mostly on bundle-maps of the following form.
\begin{definition}\label{orientor}
Let $g:M\to N$ be a map and let $Q,K$ be $\Z/2$-bundles over $M,N$, respectively. A \textbf{\orientor{g} of $Q$ to $K$} is a graded bundle map \[
G:Q\to \zcort{g}\otimes_{\Z/2} g^*K.
\] Its \textbf{degree} is the usual degree as a bundle map, where $\zcort{g}\otimes g^*K$ is, as usual, the graded tensor product and $\zcort{g}$ is concentrated in degree $-\text{reldim } g$.
A \eorientor{g} of $K$ is a \orientor{g} of $g^*K$ to $K$.
\end{definition}
\begin{terminology}
if $g=\pi^M:M\to *$ is the constant map, then we say \orientor{M} for \orientor{g}.
\end{terminology}
\begin{definition}[Orientation as an orientor]
\label{orientation as orientor}
Let $f:M\to N$ be a relatively orientable map of orbifolds with corners. The section $\mO^f:M\to \zcort{f}$ can be extended uniquely to a $\Z/2$ equivariant map
\[
\phi^{\mO^f}:\underline{\Z/2}\to \zcort{f}
\]
which satisfies
\[\phi^{\mO^f}(1)=\mO^f.\]
The map $\phi^{\mO^f}$ can be considered as an \eorientor{f} of $\underline{\Z/2}$. If $f$ is a local diffeomorphism, we denote by
\[
\phi_f:=\phi^{\mO^f_c}.
\]
If $N=*$ and $M$ is oriented with orientation $\mO^M$, then we abbreviate
\[
\phi_M:=\phi^{\mO^M}.
\]
\end{definition}
\begin{example}\label{symmetry isomorphism as orientor example}
Let $A,B$ be $\Z/2$ vector bundles over an orbifold $M$. The symmetry operator $\tau_{A,B}:A\otimes B\to B\otimes A$ of Definition \ref{tensor product} may be considered as an \orientor{\Id_M}. More generally, any bundle map of bundles over an orbifold $M$ may be considered as an \orientor{\Id_M}.
\end{example}
\begin{definition}\label{tensored orientor extension}
Let $M,N,g,Q,K,G$ be as in Definition \ref{orientor} and let $T$ be a $\Z/2$ bundle over $N$. Then the \textbf{right $T$ extension of $G$} is the \orientor{g} of $Q\otimes g^*T$ to $K\otimes T$ given by
\[
Q\otimes g^*T\overset{G\otimes\Id}{\xrightarrow{\hspace*{1cm}}}\zcort{g}\otimes g^*(K\otimes T).
\]
It is denoted by $G^T.$
Similarly, the \textbf{left $T $ extension of $G$} is the \orientor{g} of $g^*T\otimes Q$ to $T\otimes K$ given by
\[
g^*T\otimes Q\xrightarrow{\Id \otimes G}g^*T\otimes \zcort{g}\otimes g^*K\xrightarrow{\tau\otimes\Id}\zcort{g}\otimes g^*(T\otimes K).
\]
It is denoted by ${}^TG.$
\end{definition}
\begin{definition}\label{boundary-operator for relative orientation}
The \textbf{boundary orientor} is the \eorientor{\iota_f} of $\underline{\Z/2}$
\[
\pa_f:\underline{\Z/2}\to\zcort{\iota_f}\otimes \underline{\Z/2}
\]
given by
\[
\pa_f(1)=(-1)^f\mO_c^{\iota_f}.
\]
\end{definition}
\begin{remark}\label{boundary orientor is contraction by minus out pointing vector}
The composition of ${\pa_f}^{\zcort{f}}$ and the composition isomorphism,
\[
{\iota_f}^*\zcort{f}\overset{{\pa_f}^{\zcort{f}}}\to\zcort{\iota_f}\otimes{\iota_f}^*\zcort{f} \overset{comp.}=\zcort{f\circ\iota_f},
\]
is given by
\[
\mO^f\mapsto(-1)^{f}\mO^f\circ\mO_c^{\iota_f}.
\]
By abuse of notation, we often denote this composition by ${\pa_f}^{\zcort{f}}.$ Explicitly, it is given by the contraction with $-\nu_{out}$ on the right.
\end{remark}
\begin{definition}\label{orientor composition}
Let $M\overset{f}\to P\overset{g}\to N$ be maps and let $Q,K,R$ be {$\Z/2$~bundles} over $M,P,N$, respectively. Let $F:Q\to \zcort{f}\otimes f^*K$ be a \orientor{f} of $Q$ to $K$ and $G:K\to \zcort{g}\otimes g^*R$ be a \orientor{g} of $K$ to $R$. \textbf{The composition $G\bu F$} is the \orientor{g\circ f} of $Q$ to $R$ given as follows.
\[
Q\overset{F}\to\zcort{f}\otimes f^*K\overset{\Id \otimes f^*G}\to \zcort{f}\otimes f^*\zcort{g}\otimes f^*g^*P\overset{comp.}=\zcort{g\circ f}\otimes {(g\circ f)}^*R
\]
\end{definition}
\begin{definition}\label{pullback of orientor}
Let $M\overset{f}\rightarrow P\overset{g}\rightarrow N$, and suppose that $f$ is relatively oriented with relative orientation $\mO^f$. Let $K,R$ be $\Z/2$-bundles over $P,N$, respectively. Let $G$ be a \orientor{g} of $K$ to $R$. The \textbf{pullback of $G$ by $(f,\mO^f)$} is the \orientor{g\circ f} of $f^*K$ to $R$, given by
\[
\expinv{(f,\mO^f)}G=(-1)^{fG}G\bu \left(\phi^{\mO^f}\right)^{K},
\]
where $\phi^{\mO^f}$ is the orientor from Definition \ref{orientation as orientor}.
If $f$ is a local diffeomorphism, we write
\[
\expinv{f}G=\expinv{(f,\mO^f_c)}G=G\bu \left(\phi_f\right)^K.
\]
\end{definition}
\begin{definition}\label{pullback of orientor by pullback-diagram}
Let
\[
\begin{tikzcd}
M\times_NP\ar[r,"r"]\ar[d,"q"]&P\ar[d,"g"]\\M\ar[r,"f"]&N
\end{tikzcd}
\] be a pullback square. Let $K,R$ be $\Z/2$-bundles over $P,N$ respectively, and let $G$ be a \orientor{g} from $K$ to $R$. \textbf{The pullback of $G$ by $r$ over $f$} is the \orientor{q} of $r^*K$ to $f^*R$ given by the following composition.
\[
r^*K\xrightarrow{r^*G}r^*\zcort{g}\otimes r^*g^*R\xrightarrow{\pull{(r/f)} \otimes\Id}\zcort{q}\otimes q^*f^*R.
\]
It is denoted by $\expinv{(r/f)}G$.
\end{definition}
\begin{example}\label{trivial expinv is pullback}
Let $f:M\to N$ be a map, $Q,K$ be $\Z/2$-bundles over $N$ and let $G:Q\to K$ be a \orientor{\Id_N} of $Q$ to $K$. Consider the following pullback diagram.
\[
\begin{tikzcd}
M\ar[d,swap,"\Id_M"]\ar[r,"f"]&N\ar[d,"\Id_N"]\\M\ar[r,"f"]&N
\end{tikzcd}
\]
Then under the canonical isomorphism $\zcort{\Id_X}\simeq \underline{\Z/2}$ it holds that
\[
\expinv{(f/f)}G=f^*G.
\]
\end{example}
\begin{definition}\label{restriction of orientor to the boundary}
Let $M\overset{f}\to P\overset{g}\to N$ be such that $g\circ f$ is a surjective submersion. Assume that $f$ factorizes through the boundary of $g$ in the sense of Definition~\ref{boundary factorization}. That is, we have the following pullback diagram.
\[
\begin{tikzcd}
\pa_{g\circ f}M\ar[r,"\iota_{g\circ f}"]\ar[d,swap,dashed, "{\iota_g}^* f"]&M\ar[d,"f"]\\
\pa_g P\ar[r,swap,"\iota_g"]&P
\end{tikzcd}
\]
Let $Q,K$ be $\Z/2$-bundles over $M,P$ respectively, and let $F$ be a \orientor{f} from $Q$ to $K$. The \textbf{restriction of $F$ to the boundary} is $\expinv{\left(\iota_{g\circ f}/\iota_g\right)}F$, that is, the \orientor{{\iota_g}^*f} of $\iota_{g\circ f}^*Q$ to $\iota_g^*K$
\[
\iota_{g\circ f}^*Q\overset{\iota_{g\circ f}^*F}{\xrightarrow{\hspace*{1cm}}} \iota_{g\circ f}^*\zcort{f}\otimes \iota_{g\circ f}^*f^*K\overset{\pull{\left(\iota_{g\circ f}/\iota_g\right)}\otimes\Id}{\xrightarrow{\hspace*{2cm}}}\zcort{{\iota_g}^*f}\otimes {({\iota_g}^*f)}^*\iota_g^*K.
\]
\end{definition}
\begin{definition}\label{Boundary of orientor}
Let $M,N,g,Q,K,G$ be as in Definition \ref{orientor}. Recall that the \eorientor{\iota_g}
${\pa_g}^Q:\iota_{g}^*Q\to \zcort{\iota_g}\otimes \iota_g^*Q$ is the $Q$-extension from Definition \ref{tensored orientor extension} of the boundary orientor from Definition \ref{boundary-operator for relative orientation}. The \textbf{boundary of $G$} is the \orientor{\zcort{g\circ\iota_g}} of $\iota_g^*Q$ to $K$,
\[
\pa G=(-1)^{|G|}G\bu {\pa_g}^Q.
\]
\end{definition}
\begin{remark}
Recall that the degree of the boundary operator is $|\pa_g|=1$. Thus, if the degree of $G$ is $|G|$, then the degree of $\pa G$ is $|G|+1$.
\end{remark}

\subsection{Orientor calculus}\label{orientor calculus section}
\begin{lemma}\label{composition of orientors is associative}
The composition of orientors is associative.
\end{lemma}
\begin{lemma}\label{orientors tensor distributivity}
With the setting of Definition \ref{orientor composition}, let $T$ be a $\Z/2$-bundle over $N$. Then
\begin{align*}
\left(G\bu F\right)^T&=G^{T}\bu F^{g^*T},\\
{}^T\left(G\bu F\right)&={}^TG\bu {}^{g^*T}F.
\end{align*}
\end{lemma}
\begin{example}\label{expinv by diffeomorphism over Id}
Let $M\xrightarrow{f}P\xrightarrow{g}N$ be maps, $K,R$ be $\Z/2$-bundles over $P,N$, respectively, and let $G$ be a \orientor{g} of $K$ to $R$. Assume $f$ is a diffeomorphism. Consider the following pullback diagram.
\[
\begin{tikzcd}
M\ar[r,"f"]\ar[d,"g\circ f"]&P\ar[d,"g"]\\N\ar[r,"\Id"]&N
\end{tikzcd}
\]We have
\[
\expinv{(f/\Id)}G=\expinv{f}G.
\]
\end{example}

\subsection{Extension to arbitrary commutative rings}

\begin{definition}
Let $\A$ be a commutative ring.
Let $f:M\to N$ be a map. Consider the $\A$-representation of $\Z/2$ given by negation, $(-1)\cdot a=-a$ for $a\in\A$. Then the \textbf{$\A-$relative orientation bundle} $\cort{f}\to M$ is the local system associated to the negation representation,
\[
\cort{f}=\zcort{f}\times_{\Z/2} \A.
\]
\end{definition}
\begin{remark}
As $\zcort{f}$ is concentrated in degree $-\dim f$, so is $\cort{f}$.
\end{remark}
If $Q,K$ are local systems over $M,N$ and $g:M\to N$ is a map, then a \textbf{\orientor{g} of $Q$ to $K$} is a morphism of local systems over $M$
\[
G:Q\to \cort{g}\otimes_\A g^*K.
\]
All definitions, equations and lemmas about $\Z/2$-orientors extend naturally to orientors of local systems over $\A$.

\section{Moduli Spaces}
\label{moduli spaces section}
In this section we recall the setting and main results of \cite{orientors} regarding moduli spaces of stable curves and their associated orientors. Proofs to the Lemmas and theorems appear there.
\subsection{Open stable maps}
\label{Moduli spaces}
Let $(X_0,\w_0)$ be a symplectic manifold of dimension $2n$ and let $L_0\subset X_0$ be a Lagrangian. Let $\mu_0:H_2(X_0,L_0;\Z)\to \Z$ be the Maslov index \cite{Maslov}. The symplectic form $\w_0$ induces a map $\w_0:H_2(X_0,L_0;\Z)\to \R$ given by integration, $\b\mapsto \int_\b w_0$. Let $\Pi_0$ be a quotient of $H_2(X_0,L_0;\Z)$ by a subgroup that is contained in the kernel of $(\mu_0,\w_0):H_2(X_0,L_0;\Z)\to \Z\oplus \R$. Thus, $\mu_0,\w_0$ descend to $\Pi_0$. Let $J_0$ be an $\w_0-$tame almost target structure on $X_0$.
A $J_0$-holomorphic genus-0 open stable map to $(X_0,L_0)$ of degree $\b\in \Pi_0$ with one boundary component, $k+1$ boundary marked points, and $l$ interior marked points, is a quadruple $\mathfrak u:=(\S,u,\vec z,\vec w$) as follows. The domain $\S$ is a genus-0 nodal Riemann surface with boundary consisting of one connected component. 
The map of pairs
\[
u:(\S,\pa \S)\to (X_0,L_0)
\]
is continuous, and $J_0$-holomorphic on each irreducible component of $\S$, satisfying
\[
u_*([\S,\pa \S])=\b.
\]
The boundary marked points and the interior marked points
\[
\vec z=(z_0,...,z_k),\qquad \vec w=(w_1,...,w_l),
\]
where $z_j\in \pa \S,w_j\in \overset{\circ}\S$, are distinct from one another and from the nodal points. The labeling of the marked points $z_j$ respects the cyclic order given by the orientation of $\pa \S$ induced by the complex orientation of $\S$. Stability means that if $\S_i$ is an irreducible component of $\S$, then either $u|_{\S_i}$ is non-constant, or it satisfies the following requirement: If $\S_i$ is a sphere, the number of marked points and nodal points on $\S_i$ is at least 3; if $\S_i$ is a disk, the number of marked and nodal boundary points plus twice the number of marked and nodal interior points is at least 3. An \textbf{isomorphism of open stable maps} \[\phi:(\S,u,\vec z,\vec w)\to(\S',u',\vec z',\vec w')\] is a homeomorphism $\phi:\S\to \S'$, biholomorphic on each irreducible component, such that 
\[
u=u'\circ \phi,\qquad\qquad z_j'=\phi(z_j),\quad j=0,...,k,\qquad\qquad w_j'=\phi(w_j),\quad j=1,...,l.
\]
We denote $\mathfrak u\sim \mathfrak u'$ if there exists an isomorphism of open stable maps $\phi:\mathfrak u\to \mathfrak u'$.
Denote by $\mM_{k+1,l}(X_0,L_0,J_0;\b)$ the moduli space of $J_0$-holomorphic genus-$0$ open stable maps to $(X_t,L_t)$ of degree $\b$ with one boundary component, $k+1$ marked boundary points and $l$ marked interior points.

\subsection{Families}\label{families target definition section}
Let $\W$ be a manifold with corners. An orbifold with corners $M$ over $\W$ is a submersion $\pi^M:M\to \W$. Let $\pi^N:N\to \W$ be another orbifold with corners over $\W$ and $f:M\to N$ be a map over $\W$. Let $\xi:\W'\to \W$ be any map. As $\pi^M$ is a submersion, the fiber product $\xi^*M:=\W'_\xi\times_{\pi^M}M$ exists. We also get an induced map $\xi^*f:\xi^*M\to \xi^*N$. The situation is summed up in the following diagram.
\begin{equation}
\label{pullback along families basic diagram}    
\begin{tikzcd}
\xi^*M\ar[r,"\xi^M"]\ar[d,swap,"\xi^*f"]&M\ar[d,"f"]\\\xi^*N\ar[r,"\xi^N"]\ar[d,swap,"\xi^*\pi^{N}"]&N\ar[d,"\pi^N"]\\\W'\ar[r,"\xi"]&\W
\end{tikzcd}
\end{equation}
Moreover, for a fiber-product
\[
\begin{tikzcd}
M\times_\W P\ar[r]\ar[d]&P\ar[d,"\pi^P"]\\M\ar[r,"\pi^M"]&\W
\end{tikzcd}
\]
of orbifolds with corners over $\W$, we write $\pi^{M\times P}_M,\pi^{M\times P}_P$ for the corresponding projections. Let $T^vM \to M$ be the vertical tangent bundle along the fibration $\pi^M:M\to \W$.
For many purposes, one may assume $\W$ is a point.
\begin{definition}
Let $f:M\to N$ be a map of smooth manifolds. A \textbf{vector field along $f$} is a section $u$ of the bundle $f^*TN\to M$. A vector field $u$ along $f$ determines a linear map
\begin{align*}
    i_u:A^{k}(N)\to& A^{k-1}(M)\\
    i_u\rho\left(v_1,...,v_{k-1}\right)|_{x\in M}=&\rho_{f(x)}\left(u(x),df_x(v_1(x)),\ldots,df_x(v_{k-1}(x))\right).
\end{align*}
called interior multiplication.
\end{definition}
\begin{definition}\label{horizontal form definition}
Let $\pi^M:M\to \W$ be a manifold over $\W$. A differential form $\xi\in A^*(M)$ is called \textbf{horizontal} with respect to $\pi^M$ if its restriction to vertical vector fields vanishes.
\end{definition}
\begin{definition}\label{exact submersion definition}
Let $\pi^M:M\to \W$ be a manifold over $\W$ and let $\w\in A^2(M)$. The submersion $\pi^M:M\to \W$ is called \textbf{exact} with respect to $\w$ if $\w$ is horizontal with respect to $\pi^M$ and for every vector field $u$ on $\W$ there exists a function $f_u:M\to \R$ such that for all vector fields $\tilde u$ on $M$ with $d\pi^M(\tilde u)=u$, the $1$-form 
\[
i_{\tilde u}\w-df_u
\]
is horizontal with respect to $\pi^M$.
\end{definition}
\begin{remark}\label{exact submersion independent of lift remark}
When checking whether a submersion is exact with respect to a horizontal $2$-form, given a vector field $u$ on $\W$, it suffices to construct a lift $\tilde u$ of $u$ to $M$, and a function $f_{u}:M\to \R$ such that $i_{\tilde u}\w-df_{u}$ is horizontal. It follows that for any lift $\tilde u'$, the form $i_{\tilde u'}\w - df_{u}$ is horizontal. Indeed,
\[
i_{\tilde u'}\w-i_{\tilde u}\w=i_{(\tilde u'-\tilde u)}\w
\]
is horizontal.
\end{remark}
\begin{lemma}\label{pullback of exact manifold is exact naturality lemma}
Recall the notation of diagram~\eqref{pullback along families basic diagram}. Let $\w\in A^2(M)$ and assume $\pi^M$ is exact with respect to $\w$. It holds that $\xi^*\pi^{M}$ is exact with respect to $\left({\xi^M}\right)^*\w\in A^2(\xi^*M)$.
\end{lemma}

\begin{definition}\label{symplectic fibration definition}
Let $\pi^X:X\to \W$ be a manifold with corners over $\W$, and let $\w$ be a closed $2$-form on $X$. $\pi^X$ is called a \textbf{symplectic fibration} if it is a locally trivial fibration such that, for all $t\in \W$, $(\pi^{-1}(t),\w|_{\pi^{-1}(t)})$ is a symplectic manifold and the vertical boundary with respect to $\pi^X$ is empty. Let $L\subset X$ be a subfibration, that is, the restriction $\pi^L:=\pi^X|_L$ is a locally trivial fibration. We say that $L$ is a \textbf{Lagrangian subfibration} if $\w|_L$ is horizontal with respect to $\pi^L$. That is, the fibers of $\pi^L$ are Lagrangian submanifolds in the fibers of $\pi^X$. A Lagrangian subfibration is called \textbf{exact} if $\pi^L:=\pi^X|_L$ is exact with respect to $\w|_L$.

For a vector bundle $V\to B$, define the characteristic classes $p^\pm(V)\in H^2(B;\Z/2)$ by
\[
p^+(V)=w_2(V),\qquad p^-(V)=w_2(V)+w_1(V)^2.
\]
According to~\cite{Kirby-Pin-structures}, $p^\pm(V)$ is the obstruction to the existence of a $Pin^{\pm}$ structure on $V$. See~\cite{Kirby-Pin-structures} for a detailed discussion of the definition of the groups $Pin^{\pm}$ and the notion of $Pin^\pm$ structures.
We say that the fibration $X\supset L\to \W$ is \textbf{relatively $Pin^{\pm}$} if $p^{\pm}(T^vL)\in \Ima\left(i^*:H^2(X)\to H^2(L)\right)$, and $Pin^{\pm}$ if $p^{\pm}(T^vL)=0$. A \textbf{relative $Pin^{\pm}$ structure $\fp$ on $L$} is a relative $Pin^{\pm}$ structure on $T^vL$.
\end{definition}
\begin{remark}
The condition that the vertical boundary with respect to $\pi^X$ is empty may be replaced with an appropriate convexity property.
\end{remark}

We fix a symplectic fibration $(X,\w,\W,\pi^X)$ with an exact Lagrangian subfibration $L$ whose fibers are connected. For $t\in \W$, we write $X_t,L_t$ for the fibers of $\pi^X,\pi^L$, respectively, and $\w_t$ for the restriction of $\w$ to $X_t$.
Set
\[
\zort{L}:=\zcort{\pi^L}[1-n].
\]
\begin{definition}\label{vertical orientability definition}
We say that the fibration $L$ is \textbf{vertically orientable} if $(\pi^L_*\zort{L})\neq\emp$. This is equivalent to the fiber being orientable.
\end{definition}
\begin{definition}
Let $b\in \Z$. We define a sheaf on $\W$
\[
\mathcal{X}_L^{b}:=
\pi_*^L\left(\zort{L}^{\otimes b}\right).
\]
\end{definition}
\begin{definition}
$b\in \Z$ is called \textbf{an exponent for $L$} if $\mathcal X_L^{b}$ is nonempty. In this case, the canonical map $\pi_L^*\mathcal X_L^{b}\to \zort{L}^{\otimes b}$ is an isomorphism, since both are $\Z/2$ local systems.
\end{definition}
\begin{remark}
$b\in\Z$ is an exponent for $L$ if and only if either $b$ is even or $L$ is vertically oriented.
\end{remark}
\begin{definition}\label{vertical relative homology definition}Let $\underline{H_2}(X;\Z)$ (resp.
$\underline {H_2}(X,L;\Z)$) be the sheaf over $\W$ given by sheafification of the presheaf with sections over an open set $U\subset \W$ given by
\[
H_2\left({\left(\pi^X\right)}^{-1}(U);\Z\right),\qquad \text{resp.}\quad
H_2\left({\left(\pi^X\right)}^{-1}(U),{\left(\pi^L\right)}^{-1}(U);\Z\right).\]
\end{definition}
The sheaves $\underline {H_2}(X;\Z)$ and $\underline {H_2}(X,L;\Z)$ are the local systems with fibers ${H_2}(X_t;\Z)$ and $H_2(X_t,L_t;\Z)$ for $t\in \W$, respectively, with the Gauss Manin connection. Let \[
\underline {c_1}:\underline{H_2}(X;\Z)\to \underline\Z,\qquad \underline\mu:\underline{H_2}(X,L;\Z)\to \underline\Z
\]  be the morphisms of local systems given by the fiberwise first Chern class and Maslov index, respectively. Moreover, let 
\[\underline\w:\underline{H_2}(X;\Z)\to \underline \R,\qquad \underline\w:\underline{H_2}(X,L;\Z)\to \underline \R\] be the morphisms of local systems given over $t\in \W$ by 
\[
\underline\w|_t(\b_t)=\int_{\b_t}i_t^*\w,\qquad \b_t\in H_2(X_t;\Z) \text{   or   } H_2(X_t,L_t;\Z),
\]
where $i_t:X_t\to X$ is the inclusion.
\begin{lemma}\label{constancy of w on H2 lemma}
The morphisms $\underline{c_1},\underline\mu$ and $\underline\w$ are constant on local sections of $\underline {H_2}(X;\Z)$ and $\underline{H_2}(X,L;\Z)$.
\end{lemma}

\begin{definition}
A \textbf{target} is an octuple $\target:=(\W,X,\w,L,\pi^X,\fp,\underline\Upsilon,J)$ as follows. \begin{enumerate}
    \item 
$\W$ is manifold with corners.
\item $\pi^X:X\to \W$ is a symplectic fibration with respect to $\w$.
\item $L$ is an exact Lagrangian subfibration with a relative $Pin^\pm$ structure $\fp$.
\item The map $\pi^L:=\pi^X|_L$ is a proper submersion.
\item $\underline\Upsilon \subset \ker(\underline\mu\oplus\underline\w)$ is a subsheaf such that the quotient $\underline{H_2}(X,L;\Z)/\underline\Upsilon$ is a globally constant sheaf.
\item $J=\{J_t\}_{t\in\W}$ is a $\w$-tame almost complex structure on $T^vX$.
\end{enumerate}
The dimension of $\target$ is defined to be $\dim\target:=\dim \pi^X.$ 
\end{definition}
\begin{remark}
The above definition differs from that in~\cite{orientors}, in the additional requirement that $\pi^L$ is proper. This, because in the current paper we will use pushforward of forms along $\pi^L$. We believe that this extra assumption may be removed by working with differential forms with compact support, under appropriate geometric assumptions on $X,L$.
\end{remark}
\begin{definition}\label{set of degrees of target definition}
Let $\target:=(\W,X,\w,L,\pi^X,\fp,\underline\Upsilon,J)$ be a target. The \textbf{group of degrees of $\target$} which we denote by $\Pi:=\Pi(\target)$ is the fiber of $\underline{H_2}(X,L;\Z)/\underline\Upsilon.$
Lemma~\ref{constancy of w on H2 lemma} implies that the local-systems morphisms $\underline\mu,\underline\w$ descend to maps $\mu:\Pi\to \Z$ and $\w:\Pi\to \R$.
A degree $\b\in \Pi$ is called \textbf{admissible} if $\mu(\b)+1$ is an exponent for $L$. Denote by $\Pi^{ad}\subset\Pi$ the admissible degrees. 
\end{definition}
\begin{example}\label{rotating circle example}
Consider $\R P^1$ as lines in the $yz$ plane and $S^2$ as the unit vectors in the $xyz$ space. For $t\in \R P^1$ and a vector $\vec v\in S^2$, we denote $\vec v\perp t$ if $\vec v$ is perpendicular to $t$. Set $\W=\R P^1$ and $X= \R P^1\times S^2$. Denote by $\pi:X\to \W$ and $p:X\to S^2$ the projections. Let $\w=p^*\w_0$ and $J=p^*J_0$ where $\w_0,J_0$ are the standard symplectic form and complex structure on $S^2$, respectively. Let
\[
L=\left\{(t,\vec v)\in X\mid \vec v\perp t\right\}.
\]
Namely, $L$ is a circle rotating on its diameter. Note that $\w|_L=0$. In particular,
$L\subset X$ is an exact Lagrangian subfibration. It is both relatively $Pin^+$ and relatively $Pin^-$. This may be seen as follows. $L$ is the Klein bottle and $T^vL\simeq\left(\pi^L\right)^*\mathcal O_{\R P^1}(-1)$. By the naturality of the characteristic classes $p^\pm$, it follows that $L$ is both $Pin^+$ and $Pin^-$ as a fibration. Let $\fp$ be any $Pin^\pm$ structure on $L$. 
The fibration $L$ is vertically orientable, yet the map $\pi^L$ is not relatively orientable. Moreover, we have 
\[
H_2(X_t,L_t;\Z)=\Z\oplus \Z
\]
and parallel transporting $(x,y)\in H_2(X_t,L_t;\Z)$ along the loop $\R P^1$ we get $(x,y)\mapsto (y,x)$. Let $\underline\Upsilon=\ker(\underline\mu\oplus \underline\w)$, which is the M\"obius $\Z$ bundle over $\R P^1$. 
The octuple
\[
\target_0:=\left(\W,X,\w,L,\pi^X,\fp,\underline\Upsilon,J\right)
\]
is a target. It holds that $\Pi(\target_0)=\Z$. Alternatively, we can take $\underline\Upsilon=2\cdot\ker(\underline\mu\oplus\underline\w)$ and then $\Pi=\Z\oplus \Z/2.$
\end{example}

Let $\target=(\W,X,\w,L,\pi^X,\fp,\underline\Upsilon,J)$ be a target. Recall that the relative $Pin^\pm$ structure $\fp$ determines a class $w_\fp\in H^2(X;\Z/2)$ such that $p^\pm(T^vL)=i^*w_\fp$, where $i:L\to X$ is the inclusion. By abuse of notation, we think of $w_\fp$ as a morphism of local systems $w_\fp:\underline{H_2}(X;\Z)\to \Z/2$. Denote by $\varpi:\underline{H_2}(X;\Z)\to \underline{H_2}(X,L;\Z)$ the canonical map.
\begin{definition}\label{spherical degrees definition}
Let $\underline\Upsilon'\subset \ker\left(\underline{c_1}\oplus\underline\w\oplus w_\fp\right)$ be a subsheaf such that $\varpi(\underline\Upsilon')\subset \underline\Upsilon$ and $\underline{H_2}(X;\Z)/\underline\Upsilon'$ is a globally constant sheaf. The Abelian group of \textbf{absolute degrees} $\Pi'(\target,\underline\Upsilon')$ is the fiber of $\underline{H_2}(X;\Z)/\underline\Upsilon'$. In particular, $\underline{c_1},\underline\w$ and $w_\fp$ descend to maps $c_1:\Pi'\to\Z,\w:\Pi'\to \R$ and $w_\fp:\Pi'\to \Z/2$. Denote by $\b_0\in\Pi'$ the zero element. $\varpi$ descends to a map $\varpi:\Pi'\to \Pi$.
\end{definition}

\subsection{Moduli spaces of stable maps}\label{ssec:moduli spaces}
Fix a target $\target=(\W,X,\w,L,\pi^X,\fp,\underline\Upsilon, J)$.
For $k\geq-1,l\geq0$ and $\b\in \Pi$, denote by
\[
\mM_{k+1,l}(\b):=\mM_{k+1,l}(\target;\b):=\left\{(t,\fu)\mid t\in \W, \fu\in \mM_{k+1,l}(X_t,L_t,J_t;\b_t)\right\}.
\]
Denote by $\pi^{\mM}:\mM_{k+1,l}(\b)\to\W$ the map $(t,\mathfrak u)\mapsto t$.
Denote by
\[
\begin{split}
    evb_j^\b:\mdl{3}\to L,\qquad &j=0,...,k,\\
    evi_j^\b:\mdl{3}\to X,\qquad &j=1,...,l,
\end{split}
\]
the evaluation maps given by 
\begin{align*}
    evb_j^\b(t,(\S,u,\vec z,\vec w))&=(t,u(z_j)),\\evi_j^\b(t,(\S,u,\vec z,\vec w))&=(t,u(w_j)).
\end{align*}
We may omit the superscript $\b$ when the omission does not create ambiguity.

Similarly, for $\b\in\Pi'$, let $\mM_{l+1}(\b)$ be the moduli space of stable $J$-holomorphic spheres with $l+1$ marked points indexed from $0$ to $l$ representing the class $\b$. It is of dimension $2c_1(\b)+2n-4+2l$ and it has a canonical orientation $\mO_c^{\mM_l(\b)}$.
Let $ev_j^\b:\mM_{l+1}(\b)\to X$ be the evaluation maps. It is of relative dimension $2c_1(\b)-4+2l$.
\begin{definition}\label{canonical orientation of moduli of spheres definition}
Let $l\geq0$ and $\b\in\Pi'$. The \textbf{canonical relative orientation} $\mO^{ev_0}_c$ of $ev_0^\b$ is the relative orientation of $ev^\b_0$ satisfying $\mO_c^{\mM_{l+1}(\b)}=\mO^\w\circ \mO_c^{ev_0}$, where $\mO^\w$ is the relative orientation of $\pi^X$ provided by $\w$. 
\end{definition}

To streamline the exposition, we assume that $\mM_l(\b')$ and $\mdl{3}$ are smooth orbifolds with corners and $ev_0^{\b'}$ and $evb_0^\b$ are proper submersions for $\b'\in \Pi'$ and $\b\in\Pi$. These assumptions hold in a range of important examples~\cite[Example 1.5]{Sara1}.

In general, the moduli spaces $\mM_l(\b)$ and $\mdl{3}$ are only metrizable spaces. They can be highly singular and have varying dimension. Nonetheless, the theory of the virtual fundamental class being developed by several authors~\cite{Fu09a,FO19,FO20,HW10,HWZ21} allows one to perturb the $J$-holomorphic map equation to obtain moduli spaces that are weighted branched orbifolds with corners and evaluation maps that are smooth. Thus, we may consider pullbacks of differential forms by $ev_i^{\b'}, evb_i^{\b}$ and $evi_i^\b$. Furthermore, by averaging over continuous families of perturbations, one can make $evb_0^\b$ behave like a submersion. So, the push-forward of differential forms along $evb_0^\beta$ is well-defined. See~\cite{Fu09a,FO19,FO20}. When the unperturbed moduli spaces are smooth of expected dimension and $evb_0^\beta$ is a submersion, one can choose the perturbations to be trivial. Furthermore, as explained in~\cite{Fu09a,FO19}, one can make the perturbations compatible with forgetful maps of boundary marked points. The compatibility of perturbations with forgetful maps of interior marked points has not yet been worked out in the Kuranishi structure formalism in the context of differential forms.

\subsection{Base change}
\begin{definition}
Let $\target:=(\W,X,\w,L,\pi^X,\fp,\underline\Upsilon,J)$ be a target. Let $\W'$ be as manifold with corners and $\xi:\W'\to \W$ be any map. By Lemma~\ref{pullback of exact manifold is exact naturality lemma}, $\xi^*\pi^L:\xi^L\to \W'$ is an exact submersion with respect to $\left((\xi^X)^*\w\right)|_L$. We get a target \[\xi^*\target=
\left(\W',\xi^*X, \xi^*\w,\xi^*L,\xi^*\pi^{X},\xi^*\fp, \xi^*\underline\Upsilon,\xi^*J\right).
\]
Since pullback of sheaves is an exact functor, the canonical map
\[
\underline{H_2}(\xi^*X,\xi^*L;\Z)/\xi^*\underline\Upsilon\to \xi^*\left(\underline{H_2}(X,L;\Z)/\underline\Upsilon\right)
\] is an isomorphism, so $\xi^*\target$ is indeed a target. In particular, the canonical map \[\xi^*:\Pi(\target)\to \Pi(\xi^*\target)\] is an isomorphism.
\end{definition}
\begin{remark}
If $\xi:\W'\to \W$ is a map, $\target:=(\W,X,\w,L,\pi^X,\fp,\underline\Upsilon,J)$ is a target and $\b\in \Pi(\target)$, then 
\[
\mM_{k+1,l}(\xi^*\target;\xi^*(\b))=\xi^*\mM_{k+1,l}(\target;\b).
\]Moreover, for $i\leq k$ and $j\leq l$,
\[
evb_i^{\left(\xi^*\target\right)}=\xi^*\left(evb_i^{\target}\right),\qquad evi_j^{\left(\xi^*\target\right)}=\xi^*\left(evi_j^{\target}\right).
\]
\end{remark}

\subsection{Structure of moduli spaces}\label{structure of moduli spaces section}
The orbifold structure of $\mM_{k+1,I}(\b)$ arises from the automorphisms of open stable maps. Vertical corners of codimension $r$ along the map $\pi^{\mM}$ consist of open stable maps $(\S,u,\vec z,\vec w)$ where $\S$ has $r$ boundary nodes. For $r=0,1,2,...$ denote by 
\[\mM_{k+1,I}(\b)^{(r)}\subset\mM_{k+1,I}(\b)\] the dense open subset consisting of stable maps with no more than $r$ boundary nodes and no interior nodes. Each of these subspaces is an essential subset of $\mM_{k+1,I}(\b)$.
A precise description of the vertical corners in the case $r=1$ is given in terms of gluing maps, as follows. 

Let $k\geq-1,l\geq 0,\b\in\Pi$. Fix partitions $k_1+k_2=k+1,\b_1+\b_2=\b$ and $I\dot\cup J=[l]$, where $k_1>0$ if $k+1> 0$. When $k+1>0$, let $0< i\leq k_1$. When $k=-1$ let $i=0$. Let
\[
B^{(1)}_{i,k_1,k_2,I,J}(\b_1,\b_2)\subset \pa^v \mM_{k+1,l}^{(1)}(\b)
\]
denote the locus of two component stable maps, described as follows. One component has degree $\beta_1$ and the other component has degree $\beta_2.$ The first component carries the boundary marked points labeled $0,\ldots,i-1,i+k_2,\ldots,k,$ and the interior marked points labeled by $I.$ The second component carries the boundary marked points labeled~$i,\ldots,i+~k_2-1$ and the interior marked points labeled by $J.$ The two components are joined at the $i$th boundary marked point on the first component and the $0$th boundary marked point on the second. Let 
\[
B_{i,k_1,k_2,I,J}(\b_1,\b_2):=\overline{B^{(1)}_{i,k_1,k_2,I,J}(\b_1,\b_2)}\subset \pa^v\mdl{3}
\] denote the closure. Denote by 
\[
\iota^{\b_1,\b_2}_{i,k_1,k_2,I,J}:B_{i,k_1,k_2,I,J}(\b_1,\b_2)\to \mdl{3}
\] the inclusion of the boundary.

There is a canonical gluing map 
\[
\vartheta_{i,k_1,k_2,\b_1,\b_2,I,J}:\mdl{1}_{evb_i^{\b_1}}\times_{evb_0^{\b_2}}\mdl{2}\to B_{i,k_1,k_2,I,J}(\b_1,\b_2).
\]
This map is a diffeomorphism, unless $k = -1, I = \emptyset = J$ and $\beta_1 = \beta_2.$ In the exceptional case, $\vartheta$ is a $2$ to $1$ local diffeomorphism in the orbifold sense.
The dense open subset 
\[
\mM_{k_1+1,I}^{(0)}(\b_1)\times_L\mM_{k_2+1,J}^{(0)}(\b_2)
\]
is carried by $\vartheta_{i,k_1,k_2,\b_1,\b_2,I,J}$ onto $B^{(1)}_{i,k_1,k_2,I,J}(\b_1,\b_2)$.
We abbreviate
\begin{align*}
B&=B_{i,k_1,k_2,I,J}(\b_1,\b_2),\\
\vartheta&=\vartheta_{i,k_1,k_2,\b_1,\b_2,I,J},\\
\iota&=\iota^{\b_1,\b_2}_{i,k_1,k_2,I,J}
\end{align*} 
when it creates no ambiguity.
The images of all such $\vartheta^\b$ intersect only in codimension $2$, and cover the vertical boundary of $\mdl{3}$, unless $k=-1$ and $\b\in \text{Im}(\varpi)$.
In the exceptional case there might occur another phenomenon of bubbling, in which other boundary components $B(\hat \b)$ arise, for $\hat\b\in \varpi^{-1}(\b)$, where a generic point is a sphere of class $\hat\b$ intersecting $L$ at a marked point.
There is a diffeomorphism
\[
\widehat\vartheta_{k+1,\widehat\b}:\mM_{k+1}(\widehat \b)_{ev_0}\times_X L\to B(\widehat \b).
\]
Such spheres arise when the boundary of a disk collapses to a point. 


\subsection{Orientors over the moduli spaces}\label{Orientors over the moduli spaces section}
Let $\A$ be a graded commutative ring. Typically, we will be interested in the case where $\A$ is either $\R,\C$.
\begin{definition}\label{local systems of ring rort definition}
Denote by $\lort_{L}$ the local system of orientations of $\pi^L$ with values in $\A$ concentrated in degree $-1$. Set
\[
\rort_L:=\bigoplus_{j\in \Z}\lort_L^{\otimes j}.
\]Here, negative powers correspond to the dual local system. Denote by \[m:\rort_L\otimes \rort_L\to~\rort_L,\quad 1_L:\underline \A\to \rort_L\] the tensor product and the inclusion in degree $0$, respectively, which provide $\rort_L$ the structure of a local system of unital graded non-commutative rings.
\end{definition}

Extend the marked boundary points to $\Z$ cyclicly. In particular, $evb_{k+1}=evb_0$. 
\begin{definition}
Let $i,j\in\Z$. \textbf{The parallel transport} along the oriented boundary from $j$ to $i$ is a map $c_{ij}:(evb_j)^*\rort_L\to (evb_i)^*\rort_L$ given, over a point $(t,\S,u,\vec z,\vec w)\in \mM_{k+1,l}(\b)$, by trivializing the $\left(u|_{\pa \S}\right)^*TL$ along the oriented arc from $z_i$ to $z_j$.
\end{definition}

Set
\begin{align*}
ev:=ev^\b:=(evb_1,...,evb_k):\mM_{k+1,l}(\b)&\to \underbrace{L\times_\W\cdots\times_\W L}_{k\text{ times}},\\
ev^{cyc}:=(evb_1,...,evb_k, evb_0):\mM_{k+1,l}(\b)&\to \underbrace{L\times_\W\cdots\times_\W L}_{k+\text{ times}}.
\end{align*}

\begin{definition}\label{fundamental orientors Q definition}
Set $E:=E^k:=E_L^k:=ev^*\left(\underset{j=1}{\overset{k}\boxtimes\rort}\right)$.
Let
\[
\qor_{k,l}^\b:=\qor_{k,l}^{(\target;\b)}:E_L^k\to \cort{evb_0}\otimes(evb_0)^*\rort_L
\] indexed by
\[(k,l,\b)\in\Big(\Z_{\geq0}\times\Z_{\geq0}\times \Pi(\target)\Big)\setminus\Big\{ (0,0,\b_0),(1,0,\b_0),(2,0,\b_0),(0,1,\b_0)\Big\}\]
be the family of \orientor{evb_0^{(k,l,\b)}}s of $E^k_L$ to~$\rort_L$ constructed in~\cite[Definition 6.16]{orientors}.
\end{definition}
Let 
\[
p_2:\mM_{k_1+1,I}(\b_1){}_{evb_i^{\b_1}}\times_{evb_0^{\b_2}}\mM_{k_2+1,J}(\b_2)\to \mM_{k_2+1,J}(\b_2) 
\]denote the projection. The orientors $\qor_{k,l}^\b$ are surjective for $k\geq 1$ and injective for $k=0, 1$. The family of orientors from Definition~\ref{fundamental orientors Q definition} satisfies the following theorems.

\begin{theorem}\label{boundary of q theorem}
Let $k,l\geq0$ and $\b\in\Pi.$ Let $k_1,k_2\geq0$ be such that $k_1+k_2=k+1$, let $I\dot\cup J=[l]$ and $\b_1+\b_2=\b$. Set\[
E_1=\underset{j=1}{\overset{i-1}\boxtimes}\left(evb_j^{\b_1}\right)^*\rort_L,\qquad E_3=\underset{j=i+k_2}{\overset{k}\boxtimes}\left(evb_{j-k_2+1}^{\b_1}\right)^*\rort_L.
\]
The following equation of \orientor{\left(evb_0^{\b}\circ\iota\circ\vartheta\right)}s of $\vartheta^*\iota^*{ev^\b}^*\left(\underset{j=1}{\overset{k}\boxtimes}\rort_L\right)$
to $\rort_L$ holds.
\begin{equation*}
\expinv{\vartheta}(\pa \qor_{k,l}^\b)=(-1)^s\qor_{k_1,I}^{\b_1}\bu {}^{E_1}\left(\expinv{\left(p_2/evb_i^{\b_1}\right)}\qor_{k_2,J}^{\b_2}\right)^{E_3}
\end{equation*}
Here,
\[
s=i+ik_2+k+\d\mu(\b_1)\mu(\b_2),
\]
where $\d\in \{0,1\}$ is $0$ exactly when $\mathfrak p$ is a relative $Pin^+$ structure.
\end{theorem}
\begin{theorem}\label{energy zero q theorem}
In case $(k,l,\b)\in\{(2,0,\b_0),(0,1,\b_0)\}$, the map \[evb_0^{\b_0}=\cdots=evb_k^{\b_0}\]
is a diffeomorphism, and we have
\begin{align*}
    \qor_{2,0}^{\b_0}&=\expinv{\left(evb^{\b_0}_0\right)}m,\\
    \qor_{0,1}^{\b_0}&=\expinv{\left(evb^{\b_0}_0\right)}1_L.
\end{align*}
\end{theorem}
\begin{definition}\label{efield and Otm definition}
Let $\efield_L:=\pi^L_*\rort_L$ be the sheaf pushforward along $\pi^L$. In~\cite{orientors}, we construct a surjective \orientor{L} of $\rort_L$ to $\efield_L$,
\[
\Otm:\rort_L\to \cort{L}\otimes \left(\pi^L\right)^*{\efield_L}.
\]
Moreover, we denote by $\Otm_{odd}$ the orientor that agrees with $\Otm$ on the odd homogeneous part of $\rort_L$ and vanishes on the even homogeneous part of $\rort_L$.
\end{definition}
Informally, $\Otm$ splits off one copy of $\lort_L$, shifts it by degree $1-n$ to $\cort{L}$, and maps the remaining tensor products to $\efield_L$ depending on whether they admit a vertical section.

Denote by $f:\mM_{k+1,l}(\b)\to \mM_{k+1,l}(\b)$ the map that cyclicly shifts the boundary points $(z_0,...,z_k)$ as follows, 
\[
f^\b(\S,u,(z_0,...,z_k),\vec w)=(\S,u,(z_1,...z_k,z_0),\vec w).
\]
The map $f$ is a diffeomorphism. 
Set \begin{align*}
    ev^{cyc}:=&(evb_1,...,evb_k,evb_0):\mM_{k+1,l}(\b)\to L^{\times_\W k+1},\\
    E^{cyc}:=&\left(ev^{cyc}\right)^*\rort_L^{\boxtimes k+1}.
\end{align*}
Let
\begin{align*}
    \tau: \rort_L^{\boxtimes k+1}&\to\rort_L^{\boxtimes k+1},\\
    a_0\otimes \cdots\otimes a_k&\mapsto (-1)^{|a_0|\cdot\sum_{j=1}^k|a_j|}a_1\otimes \cdots \otimes a_k\otimes a_0 
\end{align*}
denote the graded symmetry isomorphism.

\begin{theorem}\label{cyclic theorem}
The following equation of \orientor{\mM_{k+1,l}(\b)}s of $E^{cyc}$ to $\efield_L$ holds.
\[
\expinv{ f }\left(\Otm_{odd}\bu m\bu\left(\qor_{k,l}^\b\otimes \Id\right)\right)\bu (ev^{cyc})^*\t=(-1)^k\Otm_{odd} \bu m\bu \left(\qor_{k,l}^\b\otimes \Id\right).
\]
\end{theorem}
\begin{definition}\label{findamental orientors Q -1 definition}
In~\cite{orientors}, we construct a family of \orientor{\pi^{\mM_{0,l}(\b)}}s of $\underline \A$ to $\efield_L$
\[
\qor^\b_{-1,l}:=\qor^{(\target;\b)}_{-1,l}:\underline\A\to \cort{\pi^{\mM_{0,l}(\b)}}\otimes \left(\pi^{\mM_{0,l}(\b)}\right)^*\efield_L,
\]
indexed by
\[
(l,\b)\in\Z_{\geq0}\times \Pi(\target).
\]
\end{definition}
The orientors $\qor_{-1,l}^\b$ are injective.
Let $\Pi^{ad}(\target)\subset\Pi(\target)$ be the subset of degrees such that $\pi^L_*\lort^{\mu(\b)+1}$ is nonempty.
For $\b\notin \Pi^{ad}(\target)$ we have $\qor^\b_{-1,l}\equiv0$. This is inevitable, since $\pi^{\mM_{0,l}(\b)}$ is relatively non-orientable in this case. Moreover, we define $\qor_{-1,1}^{\b_0}=0$ and $\qor_{-1,0}^{\b_0}=0$.
\begin{theorem}\label{boundary of q k=-1 theorem}
Let $l\geq 0$ and $\b\in \Pi$. Let $I\dot\cup J=[l]$ and $\b_1+\b_2=\b$. The following equation of \eorientor{\mM_1\times_L\mM_2}s of $\A$ holds.
\[\expinv{\vartheta}\left(\pa \qor_{-1,l}^\b\right)=-\Otm \bu m\bu
\left(\qor_{0,I}^{\b_1}\right)^{\rort_L}\bu \expinv{\left(p_2/evb_0^{\b_1}\right)}\qor_{0,J}^{\b_2}
\]
\end{theorem}
Let $Fi:=Fi^\b_{k+1,l}:\mM_{k+1,l+1}(\b)\to \mM_{k+1,l}(\b)$ denote the map that forgets the $l+1$st marked interior point, and stabilizes the resulting curve. Similarly, let $Fb:=Fb^\b_{k+1,l}:\mM_{k+2,l}(\b)\to \mM_{k+1,l}(\b)$ denote the map that forgets the $k+1$st marked boundary point, and stabilizes the resulting curve. The maps $Fi,Fb$ have canonical relative orientations $\mO^{Fi}, \mO^{Fb}$, respectively, for which the following holds.
\begin{theorem}\label{factorization q through forgetful theorem}
Let $k\geq -1,l\geq 0$ and $\b\in \Pi$. The following equation of \orientor{evb_0}s holds.
\[
\qor_{k,l+1}^\b=\expinv{(Fi,\mO^{Fi})}\qor_{k,l}^\b.
\]
Denote by $evb_0^{k+1}$ (resp. $evb_0^k$) the evaluation map for $\mM_{k+2,l}(\b)$ (resp. $\mM_{k+1,l}(\b)$).
The following equation of \orientor{evb_0^{k+2}}s of $E^{k+1}$ to $\rort$ holds
\[
\qor_{k+1,l}^\b=m\bu \left(\expinv{\left(Fb_{k+1,l},\mO^{Fb}\right)}\qor^\b_{k,l}\right)^\rort\bu \,\,{}^{Fb^*E^k}\left(c_{k+2,k+1}\right).
\]
\end{theorem}

\subsection{Base change}
\label{q-naturality families}

Let $\xi:\W'\to \W$ be a map.
Let $\target=\left(\W,X,\w,L,\pi^X,\fp,\underline\Upsilon,J\right)$ be a target over $\W$. Let
\[\pull\xi_\rort:\left(\xi^{L}\right)^*\rort_L\to \rort_{\xi^*L}\]
be the map given by $\pull{\left(\xi^{L}/\xi\right)}:{\xi^{L}}^*\lort_L\to \lort_{\xi^*L}$ extended as an algebra homomorphism to ${\xi^{L}}^*\rort_L$. Set
\[\pull\xi_{\efield}:=\pi^L_*\pull\xi_{\rort}:\xi^*\efield_L\to \efield_{\xi^*L}.\]
We think of $\pull\xi_\rort$ and $\pull\xi_\efield$ as \orientor{\Id_{\xi^*L} \text{ and } \Id_{\W'}}s, respectively.
\begin{remark}\label{naturality of m families}
$\pull\xi_\rort,\pull\xi_\efield$ are algebra homomorphisms with respect to the corresponding direct sum and tensor multiplication maps $m_L,m_{\xi^*L}$.
\end{remark}
\begin{proposition}\label{naturality of Otm families}
With the above notations, the following diagram is commutative.
\[
\begin{tikzcd}
\xi^*\rort_L\ar[rr,"\pull\xi_{\rort}"]\ar[d,swap,"\expinv\xi \Otm^L"]&&\rort_{\xi^*L}\ar[d,"\Otm^{\xi^*L}"]\\
\cort{\pi^{\xi^*L}}\otimes \xi^*\eort_{L}\ar[rr,swap,"1\otimes (\pi^L)^*\pull\xi_{\efield}"]&&\cort{\xi^*L}\otimes \eort_{\xi^*L}
\end{tikzcd}
\]
That is, the following equation of \orientor{\pi^{\xi^*L}}s holds.
\[
\pull\xi_{\efield}\bu \expinv\xi \Otm^L=\Otm^{\xi^*L}\bu \pull\xi_{\rort}.
\]
\end{proposition}

Let $k\geq0,l\geq 0$ and $\b\in \Pi(\target)$. Abbreviate \[\mM:=\mM_{k+1,l}(\target;\b),\qquad \mM':=\mM_{0,l}(\xi^*\target;\xi^*\b).\] Let 
\begin{equation}
\label{xi E not blackboard definition}
\pull{\xi}_E:=\underset{j=1}{\overset{k}\boxtimes} {\left(evb_j^{\b}\right)}^*\pull\xi_\rort:{\xi^{\mM}}^*E^k_L\to E^k_{\xi^*L}.
\end{equation}

We think of $\pull\xi_E$ as an \orientor{ \Id_{\mM_{k+1,l}(\xi^*\target;\xi^*\b)}}.

\begin{theorem}\label{naturality of Q families}
The following diagram is commutative.
\[
\begin{tikzcd}
{\xi^{\mM}}^*E^k_L\ar[rr,"\pull\xi_E"]\ar[d,swap,"\expinv{\xi}\qor^{\left(\target;\b\right)}_{k+1,l}"]&&E^k_{\xi^*L}\ar[d,"\qor^{\left(\xi^*\target;\xi^*\b\right)}_{k+1,l}"]\\
\cort{evb_0^{\left(\xi^*\target\right)}}\otimes {\xi^{\mM}}^*\left(evb_0^{\target}\right)^*\rort_L\ar[rr,swap,"1\otimes evb_0^*\left(\pull{\xi}_\rort\right)"]&&\cort{evb_0^{\left(\xi^*\target\right)}}\otimes\left(evb_0^{\left(\xi^*\target\right)}\right)^*\rort_{\xi^*L}
\end{tikzcd}
\]
That is, the following equation of orientors holds.
\[
\pull\xi_\rort\bu\expinv{\xi}\qor^{\left(\target;\b\right)}_{k+1,l}=\qor^{\left(\xi^*\target;\xi^*\b\right)}_{k+1,l}\bu \pull\xi_E.
\]
Similarly, for $\b\in \Pi^{ad}(\target)$, the following diagram is commutative.
\[
\begin{tikzcd}
\underline\A\ar[drr,"Q_{-1}^{\left(\xi^*\target;\xi^*\b\right)}"]\ar[d,swap,"\expinv{\xi}\qor^{\left(\target;\b\right)}_{-1,l}"]\\\cort{\pi^{\mM'}}\otimes {\pi^{\mM'}}^*\xi^*\efield_L\ar[rr,swap,"1\otimes {\pi^{\mM'}}^*\pull\xi_{\efield}"]
&&\cort{\pi^{\mM'}}\otimes {\pi^{\mM'}}^*\efield_{\xi^*L}
\end{tikzcd}
\]That is, the following equation of orientors holds.
\[
\pull\xi_{\efield}\bu \expinv\xi Q^{\left(\target;\b\right)}_{-1,l}=Q^{\left(\xi^*\target;\xi^*\b\right)}_{-1,l}
\]
\end{theorem}


\section{Pushforward of forms}


For a detailed discussion of differential forms on orbifolds with corners, we refer to \cite{Sara-corners}.

\subsection{Pushforward with relative orientation}
Let $f:M\to N$ be a relatively-oriented proper surjective submersion of orbifolds with corners. Let $\mO^f$ be a relative orientation of $f$. Denote by
\[
{(f,\mO^f)_*}:A(M)\to A(N)
\]
the \textbf{oriented pushforward of forms through} $f$ defined in \cite[Section 4.1]{Sara-corners}. Note that ${(f,\mO^f)_*}$ is of degree $-m+n$. The following is proven in \cite[Theorem 1]{Sara-corners}:
\begin{proposition}
\label{properties of oriented pushforward}
The following properties characterize the oriented pushforward.
\begin{enumerate}
    \item Integration: For a compact oriented orbifold $M$ with orientation $\mO^M$, and a differential form $\a\in~A(M)$
    \[
    {(\pi^M,\mO^M)_*}(\a)=\int\limits_{M,\mO^M}\a.
    \]
    \item Fubini's Theorem: Let $g:P\to M$, $f:M\to N$ be proper submersions with relative orientations $\mO^g,\mO^f$. Then
    \[
    {(f\circ g,\mO^f\circ\mO^g)_*}={(f,\mO^f)_*}\circ {(g,\mO^g)_*}.
    \]
    \item Linearity: Let $f:M\to N$ be a proper submersion and $\a\in A(N)$, $\b\in A(M)$. Then
    \[
    {(f,\mO^f)_*}(f^*\a\we \b)=\a\we\left( {(f,\mO^f)_*}\b\right), \qquad {(f,\mO^f)_*}(\b\we f^*\a)=(-1)^{f\a}\left({(f,\mO^f)_*}\b\right)\we \a.
    \]
    \item Fiberwise: Let 
    \[
    \begin{tikzcd}
    M\times_N P\ar[r,"p"]\ar[d,"q"]&P\ar[d,"g"]\\
    M\ar[r,"f"]&N
    \end{tikzcd}
    \] be a pullback diagram of smooth maps, where $g$ is a proper submersion with relative orientation $\mO^g$. It follows that $q$ is also a proper submersion, with relative orientation $\pull{(p/f)}\mO^g$ given in Definition \ref{pullback fiber product orientation convention}. Then for ${\a\in A(P)}$
    \[
    f^*\left({(g,\mO^g)_*}\a\right)={(q,\pull{(p/f)}\mO^g)_*}\left(p^*\a\right).
    \] 
\end{enumerate}
\end{proposition}
Furthermore, we have the following generalization of Stoke's theorem

\begin{proposition}[Relatively oriented Stoke's theorem]
\label{relatively oriented stokes}
Let $f:M\to N$ be a proper submersion and $\xi\in A(M)$, and let ${\iota_f :\pa_f M\to M}$ be the vertical-boundary inclusion. Recall the canonical relative orientation $\mO_c^{\iota_f}\in\cort{\iota_f}$ given in Definition \ref{relative orientation of boundary}.
Then
\[
d\left((f,\mO^f)_*\xi\right)=(f,\mO^f)_*(d\xi)+(-1)^{f+\xi}(f\circ\iota_f,\mO^f\circ \mO_c^{\iota_f})_*(\iota_f^*\xi).
\]
\end{proposition}
\begin{remark}[Sign difference]
The sign in this proposition differs by the sign $(-1)^{\dim N}$ from the corresponding proposition in \cite{Sara-corners}. This comes from the difference in the choice of relative orientation $\mO_c^{\iota_f}=(-1)^N o_c^{\iota_f}$.
\end{remark}
\subsection{Differential forms with values in local systems}
We assume henceforth that $\A$ is a commutative $\R$-algebra.
A local system of modules (resp. algebra) means a local system of graded $\A$-modules (resp. $\A$-algebra).
\begin{notation}[Differential forms with values in a local system]
If $Q$ is a local system of modules (resp. algebras) over $M$ then the $Q$-valued differential forms are the sections of the graded $\R$-vector (resp. $\R$-algebra) bundle
$\Lambda^*T^*M\otimes_\R Q$, i.e.
\[
A^*(M;Q):=\Gamma(\La^*T^*M\otimes Q).
\]
The tensor product of a vector bundle and a local system of modules is the standard tensor product, that is the vector bundle with transition functions given by the tensor product of the transition functions of the factors.
They inherit their additive (resp. multiplicative) structure from the corresponding structure on $A^*T^*M\otimes Q.$
\end{notation}
\begin{notation}[Functoriality of $\Gamma(\cdot)$]
\label{functoriality of forms and Gamma}
Let $Q,S$ be local systems of $\A$-modules (resp. $\A$-algebras) over $M$.
\begin{enumerate}
    \item \label{functoriality of forms and Gamma morphisms} A morphism of local systems $F:Q\to S$ induces a graded-linear map (resp. graded-homomorphism)
\[
F_*:A(M,Q)\longrightarrow A(M,S)
\]
as 
\[
F_*:=\Gamma(\Id\otimes F).
\]
\item The map \[(\La^*T^*M\otimes Q)\otimes (\La^*T^*M\otimes S)\overset{ \tau_{Q,\La^*T^*M}}\longrightarrow (\La^*T^*M\otimes\La^*T^*M)\otimes (Q\otimes S)\overset{\we\otimes \Id\otimes\Id}\longrightarrow\La^*T^*M\otimes Q\otimes S\]
induces an extended multiplication
\[
\bigwedge:A(M;Q)\otimes A(M;S)\to A(M;Q\otimes S).
\]
\end{enumerate}
\end{notation}
\begin{proposition}
\label{pushforward of symmetry Proposition}
Let $Q_1,Q_2$ be local systems over $M$, and let 
\[\tau:Q_2\otimes Q_1\to Q_1\otimes Q_2\]
denote the graded symmetry operator. Let $\xi_i\in A(M;Q_i)$ be with degree $|\xi_i|$, for $i=1,2$. Then
\[
\tau_*(\xi_2\we\xi_1)=(-1)^{|\xi_1||\xi_2|}\xi_1\we\xi_2.
\]
\end{proposition}
\begin{proof}
Assume, without loss of generality, that $\xi_i=\a_i\otimes q_i$ with $\a_i\in A(M)$ and $q_i$ is a section of $Q_i$. On one hand,
\begin{align*}
\tau_*(\xi_2\we\xi_1)&=(-1)^{|q_2||\a_1|}\tau_*(\a_2\we\a_1\otimes q_2\otimes q_1)\\
&=(-1)^{|q_2||\a_1|}\a_2\we\a_1\otimes \tau(q_2\otimes q_1)\\
&=(-1)^{|q_2||\xi_1|+|q_1||q_2|+|\a_1||\a_2|}\a_1\we\a_2\otimes q_1\otimes q_2\\
&=(-1)^{|q_1||\a_2|+|q_2||\xi_1|+|q_1||q_2|+|\a_1||\a_2|}\xi_1\we\xi_2.
\end{align*}
However, the proposition follows since $|\xi_i|=|\a_i|+|q_i|$.
\end{proof}

\subsection{Pushforward of orientation-valued forms}\label{pushforward section}
Using partitions of unity, we can define a more general operation. For a proper submersion $f:M\to N$, not necessarily relatively oriented, and a local system $K$ over $N$, we define the pushforward
\[
f_*:A(M;\cort{f}\otimes f^*K)\to A(N,K)
\] as follows. Note that it is of null degree.

Let $U\subset M$ be an open subset such that both $\cort f|_U$ and $K|_{f(U)}$ are trivial. Let $\xi \in A\left(U;\cort{f}|_U\otimes f^*K|_U\right)$. Then $\xi $ can be written as a sum of differential forms of the form
\[
\a\otimes \mO^f\otimes f^*k,
\]
where $\a\in A(U)$, $\mO^f$ is a local relative orientation of $f$ and $k$ is a parallel section of $K|_{f(U)}$.
We define
\[
f_*(\a\otimes \mO^f\otimes f^*k)=\left((f,\mO^f)_*\a\right)\otimes k
\]
and extend linearly to $A\left(U;\cort{f}|_U\otimes f^*K|_U\right)$. For a global differential form 
\[\xi\in A(M;\cort{f}\otimes f^*K)\] we define $f_*\xi$ using a partition of unity.

\begin{proposition}[Properties of pushforward]
\label{properties of pushforward}
The following properties characterize the pushforward.
\begin{enumerate}
    \item Integration: For a compact orientable orbifold $M$, and ${\a\otimes \mO^M\in A(M;\cort{M})}$
    \[
    {{\pi^M}_*}(\a\otimes \mO^M)=\int\limits_{M,\mO^M}\a.
    \]
    \item Fubini's Theorem: Let $g:P\to M$, $f:M\to N$ be proper submersions and let $K$ be a local system over $N$. Then, under the canonical isomorphism $\cort{f\circ g}\to \cort{g}\otimes g^*\cort{f}$ from Definition \ref{composition isomorphism}, the following diagram is commutative.
    \[
    \begin{tikzcd}
    A\left(P;\cort{f\circ g}\otimes (f\circ g)^*K\right)\ar[dr,swap,"(f\circ g)_*"]\ar[r,"g_*"]&A(M;\cort{g}\otimes g^*K)\ar[d,"f_*"]\\
    &A(N;K)
    \end{tikzcd}
    \]
    \item Linearity: Let $f:M\to N$ be a proper submersion, let $S,K$ be local systems over $N$. Let $\eta\in A(N;S)$, ${\xi\in A(M;\cort{f}\otimes f^*K)}$. Then the following diagram is commutative.
    \[
    \begin{tikzcd}
    A(M;\cort{f}\otimes f^*K)\otimes A(N;S)\ar[r,"f_*\otimes\Id"]\ar[d,"- \we f^*-"]&A(N;K)\otimes A(N;S)\ar[d,"\we"]\\
    A\left(M;\cort{f}\otimes f^*K\otimes f^*S\right)\ar[r,"f_*"]&A(N;K\otimes S)
    \end{tikzcd}
    \]
    That is,
    \[
    f_*(\xi\we f^*\eta)=f_*\xi\we\eta.
    \]
    It also follows that \[f_*(f^*\eta\we \xi)=\eta\we f_*\xi.\]
    \item\label{pushforward orientors:base change property} Base Change: Let 
    \[
    \begin{tikzcd}
    M\times_N P\ar[r,"p"]\ar[d,"q"]&P\ar[d,"g"]\\
    M\ar[r,"f"]&N
    \end{tikzcd}
    \] be a pullback diagram of smooth maps, where $g$ is a proper submersion. It follows that $q$ is also a proper submersion. Let $K$ be a local system over $N$. Extend the isomorphism $\pull{(p/f)}:p^*\cort{g}\simeq \cort{q}$ given in Definition \ref{pullback fiber product orientation convention} to differential forms,
    \[
    \pull{(p/f)}:A(M\times_NP;p^*\cort{g}\otimes p^*g^*K)\to A(M\times_NP;\cort{q}\otimes q^*f^*K).
    \] Then
    \[
    f^*g_*=q_*\pull{(p/f)}p^*.
    \]
\end{enumerate}
\end{proposition}
\begin{proof}
\begin{enumerate}
    \item This follows directly from Property $1$ of the oriented pushforward.
    \item Let $\xi \in A\left(P;\cort{f\circ g}\otimes (f\circ g)^*S\right)$ be a form. Without loss of generality, we may assume that
    \[
    \xi=\a\otimes (\mO^{f}\circ\mO^g)\otimes (f\circ g)^*s,
    \]
    where $\a\in A(P)$, $\mO^f,\mO^g$ are relative orientations of $f,g$, respectively, and $s$ is a section of $S$.
    So
    \begin{align*}
    (f\circ g)_*\xi&=(f\circ g,\mO^f\circ \mO^g)_*(\a)\otimes s\\&=\left((f,\mO^f)_*\circ (g,\mO^g)_*\a\right)\otimes s\\&=f_*\left((\left(g,\mO^g)_*\a\right)\otimes \mO^f\otimes f^*s\right)
    \\&=f_*\circ g_*(\a\otimes f^*\mO^g\otimes\mO^f\otimes g^*f^*s)\\&=f_*\circ g_*\xi.
    \end{align*}
    
    \item
Without loss of generality  $\xi=\a\otimes \mO^f\otimes f^*k$, and $\eta=\b\otimes s$, where $\a\in A(M),\b\in A(N),\mO^f$ is a relative orientation of $f$, and $k,s$ are sections of $K,S$, respectively. Then
\begin{align*}
    f_*(f^*\eta\we \xi)&=(-1)^{s(\a-f)}f_*(f^*\b\we \a\otimes \mO^f\otimes s\cdot k)\\&=(-1)^{s(\a-f)}(f,\mO^f)_*(f^*\b\we\a)\otimes s\cdot k\\
    &=(-1)^{s(\a-f)}\b\we(f,\mO^f)_*\a\otimes s\cdot k\\&=\b\otimes s\we\left((f,\mO^f)_*\a\otimes k\right)=\eta\we f_*\xi,\\
    f_*(\xi\we f^*\eta)&=(-1)^{(k-f)\b}f_*(\a\we f^*\b\otimes \mO^f\otimes k\cdot s)\\&=(-1)^{(k-f)\b}(f,\mO^f)_*(\a\we f^*\b)\otimes k\cdot s\\
    &=(-1)^{k\b}\left((f,\mO^f)_*\a\right)\we \b\otimes k\cdot s\\&=\left((f,\mO^f)_*\a\otimes k\right)\we\b\otimes s= f_*\xi\we\eta.
\end{align*}
\item This follows directly from property $4$ of the oriented pushforward.
\end{enumerate}
\end{proof}
Furthermore, we have the following generalization of Stoke's theorem.
\begin{proposition}{Stoke's theorem}
\label{Relative Stoke's}
Let $f:M\to N$ be a proper submersion and let $\xi\in A(M;\cort{f}\otimes f^*K)$. Then
\[
d(f_*\xi)=f_*d\xi+(f\circ {\iota_f})_*(\left(\pa_f\right)_*\iota_f^*\xi),
\]
where $\pa_f$ is the boundary operation of relative orientation from Definition \ref{boundary-operator for relative orientation}.
\end{proposition}
\begin{proof}
Without loss of generality, write $\xi=\a\otimes \mO^f\otimes k$. Recall, $\pa\mO^f=(-1)^{f}\mO^f\circ\mO_c^{\iota_f}$, and thus
\[
\pa_*\xi=(-1)^\a\a\otimes \pa\mO^f\otimes k=(-1)^{f+\a}\a\otimes (\mO^f\circ\mO_c^{\iota_f})\otimes k.
\]
Therefore,
\begin{align*}
d(f_*\xi)&=d\left((f,\mO^f)_*\a\right)\otimes k=(f,\mO^f)_*d\a\otimes k+(-1)^{f+\a}(f\circ {\iota_f},\mO^f\circ\mO^{\iota_f})_*\a\otimes k\\
&=f_*(d\xi)+(f\circ {\iota_f})_*(\pa_*\xi).
\end{align*}
\end{proof}

\subsection{Pushforward by orientors}
Now, we investigate the interaction between the pushforward of forms and the pushforward of orientors. For a proper submersion $g:M\to P$, bundles $Q,K$ over $M,P$, respectively, and a \orientor{g} $G:Q\to \cort{g}\otimes g^*K$,
we are interested in the composition 
\[
\begin{tikzcd}
A(M;Q)\ar[r,"G_*"]&A(M;\cort{g}\otimes g^*K)\ar[r,"g_*"]&A(P;K).
\end{tikzcd}
\]
\begin{proposition}[Integration]
\label{pushforward by phi is minus the oriented pushforward}
Let $f:M\to N$ be a proper submersion with relative orientation $\mO^f$. Let $\a\in A(M).$ Then
\[
f_*\phi^{\mO^f}_*\a=(-1)^{f|\a|}(f,\mO^f)_*\a.
\]
\end{proposition}
\begin{proof}
We calculate
\[
f_*\phi^{\mO^f}_*\a=(-1)^{f|\a|}f_*(\a\otimes \mO^f)=(-1)^{f|\a|}(f,\mO^f)_*\a.
\]
\end{proof}
\begin{proposition}[Functoriality]
\label{Fubini prop}
Let $M\overset g\longrightarrow P\overset f\longrightarrow N$ be proper submersions, $Q,K,S$ be local systems over $M,P,N$ respectively. Let $G$ be a \orientor{g} of $Q$ to $K$ and let $F$ be a \orientor{f} of $K$ to $S$.
Then, \[
(f\circ g)_*\circ\left(F\bu G\right)_*=(f_*\circ F_*)\circ (g_*\circ G_*).
\]
\end{proposition}

\begin{proof}
Since 
\[
(f\circ g)_*=f_*\circ g_*,\qquad (F\bu G)_*=F_*\circ G_*,
\]
it suffices to show
\[
g_*\circ(\Id_{\cort{g}}\otimes g^*F)_*=F_*\circ g_*.
\]
For $\xi\in A(M;\cort{g}\otimes g^*K)$, we may assume $\xi=\a\otimes \mO^g\otimes k$. Then,
\begin{align*}
g_*\circ (\Id_{\cort{g}}\otimes g^*F)_*\xi&=(-1)^{F(\a-g)}g_*\left(\a\otimes \mO^g \otimes Fk\right)\\
&=(-1)^{F(\a-g)}(g,\mO^g)_*\a\otimes Fk=F_*\left(g_*\xi\right).
\end{align*}
\end{proof}

\begin{proposition}[Module-like behavior]
\label{Module-like behavior  proposition}
Let $f:M\to N$ be a surjective proper submersion. let $Q$ be a local system over $M$ and let $X,K,Y$ be local systems over $N$. Let $F$ be an \orientor{f} of $Q$ to $K$. Then the following diagram is commutative.
\[
\begin{tikzcd}
A(N;X)\otimes A(M;Q)\otimes A(N;Y)\ar[rr,"\bigwedge \circ (f^*\otimes \Id\otimes f^*)"]\ar[d,swap,"\Id\otimes\left( f_*\circ F_*\right)\otimes\Id"]&&A(M;f^*X\otimes Q\otimes f^*Y)
\ar[d,"f_*\circ\left({}^XF^Y\right)_*"]\\
A(N;X)\otimes A(N;K)\otimes A(N;Y)\ar[rr,"\bigwedge"]&&A(N;X\otimes K\otimes Y)
\end{tikzcd}
\]
\end{proposition}
\begin{proof}
Let ${\a\otimes x\in A(N;X),}{\xi\otimes q\in A(M;Q)}$ and ${\b\otimes y\in A(N;Y)}$. Assume, without loss of generality, that ${Fq=\mO^f\otimes k}$. Following the left and bottom arrows, we obtain
\begin{align*}
\bigwedge\circ(\Id\otimes f_* \otimes\Id)&\circ(\Id\otimes F_*\otimes\Id)(\a\otimes x\otimes \xi\otimes q\otimes\b\otimes y)\\
=
(-1)^{F(\a+x+\xi)}&\quad\bigwedge\circ(\Id\otimes f_*\otimes\Id)\left(\a\otimes x\otimes(\xi\otimes \mO^f\otimes k)\otimes \b\otimes y   \right)\\
=(-1)^{F(\a+x+\xi)}&\quad (\a\otimes x)\we (f_*(\xi\otimes \mO^f)\otimes k)\we(\b\otimes y)\\
=(-1)^{x(\xi-f+\b)+k\b+F(\a+x+\xi)}&\quad \a\we f_*(\xi\otimes \mO^f)\we\b\bigotimes x\otimes k\otimes y.
\end{align*}
Observe that
\begin{align*}
&\quad(\tau_{X,\cort{f}}\otimes\Id_K\otimes\Id_Y)\circ(\Id\otimes F\otimes\Id)(x\otimes q\otimes y)\\
=(-1)^{Fx}&\quad(\tau_{X,\cort{f}}\otimes\Id_K\otimes\Id_Y)(x\otimes \mO^f\otimes k\otimes y)\\
=(-1)^{fx+Fx}&\quad \mO^f\otimes x\otimes k\otimes y.
\end{align*}
Following the top and right arrows now, we obtain
\begin{align*}
(-1)^{fx+F(x+\a+\xi+\b)+x(\xi+\b)+q\b}&\quad f_*\,\, (f^*\a\otimes\xi\otimes f^*\b\otimes \mO^f\otimes x\otimes k\otimes y) \\
=(-1)^{fx+F(x+\a+\xi+\b)+x(\xi+\b)+q\b}&\quad  f_*(f^*\a\we\xi\we f^*\b\otimes \mO^f)\bigotimes x\otimes k\otimes y\\
=(-1)^{f\b+fx+F(x+\a+\xi+\b)+x(\xi+\b)+q\b}&\quad \a\we f_*(\xi\otimes \mO^f)\we\b\bigotimes x\otimes k\otimes y.
\end{align*}
Comparing the signs between the expressions, one can see that the only elements that do not immediately cancel out are
\[
k\b+f\b+F\b+q\b=(k+f+F+q)\b.
\]
However, $Fq=\mO^f\otimes k$ so the degrees satisfy $F+q=k-f$, which reduces to \[k+f+F+q=_2 0.\]
\end{proof}
\begin{proposition}[Base Change]
\label{base change orientors proposition}
Consider the following fiber-product diagram.
\[
\begin{tikzcd}
M\times_NP\ar[r,"p"]\ar[d,"q"]&P\ar[d,"g"]\\M\ar[r,"f"]&N
\end{tikzcd}
\]
Let $K,S$ be local systems over $P,N$, respectively. Let $G$ be a \orientor{g} of $K$ to $S$. Then
\[
f^*g_*G_*=q_*\left(\expinv{\left(p/f\right)}G\right)_*p^*.
\]
\end{proposition}
\begin{proof}
This follows immediately from property~\ref{pushforward orientors:base change property} of Proposition~\ref{properties of pushforward}.
\end{proof}
\begin{proposition}[Stoke's Theorem]
\label{Stoke's theorem}
\[d(g_*G_*\xi)=(g\circ\iota)_*(\pa G)_*\iota_g^*\xi+(-1)^Gg_*G_*d\xi.\]
\end{proposition}
\begin{proof}
Apply Proposition \ref{Relative Stoke's} to the form $G_*\xi$, and note that \[\pa(G_*\xi)=(\pa G)_*\xi,\qquad dG_*\xi=(-1)^GG_*d\xi.\]
\end{proof}
\begin{proposition}
\label{pullbcak inverse pushforward orientors}
Let $M\xrightarrow{f}P\xrightarrow{g}N$ be proper submersions, and assume $\mO^f$ is a relative orientation of $f$. Let $K,S$ be local systems over $P,N$, respectively. Let $G$ be a \orientor{g} of $K$ to $S$. Let $\xi\in A(P;K),\eta\in A(M)$. Then
\[
(g\circ f)_*\left(\expinv{\left(f,\mO^f\right)}G\right)_*\left(f^*\xi\we\eta\right)=(-1)^{f(G+|\xi|+|\eta|)}g_*G_*\left(\xi\we (f,\mO^f)_*\eta\right).
\]
In particular, when $f$ is a diffeomorphism,
\[
(g\circ f)_*\left(\expinv{f}G\right)_*(f^*\xi)=g_*G_*\xi.
\]
\end{proposition}
\begin{proof}
We calculate,
\begin{align*}
(g\circ f)_*\left(\expinv{\left(f,\mO^f\right)}G\right)_*\left(f^*\xi\we\eta\right)\overset{\text{Def.~\ref{pullback of orientor}}}=&(-1)^{Gf}(g\circ f)_*\left(G\bu\phi^{\mO^f}\right)_*\left(f^*\xi\we\eta\right)\\
\overset{\text{Prop.~\ref{Fubini prop}}}=&(-1)^{Gf}g_*G_*f_*\phi^{\mO^f}_*\left(f^*\xi\we\eta\right)\\
\overset{\text{Prop.~\ref{Module-like behavior  proposition}}}=&(-1)^{(G+|\xi|)f}g_*G_*\left(\xi\we f_*\phi^{\mO^f}_*\eta\right)\\
\overset{\text{Prop.~\ref{pushforward by phi is minus the oriented pushforward}}}=&(-1)^{(G+|\xi|+|\eta|)f}g_*G_*\left(\xi\we \left(f,{\mO^f}\right)_*\eta\right).
\end{align*}
This proves the first statement. Assuming $f$ is a diffeomorphism, the second statement is a special case of the first one, in which $\mO^f=\mO^f_c$ and $\eta=1$.

Recalling Example \ref{expinv by diffeomorphism over Id}, the second statement also follows at once by Proposition \ref{base change orientors proposition} applied to the following pullback diagram.
\[
\begin{tikzcd}
M\ar[r,"f"]\ar[d,swap,"g\circ f"]&P\ar[d,"g"]\\N\ar[r,swap,"\Id_N"]&N
\end{tikzcd}
\]
\end{proof}

\section{Currents}
\label{currents section}
For a detailed discussion of currents on oriented orbifolds with corners, see~\cite{Sara-corners}.  All results in~\cite{neat-embeddings-of-orbifolds-with-corners} continue to hold in the following generalization to non-orientable orbifolds, as explained below. We use the results stated in this section to prove Proposition~\ref{divisors}, which implies the divisor property~\ref{algebra deformation conclusion theorem: divisor} in Theorem~\ref{algebra deformation conclusion theorem}. Our proof relies on the case where $\W=\{*\}$, which appears in~\cite[Proposition 4.16]{Sara1}. The course of the proof uses Proposition~\ref{0 current with constant fibers is a function proposition} which relates our setting to that of~\cite[Proposition 4.16]{Sara1}  

\subsection{Main definitions}
Let $\pi:=\pi^M:M\to \W$ be an orbifold with corners over a manifold with corners. A set $A\subset M$ is called \textbf{proper} with respect to $\pi$ if $\pi|_A:A\to \W$ is a proper map. Let $E=\oplus_{i\in \Z}E_i$ be a local system of free $\A$-modules over $M$, such that $\dim E_i<\infty$ for all $i\in \Z$. Recall that by Definition \ref{dual graded vector space definition}, $E^{\vee\vee}=E$.
For a local system $S$ over $M$ and a smooth map $e:N\to M$, denote by $A^*_c(M,e;S)\subset A^*(M;S)$ the subspace of differential forms with \textbf{proper} support with respect to $\pi$, which vanish when pulled back by $e$. When $S=\underline\R$ abbreviate $A^*_c(M,e)$. When $e:N\to M$ is an inclusion of a submanifold, abbreviate $A^*_c(M,N;S)$.
Recall that the ring $A^*(\W)$ acts on $A_c^*(M,e;S)$ for any local system $S$ over $M$ and any smooth map $e:N\to M$ by
\[
\xi\cdot \eta={\pi^M}^*\xi\we\eta,\qquad \xi\in A^*(\W),\quad \eta\in A_c^*(M,e;S).
\]
\begin{definition}\label{linear functionals definition}
Let $\pi:=\pi^M:M\to \W$ and $E$ be as above, and let $e:N\to M$ be a smooth map of orbifolds with corners such that $\pi\circ e:N\to \W$ is an orbifold with corners over $\W$. 
A \textbf{graded $A^*(\W)-$linear functional} $\zeta$ is a map \[\zeta:A_c^{*}\left(M,e;E^\vee\otimes \cort{\pi}\right)\to A^{*+|\zeta|}(\W)\] such that
\[
\zeta({\pi}^*\xi\we\eta)=(-1)^{|\zeta|\cdot|\xi}\xi\we \zeta(\eta).
\]
\end{definition}
Let $B\subset \pa_\pi M$ be a closed and open subset, that is, a union of vertical boundary components. Denote by $B^c=\pa_\pi M\setminus B$ the compliment subset of $\pa_\pi M$, which is closed and open. It comes with a canonical map $e_{B^c}:B^c\to M$, which we use tacitly in the following definition, in accordance with the discussion above.
\begin{definition}\label{currents definition}
Let $\pi:M\to \W, E, B$ be as above.
The space of \textbf{vertical currents on $M$ which vanish on $B$} along $\pi$ of cohomological degree $k$ with coefficients in $E$, denoted $\overline A^k_\pi(M,B;E)$, is the graded $A^*(\W)-$linear functionals \[A^{*}_c\left(M,{B^c};E^\vee\otimes \cort{\pi^M}\right)\to A^{*+k}(\W).\]
We often forget the adjective ``vertical", which hopefully creates no confusion, as $\pi$ is specified.

The graded $\A$-module $\overline A^*_{\pi}(M,B;E)$ is equipped with a differential as follows.
\begin{gather}
   \notag d:\overline A^k_\pi(M,B;E)\to \overline A^{k+1}_\pi(M,B;E),\\
   d\zeta(\a)=d\left(\zeta(\a)\right)-(-1)^{|\zeta|}\zeta(d\a).
\label{differential of currents equation}
\end{gather}
It is a routine calculation to check that $d\zeta$ is $A^*(\W)$-linear. We abbreviate $A^*_\pi(M;E)=A^*_\pi(M,\emp;E)$.
\end{definition}
\subsection{Structure}

Denote by $\nu_E:E\otimes E^\vee\to\mathbb F$ the canonical pairing $(v,v^\vee)\mapsto v^\vee(v)$.
Differential forms which vanish on $B$ are embedded as subspaces of currents which vanish on $B$ as follows.

\begin{definition}\label{inclusion of forms in currents definition}
Denote by 
\[
\phi:=\phi_B:A^k(M,B;E)\hookrightarrow \overline A^k_\pi(M,B;E)
\]
the inclusion given by
\begin{equation}
\label{forms as currents equation}
\phi(\eta)(\a)={\pi^M}_*\left(\nu_E\otimes\Id_{\cort{M}}\right)_*(\eta\we\a), \quad \a\in A_c^{\dim M-k}(M,{B^c}; E^\vee\otimes \cort{\pi^M}).
\end{equation}
\end{definition}
\begin{lemma}\label{differential commutes with inclusion of forms lemma}
With the above notations, we have
\[
d\left(\phi(\eta)\right)=\phi(d\eta).
\]
\end{lemma}

Let $S$ be an $\mathbb F$-algebra and let $\mu:S\otimes E\to E$ (resp. $\mu:E\otimes S\to E$) be a left (resp. right) module structure.
In this situation, $\overline A^*(M,B;E)$ is a left (resp. right) module over $A^*(M;S)$, with action
\[
(\eta\we\zeta)(\g):=(-1)^{|\eta|\cdot|\zeta|}\zeta\left(\left(\mu^\vee\right)_*(\eta\we\g)\right),\qquad\text{resp.}\quad (\zeta\we\eta)(\g):=\zeta\left(\left(\mu^\vee\right)_*(\eta\we\g)\right),
\]
for
\[
\g\in A_c^*(M,B^c; E^\vee\otimes \cort{\pi^M}),\quad \eta\in A^*(M;S),\quad \zeta\in \overline A^*(M,B;E).
\]
This module structure makes $\phi$ a module homomorphism.

Let $\pi^N:N\to \W$ and $\pi^M:M\to \W$ be orbifolds with corners over $\W$, and let $f:M\to N$ be a smooth map over $\W$, that is, $\pi^N\circ f=\pi^M$. In particular, $\pa_fM$ is a closed and open subset of the vertical boundary $\pa_\pi M$.
\begin{definition}\label{pushforward of currents along map}
Let $f:M\to N$ be as above, and let $E$ be a local system over $N$. The \textbf{pushforward} of currents along $f$
\[
f_*:\overline A^*_{\pi^M}(M,\pa_f M;\cort{f}\otimes f^*E)\to \overline A^*_{\pi^N}(N;E)
\]
is defined as follows.
\[
(f_*\zeta)(\xi) = \zeta(f^*\xi), \qquad \zeta\in \overline A^k_{\pi^M}(M,\pa_f M;\cort{f}\otimes f^*E),\quad \xi\in A^{\dim M-k}(N; E^\vee\otimes \cort{\pi^N}).
\]
Here, we use the composition isomorphism
\[
\cort{f}^\vee\otimes \cort{\pi^M}\simeq f^*\cort{\pi^N}
\]
to interpret $f^*\xi$ as an element of $A\left(M;f^*E^\vee\otimes \cort{f}^\vee\otimes \cort{\pi^M}\right).$
\end{definition}

\begin{definition}
\label{pushforward of currents by bundle map}
Let $\pi^M:M\to \W$ be an orbifold with corners over $\W$, and let $Q,K$ be local systems over $M$. Let $F$ be a morphism of local systems from $Q$ to $K$. The \textbf{pushforward of currents by $F$} is a map
\[
F_*: \overline A^k_{\pi^M}(M,B; Q)\to \overline A^k_{\pi^M}(M,B; K)
\]
by
\[
(F_*\zeta)(\xi) =(-1)^{F
\cdot|\zeta|} \zeta({\left({F^\vee}\otimes\Id_{\cort{\pi^M}}\right)}_*\xi),\]
for\[ \zeta\in \overline A^k_{\pi^M}(M,B; Q),\quad \xi\in A^{\dim M-k}\left(M,B^c;K^\vee\otimes \cort{\pi^M}\right),
\]
where $F^\vee$ is from Definition \ref{dual graded linear map definition}.
\end{definition}
\begin{lemma}\label{pushforward commutes with inclusion of forms lemma}
Let $f:M\to N$ be a smooth map of orbifolds over $\W$, and let $Q,K$ be local systems over $M,N$, respectively. Let $F$ be an \orientor{f} of $Q$ to $K$. Then
\begin{equation}
\label{current pushforward on forms equation}
f_*F_*\phi(\eta)=\phi(f_*F_*\eta),\qquad \eta\in A(M;Q).
\end{equation}
\end{lemma}
\begin{proof}
Let $\xi \in A(N;K^\vee\otimes \cort{\pi^N})$. On one hand,
\begin{align*}
\left(\phi(f_*F_*\eta)\right)(\xi)&\overset{\text{  eq. \ref{forms as currents equation}  }}=\,\,{\pi^N}_*\left(\nu_K\otimes\Id_{\cort{\pi^N}}\right)_*\left(f_*F_*\eta\we \xi\right)\\
&\overset{\text{Prop. \ref{Module-like behavior  proposition}}}={\pi^N}_*\left(\nu_K\otimes\Id_{\cort{\pi^N}}\right)_*f_*\left(F^{K^\vee\otimes \cort{\pi^N}}\right)_*\left(\eta\we f^*\xi\right)\\
&\overset{\text{Prop. \ref{Fubini prop}}}={\pi^M}_*{\left(\nu_K\bu F^{K^\vee}\right)^{\cort{\pi^N}}}_*(\eta\we f^*\xi).
\end{align*}
On the other hand,
\begin{align*}
\left(f_*F_*\phi(\eta)\right)(\xi)&\overset{\text{Def. \ref{pushforward of currents along map}}}=\left(F_*\phi(\eta)\right)\left(f^*\xi\right)\\
&\overset{\text{Def. \ref{pushforward of currents by bundle map}}}=(-1)^{F|\eta|}\phi(\eta)\left(\left(F^\vee\otimes\Id_{\cort{\pi^M}}\right)_*f^*\xi\right)\\
&\overset{\text{  eq. \ref{forms as currents equation}  }}=\,\,(-1)^{F|\eta|}{\pi^M}_*\left(\nu_Q\otimes\Id_{\cort{\pi^M}}\right)_*\left(\eta\we \left(F^\vee\otimes\Id_{\cort{\pi^M}}\right)_*f^*\xi\right)\\
&={\pi^M}_*\left(\nu_Q\otimes\Id_{\cort{\pi^M}}\right)_*\left( \left(\Id_Q\otimes F^\vee\otimes\Id_{\cort{\pi^M}}\right)_*\left(\eta\we f^*\xi\right)\right)\\
&={\pi^M}_*\left(\nu_Q\otimes\Id_{\cort{\pi^M}}\right)_*\left( \left(\Id_Q\otimes F^\vee\otimes\Id_{\cort{\pi^M}}\right)_*\left(\eta\we f^*\xi\right)\right)
\end{align*}
Then equation~\eqref{current pushforward on forms equation} is follows from the commutativity of the following diagram, which is a consequence of Remark \ref{dual map diagram definition remark}.
\[
\begin{tikzcd}
Q\otimes f^*K^\vee\otimes \cort{f}^\vee\otimes \cort{\pi^M}\ar[rrr,"\Id_Q\otimes F^\vee\otimes\Id_{\cort{\pi^M}}"]\ar[d,swap,"F^{\left(K^\vee\right)}\otimes comp."]&&&Q\otimes Q^\vee\otimes \cort{\pi^M}\ar[d,"\nu_Q\otimes\Id_{\cort{\pi^M}}"]\\
\cort{f}\otimes f^*K\otimes f^*K^\vee\otimes f^*\cort{\pi^N}\ar[rr,"\Id_{\cort{f}}\otimes \nu_K\otimes\Id_{\cort{\pi^N}}"]&&\cort{f}\otimes f^*\cort{\pi^N}\ar[r,"comp."]& \cort{\pi^M}
\end{tikzcd}
\]
\end{proof}

\subsection{Restriction of currents}
Let $\xi:\W'\to \W$ be a closed neat embedding of orbifolds with corners, as defined in~\cite{neat-embeddings-of-orbifolds-with-corners}. Denote by $\pi':M'\to \W'$ the pullback of $\pi:M\to \W$ along $\xi$ and denote by $\xi^M:M'
\to M$ the pullback of $\xi$ along $\pi$. As shown in~\cite{neat-embeddings-of-orbifolds-with-corners}, $\xi^M$ is a closed neat embedding. Let $B\subset \pa_\pi M$ be a closed and open subset. Set $B':=\xi^*B\subset \pa_{\pa'}M'$.
The following definition generalizes that of~\cite{neat-embeddings-of-orbifolds-with-corners} to vertical currents with values in a local system.
\begin{definition}\label{restriction of vertical currents definition}
With the above notations, the \textbf{restriction} of currents
\[(\xi^M)^*:\overline A^*_\pi(M,B;E)\to \overline A^*_{\pi'}(M',B';\xi^*E)\]
is given as follows. 
For $\g'\in A^*_c(M',{B'}^c;E^\vee\otimes \cort{\pi^M})$, let $\g$ be any extension of $\g'$ to $M$, which vanishes on $B$. For a current $\a\in\overline A^*_\pi(M,B;E)$, define 
\[
(\xi^M)^*\a(\g')=(\xi^M)^*\left(\a(\g)\right).
\]
\end{definition}
\begin{remark}
The existence of $\g$ and the independence of $(\xi^M)^*\a$ on the choice of $\g$ in Definition~\ref{restriction of vertical currents definition} are the core of~\cite{neat-embeddings-of-orbifolds-with-corners}. 
\end{remark}
\begin{lemma}\label{restriction and differential commute lemma}
The restriction and the differential commute. That is,
\[
d(\xi^M)^*\a=(\xi^M)^*d\a.
\]
\end{lemma}
\begin{proof}
This follows from the corresponding property for differential forms, and the definitions of restriction and differential of currents.
\end{proof}
\begin{lemma}\label{restriction and pushforward commute lemma}
Let $\pi^N:N\to \W$ be an orbifold with corners over $\W$, and let $f:M\to N$ be a smooth map over $\W$. Let $Q$ be a local system over $M$ and $R$ be a local system over $N$, and $F$ be an \orientor{f} of $Q$ to $R$. Set $N':=\xi^*N,f':=\xi^*f$ and $F'=\xi^*F$. Then
\[
(\xi^N)^*(f_*F_*\a)=(f'_*F'_*)(\xi^M)^*\a,\qquad \a\in \overline A^*_{\pi^M}(M,\pa_f M;\cort{f})
\]
\end{lemma}
\begin{proof}
This follows immediately from the definition of restriction and pushforward of currents.
\end{proof}
\begin{definition}
For $t\in \W$, denote by $i_t:\{t\in \W\}\to \W$ the inclusion.
A current $\a\in \overline A^*_\pi(M,B;E)$ is called \textbf{horizontal,} if $(i_t^M)^*\a=0$ for all $t\in \W$.
\end{definition}
The proof of the following proposition appears in~\cite{neat-embeddings-of-orbifolds-with-corners}.
\begin{proposition}\label{0 current with constant fibers is a function proposition}
Let $\a\in A^0_\pi(M,B;E)$ be a current and $f\in A^0(\W)$ be such that for all $t\in \W$ and for all $\g\in A^*_c\left(M|_t,B|_t,E^\vee\otimes \cort{\pi^M}\right)$ we have \[((i_t^M)^*\a)(\g)=f(t)\cdot \int_{M_t}\g.\]
Then
\[
\a=f\cdot\phi(1),
\]
where \[\phi:A^*(M,B;E)\to \overline A^*(M,B;E)\] is the abovementioned inclusion and $1\in A^*(M,B;E)$ is the unit form. 
\end{proposition}

\section{Structure}\label{Structure section}
\subsection{The Algebra}
\label{operators subsection}
Fix a target $\target=(\W,X,\w,L,\pi^X,\fp,\underline\Upsilon,J)$.

Let $\A$ be a commutative $\R$-algebra.
Recall from Section~\ref{introduction setting section} the graded rings
\[
\tilde\La:=\tilde\La^{\target}:=\left\{\sum_{i=0}^\infty a_i T^{\b_i}\Big| a_i\in\A,\b_i\in\Pi,\w(\b_i)\geq 0,\lim_{i\to \infty}w(\b_i)=\infty \right\},
\]
and
\[
R:=R^\target:=A^*(\W;\efield_L )\otimes \tilde\La[[t_0,...,t_N]],\qquad Q:=\A[t_0,...,t_N],
\]
thought of as differential graded algebras with trivial differential. Moreover, recall \[C:=C^\target:=A^*(L;\rort_L)\otimes \tilde\La[[t_0,...,t_N]],\qquad D:=D^\target:=A^*(X;Q)\]
treated as graded modules over $R,Q$, respectively.
Let 
$\nu:\tilde\La^\target[[t_0,\ldots,t_N]]\to \R$
be the valuation given by equation~\eqref{valuation equation}.
This valuation extends to a valuation on $R,C,Q,D$ and their tensor products, which we also denote by $\nu$. Define ideals
\[
\mathcal I_R:=\{\a\in R|\nu(\a)>0\},\qquad (\text{resp.}\quad \mathcal I_Q:=\{\a\in Q|\nu(\a)>0\})
\]
or $R$ (resp. $Q$). 
\begin{definition}
\label{Poincare pairing orientors}Recall Definition \ref{efield and Otm definition}.
The signed Poincar\'e pairing \[
\langle,\rangle=\langle,\rangle^\target:C\otimes C\to R\]
is the pairing
\[
\left\langle\xi,\eta\right\rangle = (-1)^{|\xi|+n(|\xi|+|\eta|)}{\pi^L}_*\left(\Otm\bu m\right)_*\left(\xi\we\eta\right).
\]
Let \[\oddpairing{,}=\langle,\rangle_{\text{odd}}^\target:C\otimes C\to R\] be the pairing given by
\[
\oddpairing{\xi,\eta} = (-1)^{|\xi|+n(|\xi|+|\eta|)}{\pi^L}_*\left(\Otm_{odd}\bu m\right)_*\left(\xi\we\eta\right).
\]
\end{definition}
\begin{remark}\label{alternative definition of odd pairing remark}
for $i=0,1$, set $\rort_i$ the sub-local systems of $\rort_L$ of even and odd degrees, respectively. Set $C_i=A(L;\rort_i)\otimes \La[[t_0,...,t_N]]$. Then for $i\neq j\in\{0,1\},$ the pairing $\oddpairing{,}$ vanishes on $C_i\otimes C_i$ and agrees with $\langle,\rangle$ on $C_i\otimes C_j$.
\end{remark}

For integers $k,l\geq 0$ and lists of integers $a=(a_1,...,a_k)\in\left(\Z\right)^{\times k},c=(c_1,...,c_l)\in \left(\Z\right)^{\times l}$, define
\[
\e(a,c):=1+\sum_{j=1}^kj\cdot(a_j+1)+k\left(\sum_{j=1}^ka_j+\sum_{j=1}^lc_j\right),
\]
and set 
\[
\e(c)=n\cdot\sum_{j=1}^lc_j.
\]
To simplify notation in the following, we allow differential forms as input, in lieu of their degrees. In particular, for lists $\a\in C^{\times k}$ and $\g\in D^{\times l}$,
\[
\e(\a,\g):=1+\sum_{j=1}^kj\cdot(|\a_j|+1)+k\left(|\a|+|\g_j|\right).
\]
Set
\[
\e(\g)=n|\g|.
\]
\begin{definition}
\label{rho definition}
let $\rho$, possibly with no $\a$ input, be either
\[
\rho_c(\b;\a,\g):=(-1)^{\e(\a,\g)+{\left(\begin{smallmatrix}\d\mu(\b)\\ 2\end{smallmatrix}\right)}},
\]
or, assuming $\A$ contains $\C$,
\[
\rho_i(\b;\a,\g):=\begin{cases}
(-1)^{\e(\a,\g)},& \d\cdot\mu(\b)=_2 0,\\
(-1)^{\e(\a,\g)}\cdot \sqrt{-1},& \d\cdot\mu(\b)=_2 1.
\end{cases}
\]
\end{definition}
Recall the family $\qor_{k,l}^\b$ of \orientor{evb_0}s from Definition~\ref{fundamental orientors Q definition}.
\begin{definition}[the operators $\oqb{3}$]\label{def:oqb}
For $k\geq 0$, define maps of degree $2-k-2l$
\begin{align*}
\oqb{3}=\fq^{\target,\b}_{k,l}:D^{\otimes l}\otimes C^{\otimes k}  \longrightarrow C,\quad k\geq 0,
\end{align*}
and a map of degree $4-n-2l$
\begin{align*}
\oqb{-3}=\fq^{\target,\b}_{-1,l}:D^{\otimes l}\longrightarrow R,    
\end{align*}
As follows. For $\b\in\Pi, k,l\geq0$, satisfying $(k,l,\b)\notin \{(1,0,\b_0),(0,0,\b_0)\}$, define
\begin{align}
\oqb{3}(\g_1\otimes \cdots\otimes \g_l;\a_1\otimes\cdots\otimes \a_k)&:=\rho(\b;\a,\g)\left(evb_0^\b\right)_*\left(\qor_{k,l}^\b\right)_*\left(\bigwedge_{j=1}^l evi_j^*\g_j\we\bigwedge_{j=1}^kevb_j^*\a_j\right).
\label{oq in terms of q equations}
\end{align}
For the remaining cases, we define
\begin{align}\oqb{0}(\a)&:=d\a,\quad \fq^{\b_0}_{0,0}:=0,\notag\\
\oqb{-3}(\g_1\otimes \cdots\otimes \g_l)&:=\rho(\b;\g){\pi^{\mM_{0,l}(\b)}}_*\left(\qor_{-1,l}^{\b}\right)_*\left(\bigwedge_{j=1}^levi_j^*\g_j\right).
\label{oq in terms of q equations k=-1}
\end{align}
The case $\fq_{0,0}^\b$ is understood as $\rho(\b;\emp;\emp)(evb_0^\b)_*1$. 

Lastly, define similar operators using spheres as follows.
For $l\geq0,\b\in\Pi'$ such that $(l,\b)\neq(1,\b_0),(0,\b_0)$, let $\qor_{\emp,l}^\b:=\phi^{\mO_c^{ev_0}}$ be the \eorientor{ev_0} of $\underline\A$ from Definition~\ref{orientation as orientor} applied to the relative orientation from Definition~\ref{canonical orientation of moduli of spheres definition}. Define maps 
\[
\fq_{\emp,l}^\b=\fq_{\emp,l}^{\target,\b}:A^*(X;Q)^{\otimes l}\to A^*(X;R)
\]
of degree $4-2c_1(\b)-2l$ as follows. For $(l,\b)\notin \{(1,\b_0),(0,\b_0)$, define
\[
\fq_{\emp,l}^\b(\g_1,...,\g_l):=
(-1)^{w_\fp(\b)}(ev_0^\b)_*\left(\qor_{\emp,l}^\b\right)_*\left(\bigwedge_{j=1}^l{ev_j^\b}^*\g_j\right).
\]
For the remaining cases, we define
\[
\fq_{\emp,1}^0:=0, \qquad \fq_{\emp,0}^0:=0.
\]
\end{definition}
\begin{remark}
In the case $\W$ is a point and $L$ is oriented, there is a canonical isomorphism of differential graded algebras
\[
A(L;\rort_L)\to A(L)\otimes \efield.
\]
Under this isomorphism, the operators $\oqb{3}$ agree with those of \cite{Sara1}, up to extension of scalars. To see this, first notice that~$\mu(\b)$ is always even when~$L$ is orientable.
The difference in the sign $\e(\a,\g)$ of $\oqb{3}$ between this paper and \cite{Sara1} is $k(|\a|+|\g|)$. It compensates for the implicit sign in Notation~\ref{functoriality of forms and Gamma} part~\ref{functoriality of forms and Gamma morphisms}, appearing due to the Koszul signs~\ref{Koszul signs}. Heuristically, we pass $\qor$, which is of degree $2-k-2|I|$, over a form of degree $|\a|+|\g|$, since the relative orientation should appear on the right of the forms. Similarly, the sign $n|\g|$ in $\oqb{-3}$ compensates on the implicit sign of passing $\qor_{-1,l}^\b$, which is of degrees with parity $n$, over a form of degree $|\g|$. Moreover, when $L$ is oriented the relative $Pin$ structure $\fp$ and the orientation determine a relative $Spin$ structure $\mathfrak s$. The $Spin$ structure $\mathfrak s$ determines a class $w_{\mathfrak s}\in H^2(X;\Z/2)$ such that $w_2(TL)=i^*w_{\mathfrak s}$. It holds that $w_{\mathfrak s}=w_{\mathfrak p},$ so our definition agrees with that of \cite{Sara1}.

Moreover, the Poincar\'e pairing $\langle,\rangle$ agrees with that of \cite{Sara1}. The difference in the sign of the pairing between this paper and \cite{Sara1} is $(n-1)(|\xi|+|\eta|)$. It compensates for the implicit sign of passing the orientor $\Otm$, which is of degree $1-n$, over a form of degree $|\xi|+|\eta|$. 
\end{remark}

\subsection{Relations}
Let $P$ be an ordered 3-partition of $(1,...,k)$, i.e.
\begin{equation}
\label{3-partition example}
P=(1,...,i-1)\circ(i,...,i+k_2-1)\circ (i+k_2,...,k)=(1:3)\circ(2:3)\circ(3:3),
\end{equation}
and $I\dot\cup J$ be a partition of $[l]$. 
For  \[\a=\a_1\otimes\cdots\otimes\a_k\in C^{\otimes k},\qquad\g=\g_1\otimes \cdots \otimes \g_l\in D^{\otimes l},\] divide them with respect to the partitions $P,I$ as follows, 
\begin{align*}
\a^{(1:3)}:=&\a_1\otimes\cdots\otimes \a_{i-1},\\
\a^{(2:3)}:=&\a_i\otimes\cdots\otimes \a_{i+k_2-1},\\
\a^{(3:3)}:=&\a_{i+k_2}\otimes\cdots\otimes \a_k,\\
\g^I:=&\g_{i_1}\otimes\cdots\otimes \g_{i_{l_1}}\qquad i_1<\cdots<i_{l_1},\quad l_1=|I|,\\
\g^J:=&\g_{j_1}\otimes \cdots\otimes \g_{j_{l_2}}\qquad j_1<\cdots<j_{l_2},\quad l_2=|J|.
\end{align*}
In particular, $\a=\a^{(1:3)}\otimes \a^{(2:3)}\otimes \a^{(3:3)}$.

Further define
\[
sgn(\s^{\g}_{I,J}):=\sum_{\begin{smallmatrix}j<i\\i\in I\\j\in J\end{smallmatrix}}|\g_i|\cdot |\g_j|,
\]so that
\[
\bigwedge_{i\in I}evi_i^*\g_i\we \bigwedge_{j\in J}evi_j^*\g_j=(-1)^{sgn(\s^\g_{I,J})}\bigwedge_{r=1}^levi_r^*\g_r.
\]Finally, set
\begin{align*}
\iota(\a,\g;P,I)&=(|\a^{(1:3)}|+i-1)(1+|\g^J|)+|\g^I|+sgn(\s^\g_{I,J}),\\
\iota(\g,I)&=|\g^I|+sgn(\s_{I,J}^\g).
\end{align*}
For $a\in \N$ let $S_3[a]$ be the set of ordered $3-$partitions of $(1,...,a)$, as in equation~\eqref{3-partition example}.
The following proposition is the basis of the $A_\infty$ relations described in the introduction.
\begin{proposition}
\label{q-relations}
\[
\begin{split}
0=&\sum_{\begin{smallmatrix}S_3[l]\\(2:3)=\{j\}\end{smallmatrix}}(-1)^{|\g^{(1:3)}|+1}\oqb{3}(\g^{(1:3)}\otimes d\g_j\otimes \g^{(3:3)};\a)+\\&+
\sum_{\begin{smallmatrix}\b_1+\b_2=\b\\P\in  S_3[k]\\I\dot\cup J= [l]\end{smallmatrix}}(-1)^{\iota(\a,\g;P,I)}\fq_{|(1:3)|+|(3:3)|+1,I}^{\b_1}\left(\g^{I};\a^{(1:3)}\otimes \fq_{|(2:3)|,J}^{\b_2}(\g^{J};\a^{(2:3)})\otimes \a^{(3:3)})\right).    
\end{split}
\]
\end{proposition}
A proof is given in Section~\ref{proof of q-relations} below.

\begin{proposition}
\label{q-relations k=-1}
\begin{multline*}
-d\oqb{-3}(\g)=\sum_{\begin{smallmatrix}S_3[l]\\(2:3)=\{j\}\end{smallmatrix}}(-1)^{|\g^{(1:3)}|+1}\oqb{-3}(\g^{(1:3)}\otimes d\g_j\otimes \g^{(3:3)})+\\
+\frac{1}{2}
\sum_{\begin{smallmatrix}\b_1+\b_2=\b\\I\dot\cup J= [l]\end{smallmatrix}}(-1)^{\iota(\g;I)}\left\langle\fq_{0,I}^{\b_1}(\g^{I}),\fq_{0,J}^{\b_2}(\g^{J})\right\rangle+(-1)^{|\g|+1}\left\langle i^*\left(\sum_{\hat\b\in \varpi^{-1}(\b)}\fq_{\emp,l}^{\hat\b}(\g)\right),1\right\rangle
\end{multline*}
\end{proposition}
A proof is given is Section~\ref{proof of q-relations k=-1} below.

For all $k\geq 0$, define operators
\[
\fq_{k,l}=\fq_{k,l}^\target:D^{\otimes^l}\otimes C^{\otimes k} \to C
\]
by
\[
\fq_{k,l}\left(\bigotimes_{j=1}^l\g_j;\bigotimes_{j=1}^k\a_j\right)
=\sum_{\b\in \Pi}T^\b \oqb{3}\left(\bigotimes_{j=1}^l\g_j;\bigotimes_{j=1}^k\a_j\right).
\]
Similarly, define \[\fq_{-1,l}=\fq_{-1,l}^\target:D^{\otimes l}\to R\] as follows
\[
\fq_{-1,l}\left(\bigotimes_{j=1}^l\g_j\right):=\sum_{\b\in \Pi}T^\b\fq_{-1,l}^\b\left(\bigotimes_{j=1}^l\g_j\right).
\]
Set\[
\fq_{\emp,l}:=\sum_{\hat\b\in H_2(X)}T^{\varpi(\hat\b)}\fq_{\emp,l}^{\hat\b}(\g_1,...,\g_l).
\]
Summing Proposition~\ref{q-relations} for all $\b\in \Pi$, we get the following.
\begin{proposition}
\label{q-relations no b}
\[
\begin{split}
0=&\sum_{\begin{smallmatrix}S_3[l]\\(2:3)=\{j\}\end{smallmatrix}}(-1)^{|\g^{(1:3)}|+1}\fq_{k,l}(\g^{(1:3)}\otimes d\g_j\otimes \g^{(3:3)};\a)+\\&+
\sum_{\begin{smallmatrix}P\in  S_3[k]\\I\dot\cup J= [l]\end{smallmatrix}}(-1)^{\iota(\a,\g;P,I)}\fq_{|(1:3)|+|(3:3)|+1,I}\left(\g^{I};\a^{(1:3)}\otimes \fq_{|(2:3)|,J}(\g^{J};\a^{(2:3)})\otimes \a^{(3:3)})\right). 
\end{split}
\]
\end{proposition}
Similarly, summing Proposition~\ref{q-relations k=-1} for all $\b\in \Pi$, we get the following.
\begin{proposition}
\begin{align*}
-d\fq_{-1,l}(\g)=\sum_{\begin{smallmatrix}S_3[l]\\(2:3)=\{j\}\end{smallmatrix}}&(-1)^{|\g^{(1:3)}|+1}\fq_{-1,l}(\g^{(1:3)}\otimes d\g_j\otimes \g^{(3:3)})+\\
&+\frac{1}{2}
\sum_{\begin{smallmatrix}\b_1+\b_2=\b\\I\dot\cup J= [l]\end{smallmatrix}}(-1)^{\iota(\g;I)}\left\langle\fq_{0,I}(\g^{I}),\fq_{0,J}(\g^{J})\right\rangle+(-1)^{|\g|+1}\left\langle i^*\fq_{\emp,l}(\g),1\right\rangle
\end{align*}
\end{proposition}

Fix a closed form $\g\in \mI_Q D$ with $|\g|=2$. For all $k\geq 0$, define operators
\[
\om{3}=\fm^{\target,\g}_k:C^{\otimes k}\to C
\]
by
\[
\om{3}\left(\bigotimes_{j=1}^k\a_j\right):=\sum_{l}\frac{1}{l!}\oq{3}\left(\g^{\otimes l};\bigotimes_{j=1}^k\a_j\right).
\]
Similarly, define $\om{-1}\in R$ by
\[
\fm_{-1}^{\g}:=\sum_{l}\frac{1}{l!}\fq_{-1,l}\left(\g^{\otimes l}\right).
\]
\begin{proposition}[$A_\infty$ relations]
\label{A infty relations om proposition}
The operators $\{\om{3}\}_{k\geq0}$ define an $A_\infty$ structure on $C$. That is,
\[
\sum_{S_3[k]}(-1)^{\sum_{j\in(1:3)}(|\a_j|+1)}\om{1}\left(\a^{(1:3)}\otimes\om{2}(\a^{(2:3)})\otimes\a^{(3:3)}\right)=0.
\]
\end{proposition}
\begin{proof}
Since we have assumed $d\g=0$ and $|\g|=2$, this follows from Proposition~\ref{q-relations no b}.
\end{proof}
\subsection{Proof for  \texorpdfstring{$k\geq 0$}{TEXT}}
\label{proof of q-relations}
In this section, we prove Proposition~\ref{q-relations}. Thus, we fix the following.

Let $P\in S_3[k],\quad I\dot\cup J=[l]$ be partitions, and $\b_1,\b_2\in\Pi$ such that $\b_1+\b_2=\b$. Let $\a=\a_1\otimes\cdots\otimes \a_k\in C^{\otimes k}$ and $\g=\g_1\otimes \cdots\otimes \g_l\in D^{\otimes l}$. Let $k_1=|(1:3)|+|(3:3)|+1$, $k_2=(2:3)$ and $i=|(1:3)|+1$.

Recall from Section~\ref{structure of moduli spaces section} the vertical boundary component $B:=B_{i,k_1,k_2,I,J}(\b_1,\b_2)$ and the gluing map,
\[
\mdl{1}\times_L\mdl{2}\overset{\vartheta}\rightarrow B.
\]
Let $\iota:B\rightarrow \pa^v\mdl{3}$ denote the inclusion.

Define the list $\tilde\a:=(\a^{(1:3)},|\g^J|+|\a^{(2:3)}|-k_2,\a^{(3:3)})$.
Set 
\[
\zeta(\a,\g;P,I):=(i-1)|\g^J|+ik_2+k+(k_2+1)\left(|\a^{(1:3)}|+|\g^I|\right).
\]
\begin{lemma}
\label{epsilon lemma}
With the notation above,
\begin{equation*}
\label{epsilon}
\e(\a^{(2:3)},\g^J)+\e(\tilde\a,\g^I)=_2 \e(\a,\g)+\zeta(\a,\g;P,I).
\end{equation*}
\end{lemma}
\begin{proof}
Set
\[
\e_1(\a):=1+\sum_{j=1}^kj\cdot(|\a_j|+1),\qquad\e_2(\a):=k |\a|,\qquad \e_3(\a,\g)=k|\g|.
\]
So $\e(\a,\g)=\e_1(\a)+\e_2(\a)+\e_3(\a,\g).$
Lemma~2.9 of \cite{Sara1} reads
\begin{equation}
    \e_1(\tilde\a)+\e_1(\a^{(2:3)})-\e_1(\a)=_2 i\cdot|\g^J|+k_1k_2+k+(k_2+1)|\a^{(3:3)}|+ik_2+|\a^{(2:3)}|. 
\end{equation}
We calculate
\begin{equation}\begin{split}
\e_2(\a^{(2:3)})+\e_2(\tilde\a)-\e_2(\a)&= k_2\cdot|\a^{(2:3)}|+k_1\cdot (|\a|+|\g^J|-k_2)-k\cdot|\a|\\
&=k_1(|\g^J|-k_2)-k_2\cdot(|\a^{(1:3)}|+|\a^{(3:3)}|)+|\a|.
\end{split}
\end{equation}
Moreover,
\begin{align}
\e_3(\tilde\a,\g^I)+\e_3(\a^{(2:3)},\g^J)-\e_3(\a,\g)&=k_1|\g^I|+k_2|\g^J|-(k_1+k_2-1)\left(|\g^I|+|\g^J|\right)\\
\notag&=-(k_2-1)|\g^I|-(k_1-1)|\g^J|.
\end{align}
Therefore, 
\begin{equation}
    \begin{split}
        \e(\tilde\a,\g^I)+\e(\a^{(2:3)},\g^J)-\e(\a,\g)=&i|\g^J|+ik_2+k+(k_2+1)|\a^{(3:3)}|+|\a^{(2:3)}|\\
        &+k_1|\g^J|-k_2(|\a^{(1:3)}|+|\a^{(3:3)}|)+|\a|\\
        &-(k_2-1)|\g^I|-(k_1-1)|\g^J|\\
        =&_2(i-1)|\g^J|+ik_2+k-(k_2+1)\left(|\a^{(1:3)}|+|\g^I|\right).
    \end{split}
\end{equation}
\end{proof}
\begin{lemma}
\label{rho lemma}
With the notation above, 
\[
\rho\left(\b_1;\tilde\a,\g^I\right)\rho\left(\b_2;\a^{(2:3)},\g^J\right)=(-1)^{\d\mu(\b_1)\mu(\b_2)+\zeta(\a,\g;P,I)}\rho(\b;\a,\g).
\]
\end{lemma}
\begin{proof}
This is a consequence of Lemma~\ref{epsilon lemma} and, in case $\rho =\rho_c$, of the algebraic fact
\begin{equation}
\label{sum choose 2 equation}
{\begin{pmatrix}a+b\\ 2\end{pmatrix}}={\begin{pmatrix}a\\ 2\end{pmatrix}}+{\begin{pmatrix}b\\ 2\end{pmatrix}}+ab.
\end{equation}
\end{proof}

\begin{proof}[Proof of Proposition~\ref{q-relations}]\label{q-relations proof calculation}
We abbreviate
\begin{align*}
\qor=\qor_{k,l}^\b\qquad \qor_1=\qor_{k_1,I}^{\b_1}\qquad \qor_2=\qor_{k_2,J}^{\b_2}.
\end{align*}
Set 
\begin{align*}
E_1:= \underset{j=1}{\overset{i-1}\boxtimes} {\left(evb_j^{\b_1}\right)}^*\rort,\qquad
E_2:=\underset{j=i}{\overset{k_2+i-1}\boxtimes} {\left(evb_{j+i-k_2}^{\b_2}\right)}^*\rort,\qquad
E_3:=\underset{j=k_2+i}{\overset{k_1}\boxtimes} {\left(evb_{j-k_2+1}^{\b_1}\right)}^*\rort.
\end{align*}

Set $\xi=\bigwedge_{j=1}^l evi_j^*\g_j\we\bigwedge_{j=1}^kevb_j^*\a_j.$
We use Stoke's Theorem~\ref{Stoke's theorem} on equation~\eqref{oq in terms of q equations} to calculate
\begin{align}
\label{stokes in proof q-relations}
\oqb{0}\left(\oqb{3}(\g;\a)\right)=&\rho(\b;\a,\g)d\!\left({evb_0^\b}_*\qor_*\xi\right)=\\&\rho(\b;\a,\g)\left(evb_0^\b\circ\iota\right)_*\!\!(\pa \qor)_*\iota^*\xi+(-1)^{|q|}\rho(\b;\a,\g){evb_0^\b}_*\qor_* d\xi.\notag
\end{align}
First, we analyze the contribution of the vertical boundary of $\mM$ with respect to $evb_0$ in equation~\eqref{stokes in proof q-relations}.
The boundary of $\mM$ is composed of boundary components $B:=B_{k_1,k_2,i,I,J}(\b_1,\b_2)$ with \[(k_i,l_i,\b_i)\notin\{
(0,0,\b_0),(1,0,\b_0)
\},\] where $l_1:=|I|,l_2:=|J|$.
On each boundary component $B$, we can apply Proposition~\ref{boundary of q theorem}, as follows.
Set
\begin{align*}
\bar\xi:=&\vartheta^*\iota^*\xi,\\
\xi_1:=&\bigwedge_{j\in I}\left(evi_j^{\b_1}\right)^*\g_j\we\bigwedge_{j=1}^{i-1}\left(evb_j^{\b_1}\right)^*\a_j,\\
\xi_2:=&\bigwedge_{j\in J}\left(evi_j^{\b_2}\right)^*\g_j\we\bigwedge_{j=i}^{k_2+i-1}\left(evb_j^{\b_2}\right)^*\a_j,\\
\xi_3:=&\bigwedge_{j=k_2+i}^{k_1}\left(evb_j^{\b_1}\right)^*\a_j.
\end{align*}
Since
\[
evb_j\circ\iota\circ\vartheta = \begin{cases}
evb_j\circ p_1,&j<i,\\
evb_{j-i}\circ p_2,&i\leq j<i+k_2\\
evb_{j-i-k_2}\circ p_1,& i+k_2\leq j<k+1,
\end{cases}
\]
we have
\begin{equation}
\label{bar xi factorization}
\bar\xi=(-1)^{s_1}p_1^*\xi_1\we p_2^*\xi_2\we p_1^*\xi_3,
\end{equation}
with 
\[
s_1=sgn(\s_{I,J}^\g)+|\g^J||\a^{(1:3)}|.
\]
By Proposition~\ref{pullbcak inverse pushforward orientors} applied to $f=\vartheta$ and $g=evb_0^\b\circ\iota|_B$, we have
\[
(evb_0^\b\circ\iota|_B)_*(\pa \qor)_*\left(\iota^*\xi\right)=(evb_0^{\b_1}\circ p_1)_*\left(\expinv{\vartheta}\pa \qor\right)_*\bar\xi.
\]
Let
\[
s=i+ik_2+k+\d\mu(\b_1)\mu(\b_2).
\]
Setting
\[
s_2:=k_2|\xi_1|=k_2\left(|\g^I|+|\a^{(1:3)}|\right),\qquad s_3:=\d\mu(\b_1)\mu(\b_2)+\zeta(\a,\g;P,I),
\]
we calculate
\begin{align*}
(evb_0^{\b_1}\circ p_1)_*&\left(\expinv{\vartheta}\pa \qor\right)_*\bar\xi=
\\\overset{\text{Prop.~\ref{boundary of q theorem}}}=&(-1)^s\,( evb_0^{\b_1}\circ p_1)_*\left( 
\qor_1\bu {}^{E_1}\left(\expinv{\left(p_2/evb_i^{\b_1}\right)}\qor_2\right)^{E_3}
\right)_*\bar\xi\\
\overset{\begin{smallmatrix}\text{Prop.~\ref{Fubini prop}}\\\text{eq.~\eqref{bar xi factorization}}\end{smallmatrix}}=&(-1)^{s+s_1} {evb_0^{\b_1}}_*{\qor_1}_*{p_1}_*\,{}^{E_1}\!\left(\expinv{\left(p_2/evb_i^{\b_1}\right)}\qor_2\right)_*^{E_3}\,\left(p_1^*\xi_1\we p_2^*\xi_2\we p_1^*\xi_3\right)\\
\overset{\text{Prp.~\ref{Module-like behavior  proposition}}}=&
(-1)^{s+s_1}{evb_0^{\b_1}}_*{\qor_1}_*\bigwedge\left(\Id\otimes \left({p_1}_*\expinv{\left(p_2/evb_i^{\b_1}\right)}\qor_2\right)\otimes \Id \right)\left(\xi_1\otimes p_2^*\xi_2\otimes \xi_3\right)
\\
\overset{\text{Koszul~\ref{Koszul signs}}}=&(-1)^{s+s_1+s_2}{evb_0^{\b_1}}_*{\qor_1}_*\left(\xi_1\we\left({p_1}_*\left(\expinv{\left(p_2/evb_i^{\b_1}\right)}\qor_2\right)_*p_2^*\xi_2\right)\we \xi_3\right)\\
\overset{\text{Prop.~\ref{base change orientors proposition}}}=&(-1)^{s+s_1+s_2}{evb_0^{\b_1}}_*{\qor_1}_*\left(\xi_1\we\left(\left(evb_i^{\b_1}\right)^*\left(evb_0^{\b_2}\right)_*(\qor_2)_*\xi_2\right)\we \xi_3\right).
\end{align*}
Therefore, equation~\eqref{oq in terms of q equations} and Lemma~\ref{rho lemma} imply
\[
\begin{split}
\rho(\b;\a,\g)(evb_0^{\b_1}\circ p_1)_*\left(\expinv{\vartheta}\pa \qor\right)_*\bar\xi&\\=(-1)^{s+s_1+s_2+s_3}\oqb{1}&\left(\g^I;\a^{(1:3)}\otimes \oqb{2}\left(\g^J;\a^{(2:3)}\right)\otimes \a^{(3:3)}\right)    
\end{split}
\]
We simplify the sign in the above equation. Recalling the definition of $\zeta(\a,\g;P,I)$, we see that
\[
s_2+s_3=_2\d\b_1\b_2+k+|\a^{(1:3)}|+|\g^I|+{ik_2}+|\g^J|(i-1).
\]
Thus, 
\[
s+s_2+s_3=_2i+|\a^{(1:3)}|+|\g^I|+|\g^J|(i-1),
\]
and 
\begin{align*}
s+s_1+s_2+s_3&=_2+ i+(1+|\g^J|)|\a^{(1:3)}| +|\g^I|+|\g^J|(i-1)  +sgn(\s^\g_{I,J})\\
    &=_21+(1+|\g^J|)(|\a^{(1:3)}|+i-1)+|\g^I|+sgn(\s^\g_{I,J})\\
    &=1+\iota(\a,\g;P,I).
\end{align*}

We turn to analyze the contribution of $d\xi$ in equation~\eqref{stokes in proof q-relations}.
Set
\begin{align*}
\bar\g:=\bigwedge_{j=1}^{l}evi_j^*\g_j,\\
\bar\a:=\bigwedge_{i=1}^kevb_i^*\a_i.
\end{align*}
For $i\leq k$ and $j\leq l$, set 
\begin{align*}
\tilde \a_i&=\left(\a_{1},...,\a_{i-1},d\a_i,\a_{i+1},...,\a_k\right),\\
\tilde \g_j&=\left(\g_{1},...,\g_{j-1},d\g_j,\g_{j+1},...,\g_l\right).
\end{align*}
Observe that
\[
\e(\tilde\a_i,\g)=\e(\a,\g)+k-i,\qquad \e(\a,\tilde \g_j)=\e(\a,\g)+k,
\]
and thus
\[
\rho(\b;\tilde\a_i,\g)=(-1)^{k-i}\rho(\b;\a,\g),\qquad \rho(\b;\a,\tilde\g_j)=(-1)^{k}\rho(\b;\a,\g).
\]
Moreover, set
\begin{align*}
\bar \g_j&:=\bigwedge_{t=1}^{j-1}evi_t^*\g_t\we evi_j^*d\g_j\we\bigwedge_{t=j+1}^levi_t^*\g_t
,\\
\bar \a_i&:=
\bigwedge_{t=1}^{i-1}evb_t^*\a_t\we evb_i^*d\a_i\we \bigwedge_{t=i+1}^kevb_t^*\a_t.
\end{align*}
Then
\[
d\xi=\sum_{j=1}^l(-1)^{|\g^{<j}|}\bar\g_j\we\bar\a+\sum_{i=1}^k(-1)^{|\g|+|\a^{<i}|}\bar\g\we\bar\a_i.
\]
Recall that $\deg q=2-k$. Therefore,
\begin{align*}
(-1)^{|\qor|}\rho(\b;\a,\g)(evb_0^\b)_*\qor_*d\xi&=\rho(\b;\a,\g)\sum_{\begin{smallmatrix}S_3[l]\\(2:3)=(j)\end{smallmatrix}}(-1)^{k+|\g^{(1:3)}|}(evb_0^\b)_*\qor_*\left(\bar\g_j\we\bar\a\right)+\\
&+\rho(\b;\a,\g)\sum_{\begin{smallmatrix}S_3[k]\\(2:3)=(i)\end{smallmatrix}}(-1)^{k+|\g|+|\a^{(1:3)}|}(evb_0^\b)_*\qor_*\left(\bar\g\we\bar\a_i\right)\\
&=\sum_{\begin{smallmatrix}S_3[l]\\(2:3)=(j)\end{smallmatrix}}(-1)^{|\g^{(1:3)}|}\oqb{3}\left(\g^{(1:3)}, d\g_j,\g^{(3:3)};\a\right)+\\
&+\sum_{\begin{smallmatrix}S_3[k]\\(2:3)=(i)\end{smallmatrix}}(-1)^{i+|\g|+|\a^{(1:3)}|}\oqb{3}\left(\g;\a^{(1:3)}, d\a_i,\a^{(3:3)}\right).
\end{align*}
Let $P_i:=(1,i-1)\circ (i)\circ (i+1,...,k)$.
Then the last sum is
\[
\sum_{i\leq k}(-1)^{1+\iota(\a,\g;P_i,[l])}\oqb{3}\left(\g;\a^{(1:3)},\oqb{0}(\a_i),\a^{(3:3)}\right).
\]
Rearranging the above results, we reach the conclusion of Proposition~\ref{q-relations}.
\end{proof}
\subsection{Proof for   \texorpdfstring{$k=-1$}{TEXT}}
\label{proof of q-relations k=-1}
In this section, we prove Proposition~\ref{q-relations k=-1}.
We concentrate on the case $L$ is not vertically orientable and $\mu(\b)=_21$. The proof of the case $\W$ is a point and $L$ is oriented can be found in \cite[Section 2.4]{Sara1}. The generalization to the case of general $\W$ and vertically oriented $L$ is omitted.
In the case $L$ is not vertically orientable, and $\mu(\b)=_20$, all terms in Proposition~\ref{q-relations k=-1} vanish. Therefore, from now on, we assume $\mu(\b)=_21$. Since $\mu(\varpi(\hat\b))=2c_1(\hat \b)$ is even for $\hat\b\in \Pi'$, it follows that $\varpi^{-1}(\b)=\emp$. So we aim to prove the following equation.
\begin{multline}
\label{q-relations k=-1 no spheres}
-d\oqb{-3}(\g)=\sum_{\begin{smallmatrix}S_3[l]\\(2:3)=\{j\}\end{smallmatrix}}(-1)^{|\g^{(1:3)}|+1}\oqb{-3}(\g^{(1:3)}\otimes d\g_j\otimes \g^{(3:3)})
\\+\frac{1}{2}
\sum_{\begin{smallmatrix}\b_1+\b_2=\b\\I\dot\cup J= [l]\end{smallmatrix}}(-1)^{\iota(\g;I)}\left\langle\fq_{0,I}^{\b_1}(\g^{I}),\fq_{0,J}^{\b_2}(\g^{J})\right\rangle
\end{multline}
\begin{lemma}\label{rho lemma k=-1}
The following equation holds.
\[\rho(\b,\g)=(-1)^{n|\g|}\rho(\b_1;\emp,\g^I)\rho(\b_2;\emp,\g^J)\]
\end{lemma}
\begin{proof}
Recall Definition~\ref{rho definition}. In particular,
\[\rho_c(\b_i;\emp,\g^i)=(-1)^{1+{{\left(\begin{smallmatrix}{\d\mu(\b_i)}\\ 2\end{smallmatrix}\right)}}},\qquad 
\rho_c(\b,\g)=(-1)^{n|\g|+{{\left(\begin{smallmatrix}{\d\mu(\b)}\\ 2\end{smallmatrix}\right)}}}.\] Moreover, the assumption $\mu(\b)\equiv_21$ implies that one of $\mu(\b_1),\mu(\b_2)$ is even, and the other is odd. By equation~\eqref{sum choose 2 equation},
\[
{{\begin{pmatrix}{\d\mu(\b)}\\ 2\end{pmatrix}}}=_2{{\begin{pmatrix}{\d\mu(\b_1)}\\ 2\end{pmatrix}}}+{{\begin{pmatrix}{\d\mu(\b_2)}\\ 2\end{pmatrix}}}.
\]
Therefore,
\[
\rho_c(\b_1;\emp,\g^I)\rho_c(\b_2;\emp,\g^J)=(-1)^{{{\left(\begin{smallmatrix}{\d\mu(\b)}\\ 2\end{smallmatrix}\right)}}}=(-1)^{n|\g|}\rho_c(\b,\g).
\]
Similarly, since exactly one of $\b_1,\b_2$ is odd, we get
\[
\rho_i(\b_1;\emp,\g^I)\rho_i(\b_2;\emp,\g^J)=\sqrt{-1}=(-1)^{n|\g|}\rho_i(\b;\g).
\]
\end{proof}

\begin{proof}[Proof of Proposition~\ref{q-relations k=-1}]\label{q-relations proof calculation k=-1}
Set $\xi=\bigwedge_{j=1}^l evi_j^*\g_j$. We use Stoke's Theorem~\ref{Stoke's theorem} on equation~\eqref{oq in terms of q equations k=-1} to calculate
\begin{equation}
d\oqb{-3}(\g)=\rho(\b;\g){\pi^{\pa^v \mM_{0,l}(\b)}}_*\left(\pa \qor_{-1,l}^\b\right)_*\iota^*\xi+(-1)^{|\qor_{-1}|}\rho(\b;\g){\pi^{\mM_{0,l}(\b)}}_*\left(\qor_{-1,l}^\b\right)_*d\xi.
\label{stokes in proof q-relations k=-1}
\end{equation}
First, we analyze the contribution
of the boundary of $\mM$ in equation~\eqref{stokes in proof q-relations k=-1}. Since we are assuming $\mu(\b)$ is odd, the boundary of a disk of degree $\beta$ cannot collapse to a point, thus the boundary of $\mM$ is composed of boundary components $B:=~B_{I,J}(\b_1,\b_2)$.
On each boundary component $B$, we can apply Proposition~\ref{boundary of q k=-1 theorem}, as follows. Fix $I\dot\cup J=[l]$ and $\b_1+\b_2=\b\in \Pi$. Set
\begin{align*}
\bar\xi:=\vartheta^*\iota^*\xi,\qquad
\xi_1:=\bigwedge_{j\in I}\left(evi_{j}^{\b_1}\right)^*\g_j,\qquad
\xi_2:=\bigwedge_{j\in J}\left(evi_{j}^{\b_2}\right)^*\g_j.
\end{align*}
It holds that
\begin{equation}
\bar\xi=(-1)^{s_1}p_1^*\xi_1\we p_2^*\xi_2,    
\end{equation}
with $s_1=sgn\left(\s^\g_{I,J}\right).$
By Proposition~\ref{pullbcak inverse pushforward orientors} applied to $f=\vartheta$ and $g=pt_B$, we have
\[
{\pi^{B}}_*\left(\pa  \qor_{-1,l}^\b\right)_*\,\iota^*(\xi)={\pi^{\mM_1\times_L\mM_2}}_*\left(\expinv{\vartheta}\pa  \qor_{-1,l}^\b\right)_*\,\bar\xi.
\]

We calculate
\begin{align*}
\rho(\b;\g)&{\pi^{\mM_1\times_L\mM_2}}_*\left(\expinv{\vartheta}\pa  \qor_{-1,l}^\b\right)_*\,\vartheta^*\iota^*(\xi)=
\\\overset{\text{Prop.~\ref{boundary of q k=-1 theorem}}}=&(-1)^{s_1+1}\rho(\b;\g){\pi^L}_*{evb_0^{\b_1}}_*{p_1}_*\left(\Otm \bu m\bu
\left(\qor_{0,I}^{\b_1}\right)^{\rort}\bu \expinv{\left(p_2/evb_0^{\b_1}\right)}\qor_{0,J}^{\b_2}\right)_*\left(p_1^*\xi_1\we p_2^*\xi_2\right)\\
\overset{\text{Prop.~\ref{Fubini prop}}}=&(-1)^{s_1+1}\rho(\b;\g){\pi^L}_*(O\bu m)_* {evb_0^{\b_1}}_*\left(\qor_{0,I}^{\b_1}\right)^{\rort}_*{p_1}_*\left(\expinv{\left(p_2/evb_0^{\b_1}\right)}\qor_{0,J}^{\b_2}\right)_*\left(p_1^*\xi_1\we p_2^*\xi_2\right)\\   \overset{\text{Prop.~\ref{Module-like behavior  proposition}}}=&(-1)^{s_1+1}\rho(\b;\g){\pi^L}_*(O\bu m)_* {evb_0^{\b_1}}_*\left(\qor_{0,I}^{\b_1}\right)^{\rort}_*\left(\xi_1\we {p_1}_*\left(\expinv{\left(p_2/evb_0^{\b_1}\right)}\qor_{0,J}^{\b_2}\right)_*p_2^*\xi_2\right)\\
\overset{\text{Prop.~\ref{base change orientors proposition}}}=&(-1)^{s_1+1}\rho(\b;\g){\pi^L}_*(O\bu m)_* {evb_0^{\b_1}}_*\left(\qor_{0,I}^{\b_1}\right)^{\rort}_*\left(\xi_1\we {evb_0^{\b_1}}^*{evb_0^{\b_2}}_*{\qor_{0,J}^{\b_2}}_*\xi_2\right)\\
\overset{\text{Prop.~\ref{Module-like behavior  proposition}}}=&(-1)^{s_1+1}\rho(\b;\g){\pi^L}_*(O\bu m)_*\left({evb_0^{\b_1}}_*{\qor_{0,I}^{\b_1}}_*\xi_1\we {evb_0^{\b_2}}_*{\qor_{0,J}^{\b_2}}_*\xi_2\right)\\
\overset{\text{Lem.~\ref{rho lemma k=-1}}}=&(-1)^{s_1+1+n|\g|}{\pi^L}_*(O\bu m)_*\left(\fq_{0,I}^{\b_1}(\g^I)\we \fq_{0,J}^{\b_2}(\g^J)\right)\\
\overset{\text{Def.~\ref{Poincare pairing orientors}}}=&(-1)^{s_1+1+|\g^I|}\left\langle \fq_{0,I}^{\b_1}(\g^I),\fq_{0,J}^{\b_2}(\g^J)\right\rangle.
\end{align*}

We turn to analyze the contribution of the term with $d\xi$ in equation~\eqref{stokes in proof q-relations k=-1}. 
For a partition $P\in S_3[l]$ with $(2:3)=\{j\}$, set
\[
\tilde \g_P=\left(\g^{(1:3)},d\g_j,\g^{(3:3)}\right).
\]
Observe that
\[
\e(\tilde\g_P)=n+\e(\g), 
\]
and therefore
\[
\rho(\b;\tilde\g_P)=(-1)^n\rho(\b;\g).
\]
Moreover, set
\[
\bar \g_P:=\bigwedge_{t=1}^{j-1}evi_t^*\g_t\we evi_j^*d\g_j\we\bigwedge_{t=j+1}^levi_t^*\g_t.
\]
Then
\[
d\xi=\sum_{\begin{smallmatrix}P\in S_3[l]\\(2:3)=\{j\}\end{smallmatrix}}(-1)^{|\g^{(1:3)}|}\bar\g_P.
\]
Recall that $\deg \qor_{-1,l}^\b=4-n-2l=n\mod 2$. Therefore,
\begin{align*}
(-1)^{|\qor_{-1}|}\rho(\b;\g){\pi^{\mM_{0,l}(\b)}}_*\left(\qor_{-1,l}^\b\right)_*(d\xi)&=\rho(\b;\g)\sum_{\begin{smallmatrix}P\in S_3[l]\\(2:3)=(j)\end{smallmatrix}}(-1)^{n+|\g^{(1:3)}|}{\pi^{\mM_{0,l}(\b)}}_*\left(\qor_{-1,l}^\b\right)_*\bar\g_P\\
&=\sum_{\begin{smallmatrix}P\in S_3[l]\\(2:3)=(j)\end{smallmatrix}}(-1)^{|\g^{(1:3)}|}\rho(\b;\tilde\g_P){\pi^{\mM_{0,l}(\b)}}_*\left(\qor_{-1,l}^\b\right)_*\bar\g_P\\
&=\sum_{\begin{smallmatrix}P\in S_3[l]\\(2:3)=(j)\end{smallmatrix}}(-1)^{|\g^{(1:3)}|}\oqb{-3}\left(\g^{(1:3)},d\g_j,\g^{(3:3)}\right).
\end{align*}
Rearranging, we obtain the proposition. The term $\frac12$ appears since the sum counts each component $B=B_{I,J}(\b_1,\b_2)=B_{J,I}(\b_2,\b_1)$ twice, while
\[
(-1)^{|\g^I|+sgn(\s_{I,J}^\g)}\left\langle \fq_{0,I}^{\b_1}(\g^I),\fq_{0,J}^{\b_2}(\g^J)\right\rangle=(-1)^{|\g^J|+sgn(\s_{J,I}^\g)}\left\langle \fq_{0,J}^{\b_2}(\g^J),\fq_{0,I}^{\b_1}(\g^I)\right\rangle.
\]
\end{proof}

\section{Properties}\label{Properties section}
Recall the definitions of $\oqb{3}, Q,R,D,C, \langle,\rangle$ from Section~\ref{operators subsection}.
\subsection{Linearity}
\begin{proposition}
\label{lineraity}
The $\fq$ operators are multilinear, in the sense that for $a\in R$ we have 
\begin{align*}
    \oqb{3}(\g_1,...,\g_l;&\a_1,...,\a_{i-1},a\cdot\a_i,...,\a_k)\\
    &=(-1)^{|a|\cdot(i+\sum_{j=1}^{i-1}|\a_j|)+\sum_{j=1}^l|\g_j|}a\cdot \oqb{3}(\g_1,...,\g_l;\a_1,...,\a_k)+\d_{1,k}\cdot da\cdot \a_1,
\end{align*}
and for $a\in Q$ we have
\[
 \oqb{3}(\g_1,...,a\cdot\g_i,...,\g_l;\a_1,...,\a_k)
    =(-1)^{|a|\cdot(\sum_{j=1}^{i-1}|\g_j|)}a\cdot \oqb{3}(\g_1,...,\g_l;\a_1,...,\a_k).
\]
In addition, the pairing $\langle,\rangle$ is $R$-bilinear in the sense of Definition~\ref{A infinity algebra definition}(2).
\end{proposition}
\begin{proof}
For $\oqb{0}=d$ we have
\[
d(a\cdot\a)=da\cdot \a+(-1)^{|a|}a\cdot d\a.
\]
For $(k,l,\b)\neq (1,0,\b_0)$, set
\[
\hat\a=(\a_1,\ldots,\a_{i-1},a\cdot\a_i,\a_{i+1},\ldots\a_k),
\]
and set
\begin{align*}
\xi=&\bigwedge_{j=1}^levi_j^*\g_j\we\bigwedge_{j=1}^{k}evb_j^*\a_j,\\
\hat\xi=&\bigwedge_{j=1}^levi_j^*\g_j\we\bigwedge_{j=1}^{i-1}evb_j^*\a_j\we evb_i^*(a\cdot\a_i)\we\bigwedge_{j=i+1}^kevb_j^*\a_j.
\end{align*}
We have
\[
\rho(\b;\hat\a,\g)=(-1)^{(k-i)\cdot |a|}\rho(\b;\a,\g),
\]
and
\[
\hat \xi=(-1)^{|a|\cdot\left(\sum_{j<i}|\a_j|+|\g|\right)}a\cdot\xi.
\]
Moreover, since $\deg \qor_{k,l}^\b=_2k$,
\[
\left(\qor_{k,l}^\b\right)_*(a\cdot\xi)=(-1)^{k\cdot|a|}a\cdot\left(\qor_{k,l}^\b\right)_*\xi.
\]
Therefore, 
\begin{align*}
    \oqb{3}(\g;\hat\a)=&(-1)^{|a|\cdot\left(i+\sum_{j<i}|\a_j|+|\g|\right)}a\cdot \oqb{3}(\g;\a).
\end{align*}
A similar calculation gives the second identity.
We turn to prove the bilinearity of the pairing. Recall that $\deg \Otm=1-n$ and $\deg m=0$. We calculate
\begin{align*}
    \langle a\cdot\xi,\eta\rangle&=(-1)^{|a|+|\xi|+n(|a|+|\xi|+|\eta|)}{\pi^L}_*\left(\Otm\bu m\right)_*(a\cdot\xi\we\eta)\\&=(-1)^{|\xi|+n(|\xi|+|\eta|)}a\cdot {\pi^L}_*\left(\Otm\bu m\right)_*(\xi,\eta)=a\cdot\langle\xi,\eta\rangle,\\
    \langle \xi,a\cdot\eta\rangle&=(-1)^{|\xi|+n(|a|+|\xi|+|\eta|)}{\pi^L}_*\left(\Otm\bu m\right)_*\left(\xi\we(a\cdot\eta)\right)\\
    &=(-1)^{|\xi|+n(|a|+|\xi|+|\eta|)+|a||\xi|}{\pi^L}_*\left(\Otm\bu m\right)_*(a\cdot \xi\we\eta)\\
    &=(-1)^{|\xi|+n(|\xi|+|\eta|)+|a|(|\xi|+1)}a\cdot {\pi^L}_*\left(\Otm\bu m\right)_*(\xi\we\eta)=(-1)^{|a|(|\xi|+1)}a\cdot \langle\xi,\eta\rangle.
\end{align*}
\end{proof}
\subsection{Pseudoisotopy}
Throughout this section, for a target $\target$, we use the superscript $\target$ to emphasize that the dependence of an object on $\target.$

Fix a target $\target=(\W,X,\w,\pi^X,L,\fp,\underline\Upsilon, J)$. Let $\xi:\W'\to\W$ be a smooth map of manifolds with corners. Set $\xi^*\target$ be the pullback target over $\W'$. Denote by 
\[
\xi_\La^*:\La^\target\to\La^{\xi^*\target}
\]
the isomorphism that sends $T^{\b}$ to $T^{\xi^*\b}$. Recall the notations $\pull\xi _\rort$ and $\pull\xi_\efield$ from Section~\ref{q-naturality families}. Define
\begin{align*}
\pull\xi_R:R^\target&\to R^{\xi^*\target},\\
\pull\xi_C:C^\target&\to C^{\xi^*\target},
\end{align*}
to be the compositions
\[
\begin{tikzcd}
A\left(\W;\efield_L\right)\otimes\La^\target \ar[r,"\xi^*\otimes1_\La"]&A\left(\W';\xi^*\efield_L\right)\otimes\La^{\target}\ar[r,"{\left(\pull\xi_{\efield}\right)_*\otimes  \xi^*_\La}"]&A\left(\W';\efield_{\xi^*L}\right)\otimes\La^{\xi^*\target},\\
A\left(L;\rort_L\right)\otimes\La^\target\ar[r,"{{\xi^L}^*\otimes1_\La}"]&A\left(\xi^*L;\xi^*\rort_L\right)\otimes\La^{\target}\ar[r,"{\left(\pull\xi_{\rort}\right)_*\otimes  \xi^*_\La}"]&A\left(\xi^*L;\rort_{\xi^*L}\right)\otimes\La^{\xi^*\target},
\end{tikzcd}
\]
respectively.
\begin{remark}\label{pull xi C is a differential graded algebra homomorphism remark}
The maps $\pull\xi_R,\pull\xi_C$ are homomorphisms of differential graded algebras.
\end{remark}
\begin{proposition}\label{naturality of q operators families}
Let $k,l,\b\in \Z_{\geq0}\times \Z_{\geq 0}\times \Pi(\target)$. Let $\a_1,...,\a_k\in C^\target$ and $\g_1,...,\g_l\in D^\target.$
Then
\[
\pull\xi_C\left(\fq_{k,l}^{(\target;\b)}(\g_1,...,\g_l;\a_1,...,\a_k)\right)=\fq_{k,l}^{\left(\xi^*\target;\xi^*\b\right)}\left({\xi^X}^*\g_1,...,{\xi^X}^*\g_l;\pull\xi_C\a_1,...,\pull\xi_C\a_k\right)
\]
and 
\[
\pull\xi_R\left(\fq_{-1,l}^{(\target;\b)}(\g_1,...,\g_l)\right)=\fq_{-1,l}^{\left(\xi^*\target;\xi^*\b\right)}\left({\xi^X}^*\g_1,...,{\xi^X}^*\g_l\right).
\]
\end{proposition}
\begin{proof}
We prove the case $k\geq0$. The proof of the case $k=-1$ is similar. The case $(k,l,\b)=(1,0,\b_0)$ is the ``differential" part of Remark~\ref{pull xi C is a differential graded algebra homomorphism remark}, and follows immediately from the definitions.

For $(k,l,\b)\neq(1,0,\b_0)$, we proceed as follows. Recall Definition~\ref{rho definition}. Clearly, 
\[
\rho\left(\xi^*\b;\pull\xi_C\a,{\xi^X}^*\g\right)=\rho(\b;\a,\g).
\]We denote the above quantity by $\rho$. Recall $E_L$ from Definition~\ref{fundamental orientors Q definition} and $\pull\xi_E$ which is defined in equation~\eqref{xi E not blackboard definition}.
Let 
\begin{align*}
\eta\in& A\left(\mM_{k+1,l}(\target;\b);E_L\right),\\ \eta'\in& A\left(\mM_{k+1,l}({\xi}^*\target;\xi^*\b);{\xi^{\mM}}^*E_L\right),\\ \overline\eta\in& A\left(\mM_{k+1,l}({\xi}^*\target;\xi^*\b);E_{\xi^*L}\right)    
\end{align*}
be given by
\begin{align*}
\eta&=\bigwedge (evi_j^\target)^*\g_j\we\bigwedge (evb_j^\target)^*\a_j,\\ 
\eta'&=\bigwedge (evi_j^{\xi^*\target})^*{\xi^X}^*\g_j\we\bigwedge (evb_j^{\xi^*\target})^*{\xi^L}^*\a_j,\\
\overline\eta&=\bigwedge (evi_j^{\xi^*\target})^*{\xi^X}^*\g_j\we\bigwedge (evb_j^{\xi^*\target})^*\pull\xi_C\a_j.
\end{align*}
Then $\eta'={\xi^\mM}^*\eta,$ and $\overline\eta=\pull\xi_E\eta'.$ In the following, we suppress $\xi_\La^*$ in the expression of $\pull\xi_C$ to simplify notation. We calculate,
\begin{align*}
\pull\xi_C\left(\fq_{k,l}^{(\target;\b)}(\g_1,...,\g_l;\a_1,...,\a_k)\right)
\overset{\text{Def.~\ref{def:oqb}}}
=&\rho\cdot\left(\pull\xi_\rort\right)_*{\xi^L}^*\left(\left({evb^\target_0}\right)_*\left(\qor_{k,l}^{(\target;\b)}\right)_*\eta\right)\\
\overset{\text{Prop.~\ref{base change orientors proposition}}}=&\rho\cdot\left(\pull\xi_\rort\right)_*(evb_0^{\xi^*\target})_*\left(\expinv\xi \qor_{k,l}^{(\target;\b)}\right)_*{\xi^\mM}^*\eta\\
\overset{\text{Fubini~\ref{Fubini prop}}}=&
\rho\cdot(evb_0^{\xi^*\target})_*\left(\pull\xi_\rort\bu\expinv\xi \qor_{k,l}^{(\target;\b)}\right)_*\eta'\\
\overset{\text{Thm~\ref{naturality of Q families}}}=&
\rho\cdot(evb_0^{\xi^*\target})_*\left( \qor_{k,l}^{(\xi^*\target;\xi^*\b)}\bu\pull\xi_E\right)_*\eta'\\
\overset{\text{Fubini~\ref{Fubini prop}}}=&
\rho\cdot(evb_0^{\xi^*\target})_*\left( \qor_{k,l}^{(\xi^*\target;\xi^*\b)}\right)_*\overline\eta\\\overset{\text{Def.~\ref{def:oqb}}}=&\fq_{k,l}^{\left(\xi^*\target;\xi^*\b\right)}\left({\xi^X}^*\g_1,...,{\xi^X}^*\g_l;\pull\xi_C\a_1,...,\pull\xi_C\a_k\right).
\end{align*}
\end{proof}
\begin{proposition}\label{naturality of pairing families}
Let $\a_1,\a_2\in C^\target$. Then
\[
\pull\xi_C\langle\a_1,\a_2\rangle^{\target}=\left\langle \pull\xi_C\a_1,\pull\xi_C\a_2\right\rangle^{\xi^*\target}.
\]
\end{proposition}
\begin{proof}
The proof is omitted. It is parallel to the proof of Proposition~\ref{naturality of q operators families} and uses Proposition~\ref{naturality of Otm families} and Remark~\ref{naturality of m families}.
\end{proof}

Let $\xi_0, \xi_1:\W'\to \W$ be smooth maps of manifolds with corners, and let $H:\W'\times[0,1]\to \W$ be a homotopy between $H(\cdot,0)=\xi_0$ and $H(\cdot,1)=\xi_1$. Denote by $\pi:\W'\times[0,1]\to \W'$ the projection. 
Set $\xi_t=H(\cdot,t)$.
We write \begin{align*}
    L_t:=&\xi_t^*L,\\
    \efield_t:=&\efield_{\xi_t^*L},\\
    R_t:=&R^{\xi_t^*\target},\\
    C_t:=&C^{\xi_t^*\target},\\
    \langle,\rangle_{t}:=&\langle,\rangle^{\xi_t^*\target}.
\end{align*}
Moreover, we denote by $C:=C^\target$. The homotopy $H$ induces diffeomorphisms $L_t\to L_0$ which in turn induce isomorphisms $I_t:\efield_t\to \efield_0$ of local systems over $\W'$.

\begin{proposition}\label{chain homotopy R proposition}
With the above notations, the following equation holds.
\[
{I_1}_*\pull{\xi_1}_R-\pull{\xi_0}_R=d\left(\pi_*{I_t}_*\pull H_R\right)+\pi_*{I_t}_*\pull H_R d.
\]
In particular, ${I_1}_*\pull{\xi_1}_R$ and $\pull{\xi_0}_R$ are cochain homotopic as maps $R\to R_0$.
\end{proposition}\begin{proof}
Denote by $\iota:\W'\times\{0,1\}\to \W'\times[0,1]$ the inclusion of the vertical boundary of $\pi$.
We calculate,
\begin{align*}
    {I_1}_*\pull{\xi_1}_R-\pull{\xi_0}_R=&(\pi\circ \iota)_*{I_t}_*\pull H_R\\
    \overset{\text{Stokes~\ref{Stoke's theorem}}}=&
    d\left(\pi_*{I_t}_*\pull H_R\right)+\pi_*{I_t}_*\pull H_R d.
\end{align*}
\end{proof}
\begin{corollary}
With the above notations, consider the following diagram of differential graded algebras.
\[
\begin{tikzcd}
C\otimes C\ar[d,swap,"\pull{\xi_1}_C\otimes \pull{\xi_1}_C"]\ar[rr,"\pull{\xi_0}_C\otimes \pull{\xi_0}_C"]&& C_0\otimes C_0\ar[d,"\langle;\rangle_0"]\\
C_1\otimes C_1\ar[r,swap,"\langle;\rangle_1"]& R_1\ar[r,swap,"{I_1}_*"]&R_0
\end{tikzcd}
\]
The composition of the left arrow and the lower arrows is chain homotopic to the composition of the upper arrow and the right arrow. More specifically, denoting by $\pi:\W\times [0,1]\to \W$ the projection, for $\a,\b\in C$,
\[
{I_1}_*\left\langle\pull{\xi_1}_C\a,\pull{\xi_1}_C\b\right\rangle_1-\left\langle\pull{\xi_0}_C\a,\pull{\xi_0}_C\b\right\rangle_0=d\left(\pi_*{I_t}_*\pull H_R\left\langle\a,\b\right\rangle\right)+\pi_*{I_t}_*\pull H_R d\left\langle\a,\b\right\rangle
\]
\end{corollary}
\begin{remark}
The fact that the above equation represents a chain homotopy follows from equation~\eqref{pairing is a chain map equation} which is proved later.
\end{remark}
\begin{proof}
Denote by $\iota:\W'\times\{0,1\}\to \W'\times[0,1]$ the inclusion of the vertical boundary of $\pi$.
We calculate,
\begin{align*}
    {I_1}_*\left\langle\pull{\xi_1}_C\a,\pull{\xi_1}_C\b\right\rangle_1-\left\langle\pull{\xi_0}_C\a,\pull{\xi_0}_C\b\right\rangle_0\overset{\text{Prop.~\ref{naturality of pairing families}}}=&{I_1}_*\pull{\xi_1}_R\left\langle\a,\b\right\rangle_1-\pull{\xi_0}_R\left\langle\a,\b\right\rangle_0\\=&(\pi\circ \iota)_*\left({I_t}_*\pull H_R\left\langle\a,\b\right\rangle_t\right)\\
    \overset{\text{Stokes~\ref{Stoke's theorem}}}=&
    d\left(\pi_*{I_t}_*\pull H_R\left\langle\a,\b\right\rangle\right)+\pi_*{I_t}_*\pull H_R d\left\langle\a,\b\right\rangle.
\end{align*}
\end{proof}
\begin{example}
Recall Example~\ref{rotating circle example}. Denote by $\mathbf{z}\subset S^2$ be the unit circle in the $xy$ coordinate plane. $\mathbf z$ is the fiber of the $z-$axis over $\W=\R P^1$. Let $\W'=\{\mathbf z\}$ and $\xi_0,\xi_1:\W'\to \W$ be the inclusion. Let $H:[0,1]\to \R P^1$ be the half roundtrip homotopy. Then $\efield_t\simeq\F[x]$. Denote by $I_t^i:\efield_t\to \efield_0$ the identification with respect to $H^i$, for $i=0,1$. It follows that $I_1^0=\Id_{\F[x]}$ and $I_1^1(x)=-x$. 
\end{example}
\begin{remark}
It is standard practice to conclude that homotopies of homotopies may provide cochain homotopies between the cochain homotopies provided by the above proposition.
\end{remark}

\subsection{Unit of the algebra}
\begin{proposition}
\label{unit of the algebra}
Fix $f\in A^0(L;\rort_L)\otimes \La[[t_1,...,t_N]],\a_1,...,\a_{k-1}\in C$ and $\g_1,...,\g_l\in A^*(X;Q)$. Then,
\[
\oqb{3}(\g;\a_1,...,\a_{i-1},f,\a_i,...,\a_{k-1})=\begin{cases}
df,& (k,l,\b)=(1,0,\b_0),\\
(-1)^{|f|}f\cdot \a_1,&(k,l,\b)=(2,0,\b_0),i=1,\\
(-1)^{|\a|}\a_1\cdot f,&(k,l,\b)=(2,0,\b_0),i=2,\\
0,&otherwise.
\end{cases}
\]
In particular, $1\in A^0(L)$ is a strong unit for the $A_\infty$ operations $\fm^\g$:
\[
\om{3}(\a_1,...,\a_{i-1},1,\a_i,...,\a_{k-1})=\begin{cases}
0,&k\geq 3\text{ or }k=1,\\
\a_1,&k=2,i=1,\\
(-1)^{|\a_1|}\a_1,&k=2,i=2.
\end{cases}
\]
\end{proposition}
\begin{proof}
The case $(k,l,\b)=(1,0,\b_0)$ is true by definition. We proceed with the proof for the other values of $(k,l,\b)$.

Let \[(k,l,\b)\in \Z_\geq1\times l\geq0\times\Pi\setminus\{(1,0,\b_0),(2,0,\b_0)\}.\] Let $i\geq k$. We show that
\[
\fq_{k+1,l}^\b(\g;\a_1,...,\a_{i-1},f,\a_i,...,\a_k)=0
\]
for all $\a_1,...,\a_k\in C$ and $f\in A^0(L;\rort_L)\otimes \La[[t_1,...,t_N]]$. We assume $i=k+1$ for simplicity. Recall the map $Fb:=Fb_{k+1,l}^\b:\mM_{k+2,l}(\b)\to \mM_{k+1,l}(\b)$ that forgets the $k+1$st point, and its orientation $\mO^{Fb}$. See Section~\ref{Orientors over the moduli spaces section}.
Denote by $evb_j^{k+1}$ and $evi_j^{k+1}$ (resp. $evb_j^k$ and $evi_j^k$) the evaluation maps for $\mM_{k+2,l}(\b)$ (resp. $\mM_{k+1,l}(\b)$). Set
\[
\xi=\bigwedge_{j=1}^l\left(evi^{k+1}_j\right)^*\g_j\we\bigwedge_{j=1}^{k}\left(evb^{k+1}_j\right)^*\a_j\we \left(evb_{k+1}^{k+1}\right)^*f.
\]
Note that \[
evi_j^{k+1}=evi_j^k\circ Fb\qquad\text{and}\qquad evb_j^{k+1}=
evb_j^k\circ Fb,\quad j\leq k
\]
Thus, writing $g=(evb_{k+1}^{k+1})^*f$, we have
\[
\fq_{k+1,l}^\b(\g;\a_1,...,\a_k,f)=\pm\rho(\b;\a,\g)(evb_0^{k+1})_*\left(\qor^\b_{k+1,l}\right)_*(Fb^*\xi\we g).
\]
The following equation holds in the sense of currents,
\begin{align*}
(evb_0^{k+1})_*&\left(\qor^\b_{k+1,l}\right)_*(Fb^*\xi\we g)\\\overset{\text{Prop~\ref{factorization q through forgetful theorem}}}=&\pm\left(evb_0^k\circ Fb\right)_*\left(m\bu\left(\expinv{\left(Fb,\mO^{Fb}\right)}\qor_{k,l}^\b\right)^\rort\bu{}^{Fb^*E^k}(c_{k+2,k+1})\right)_*(Fb^*\xi\we g)\\\overset{\begin{smallmatrix}\text{Prop. ~\ref{Fubini prop}}\\\text{Prop.~\ref{pushforward by phi is minus the oriented pushforward}}\\\text{Prop.~\ref{Module-like behavior  proposition}}\end{smallmatrix}}=&
\pm m_*{evb_0^k}_*{\qor_{k,l}^\b}^\rort_*\left(\xi\we \left(Fb,\mO^{Fb}\right)_*\left(c_{k+2,k+1}\right)_*g\right).
\end{align*}
However, ${Fb,\mO^{Fb}}_*\left(c_{k+2,k+1}\right)_*g=0$ since $\dim {Fb}=1$ and the form-degree of $g$ is zero.

When $(k,l,\b)=(2,0,\b_0)$, the map $evb_0:\mM_{3,0}(\b_0)\to L$ is a diffeomorphism. 
By Proposition~\ref{energy zero q theorem}, we have $\qor_{2,0}^{\b_0}=\left(\phi_{evb_0}\right)^{\rort_L}\bu m$. Therefore,
\[
\fq^{\b_0}_{2,0}(f,\a)=(-1)^{|f|}\left(evb_0\right)_*\left(\phi_{evb_0}\right)_*m_*\left(evb_0\right)^*(f\we \a)=(-1)^{|f|}f \a,
\]
and
\[
\fq^{\b_0}_{2,0}(\a,f)=(-1)^{|\a|}\left(evb_0\right)_*\left(\phi_{evb_0}\right)_*m_*\left(evb_0\right)^*(\a\we f)=(-1)^{|a|} \a f.
\]
\end{proof}
\subsection{Fundamental class}
\begin{proposition}
\label{fundamental class proposition}
For $k\geq -1$, 
\[
\oqb{3}(\a;1,\g_1,...,\g_{l-1})=\begin{cases}
-1,&(k,l,\b)=(0,1,\b_0),\\
0,&\text{otherwise}.
\end{cases}
\]
\end{proposition}
\begin{proof}
The proof is similar to that of the previous section and of \cite[Proposition 3.7]{Sara1}.
\end{proof}
\subsection{Cyclic structure}
\label{cyclic}  
\label{oddpairing definition}
\begin{proposition}\label{antisymmetry of pairing}
For any $\xi,\eta\in C$,
\[
\oddpairing{\xi,\eta} = (-1)^{(1+|\xi|)(1+|\eta|)+1}\oddpairing{\eta,\xi}.
\]
\end{proposition}
\begin{proof}
This follows immediately from Remark~\ref{alternative definition of odd pairing remark}.
\end{proof}
\begin{proposition}\label{cyclic structure proposition}
For $\a_1,...,\a_{k+1}\in C$ and $\g_1,...,\g_l\in D$,
\begin{multline*}
    \oddpairing{\oq{3}(\g;\a_1,...,\a_k),\a_{k+1}}=\\(-1)^{(|\a_{k+1}|+1)\cdot\sum_{j=1}^{k}(|\a_j|+1)}\oddpairing{\oq{3}(\g;\a_{k+1},\a_1,...,\a_{k-1}),\a_k}+\d_{1,k}\cdot d\oddpairing{\a_1,\a_2}.
\end{multline*}
In particular, 
\begin{equation}
\label{pairing is a chain map equation}
\oddpairing{d\xi,\eta}=d\oddpairing{\xi,\eta}+(-1)^{(1+|\xi|)(1+|\eta|)}\oddpairing {d\eta,\xi}.
\end{equation}
Equation~\eqref{pairing is a chain map equation} holds for the pairing $\langle,\rangle$ as well.
\end{proposition}
\begin{proof}
We prove an appropriate result for each $(k,l,\b)$. The contribution $\d_{1,k}\cdot d\oddpairing{\a_1,\a_2}$ comes from the case $(k,l,\b)=(1,0,\b_0)$. 
For this case, since $L$ is vertically closed, we see
\begin{align*}
\oddpairing{d\xi,\eta}=\quad&(-1)^{|\xi|+1+n(|\xi|+1+|\eta|)}{\pi^L}_*\left(\Otm_{odd}\bu m\right)_*\left(d\xi\we\eta\right)\\
=\quad&(-1)^{|\xi|+1+n(|\xi|+1+|\eta|)}{\pi^L}_*\left(\Otm_{odd}\bu m\right)_*d(\xi\we\eta)+\\&+(-1)^{n(|\xi|+1+|\eta|)}{\pi^L}_*\left(\Otm_{odd}\bu m\right)_*\xi\we d\eta\\
\overset{\text{Prop.~\ref{Relative Stoke's}}}=&(-1)^{|\xi|+n(|\xi|+|\eta|)}d\left(\pi^L_*(\Otm_{odd}\bu m)_*(\xi\we\eta)\right)+(-1)^{|\xi|}\oddpairing{\xi,d\eta}\\
\overset{\text{Prop.~\ref{antisymmetry of pairing}}}=&d\oddpairing{\xi,\eta}+(-1)^{(1+|\xi|)(1+|\eta|)}\oddpairing{ d\eta,\xi}.
\end{align*}
The subscript $odd$ could be removed from the above calculation to get the same equation for $\langle,\rangle$.
For $(k,l,\b)\neq (1,0,\b_0)$ we proceed as follows.
Recall \[f:\mdl{3}\to\mdl{3}\] the map be given by 
\[
f(t,\S,u,(z_0,...,z_k),\vec w)=(t, \S,u,(z_1,...z_k,z_0),\vec w).
\]
So,
\[
evi_j\circ f =evi_j,\quad evb_k\circ  f =evb_0,\quad evb_j\circ f =evb_{j+1},\,\,j=0,...,k-1.
\]
Let
\begin{align*}
\tau: {\rort_L}^{\boxtimes k+1}&\to{\rort_L}^{\boxtimes k+1} \\
a_0\otimes\cdots\otimes a_k&\mapsto(-1)^{|a_0|\cdot\sum_{j=1}^k|\a_j|}a_1\otimes\cdots a_{k-1}\otimes a_0
\end{align*}
denote the graded symmetry isomorphism.
Set 
\begin{align*}
\a:=(\a_1,...,\a_k),\qquad&\xi:=\bigwedge_{j=1}^levi_j^*\g_j\we\bigwedge_{j=1}^{k}evb_j^*\a_j,\\
\tilde\a:=(\a_{k+1},\a_1,...,\a_{k-1}),\qquad &\tilde \xi:=\bigwedge_{j=1}^levi_j^*\g_j\we evb_{1}^*\a_{k+1}\we\bigwedge_{j=1}^{k-1}evb_{j+1}^*\a_j,\\
&\hat\xi:=\bigwedge_{j=1}^levi_j^*\g_j\we\bigwedge_{j=1}^{k-1}evb_{j+1}^*\a_j\we evb_{0}^*\a_k.
\end{align*}
Then
\[
\rho(\b;\tilde \a,\g)=(-1)^{(k+1)|\a_{k+1}|+\sum_{j=1}^{k-1}|\a_j|}\rho(\b;\a,\g).
\]
and
\begin{align}
\label{pullback of alpha by cycle diff}
f ^*\left(\xi\we evb_0^*\a_{k+1}\right)=\hat\xi\we evb_1^*\a_{k+1}\overset{\text{Prop.~\ref{pushforward of symmetry Proposition}}}=(-1)^{|\a_{k+1}|\cdot\sum_{j=1}^k|\a_j|}\t_*\left(\tilde\xi\we evb_0^*\a_k\right).
\end{align}
Denote by
\begin{align*}
S&=k+|\g|+|\a|+n(k+|\g|+|\a|+|\a_{k+1}|),\\\tilde S&=k+|\g|+|\tilde\a|+n(k+|\g|+|\tilde \a|+|\a_{k}|).
\end{align*}
Then
\[
\tilde S-S=|\a_k|-|\a_{k+1}|.
\]
We calculate
\begin{align*}
&\oddpairing{ \oqb{3}(\g;\a_1,...,\a_k),\a_{k+1}}=\\
&\overset{\begin{smallmatrix}\text{Def.~\ref{Poincare pairing orientors}}\\\text{eq. \eqref{oq in terms of q equations}}\end{smallmatrix}}=(-1)^S\rho(\b;\a;\g){\pi^L}_*\left(\Otm\bu m\right)_*\left[\left(\left(evb_0\right)_*\left(\qor_{k,l}^\b\right)_*(\xi)\right)\we \a_{k+1}\right]\\
&\overset{\text{Prop.~\ref{Module-like behavior  proposition}}}=(-1)^S\rho(\b;\a;\g){\pi^L}_*\left(\Otm\bu m\right)_*\left(evb_0\right)_*\left(\qor_{k,l}^\b\otimes \Id\right)_*\left(\xi\we evb_0^*\a_{k+1}\right)
\\&\overset{\text{Prop.~\ref{Fubini prop}}}=(-1)^S\rho(\b;\a;\g)\left(\pi^{\mM_{k+1,l}(\b)}\right)_*\left(\Otm\bu m\bu \left(\qor_{k,l}^\b\otimes \Id\right)\right)_*\left(\xi\we evb_0^*\a_{k+1}\right),
\end{align*}
and
\begin{align*}
\left(\pi^{\mM_{k+1,l}(\b)}\right)_*&\left(\Otm\bu m\bu \left(\qor_{k,l}^\b\otimes \Id\right)\right)_*\left(\xi\we evb_0^*\a_{k+1}\right)=\\
\overset{\text{Prop.~\ref{pullbcak inverse pushforward orientors}}}=&\left(\pi^{\mM_{k+1,l}(\b)}\circ f\right)_*\expinv{f}\left(\Otm\bu m\bu \left(\qor_{k,l}^\b\otimes \Id\right)\right)_*f^*\left(\xi\we evb_0^*\a_{k+1}\right)\\
\overset{\text{eq.~\eqref{pullback of alpha by cycle diff}}}=&(-1)^{|\a_{k+1}|\cdot\sum_{j=1}^k|\a_j|}\left(\pi^{\mM_{k+1,l}(\b)}\right)_*\expinv{f}\left(\Otm\bu m\bu \left(\qor_{k,l}^\b\otimes \Id\right)\right)_*\t_*\left(\tilde\xi\we evb_0^*\a_{k}\right)\\
\overset{\text{Prop.~\ref{cyclic theorem}}}=&(-1)^{k+|\a_{k+1}|\cdot\sum_{j=1}^k|\a_j|}\left(\pi^{\mM_{k+1,l}(\b)}\right)_*\left(\Otm\bu m\bu\left(\qor_{k,l}^\b\otimes \Id\right)\right)_*\left(\tilde\xi\we evb_0^*\a_{k}\right).
\end{align*}
We conclude that
\begin{align*}
\oddpairing{ \oqb{3}(\g;\a_1,...,\a_k),\a_{k+1}}=(-1)^{T}\oddpairing{ \oqb{3}(\g;\a_{k+1},\a_1,...,\a_{k-1}),\a_{k}},
\end{align*}
with
\begin{align*}
(-1)^T\rho(\b;\a,\g)=(-1)^{S-\tilde S+k+|\a_{k+1}|\cdot\sum_{j=1}^k|\a_j|}\rho(\b;\tilde\a,\g).
\end{align*}
Therefore,
\begin{align*}
T&=_2|\a_k|+|\a_{k+1}|+(k+1)|\a_{k+1}|+\sum_{j=1}^{k-1}|\a_j|+k+|\a_{k+1}|\cdot\sum_{j=1}^k|\a_j|\\
&=_2k|\a_{k+1}|+k+(|\a_{k+1}|+1)\cdot\sum_{j=1}^{k}|\a_j|\\
&=_2(|\a_{k+1}|+1)\cdot\sum_{j=1}^{k}(|\a_j|+1).
\end{align*}
This concludes the proof.
\end{proof}
\subsection{Symmetry}
\begin{proposition}
Let $k\geq -1$. For any permutation $\s\in S_l$, 
\[
\oqb{3}(\a_1,...,\a_k;\g_1,...,\g_l)=(-1)^{s_\s(\g)}\oqb{3}(\a_1,...,\a_k;\g_{\s(1)},...,\g_{\s(l)}),
\]
where
\[
s_\s(\g):=\sum_{\begin{smallmatrix}i<j\\\s^{-1}(i)>\s^{-1}(j)\end{smallmatrix}}|\g_i|\cdot|\g_j|=\sum_{\begin{smallmatrix}i>j\\\s(i)<\s(j)\end{smallmatrix}}|\g_{\s(i)}|\cdot|\g_{\s(j)}|
\]
\end{proposition}
\begin{proof}
The proof is similar to that of \cite[Proposition 3.6]{Sara1}. The proof relies on a map $ f_{\sigma}:\mM_{k+1,l}(\b)\to \mM_{k+1,l}(\b)$ given by reordering the interior points similar to $f$ from the proof of Proposition~\ref{cyclic structure proposition}. It is easier than that of the previous section, since the $\g$'s do not interact with the orientors.
\end{proof}
\subsection{Energy zero}
\begin{proposition}
\label{zero energy}
For $k\geq 0$, 
\[
\fq_{k,l}^{\b_0}(\a;\g)=\begin{cases}
d\a_1,&(k,l)=(1,0),\\
(-1)^{|\a_1|}\a_1\we\a_2,&(k,l)=(2,0),\\
-\g_1|_L,&(k,l)=(0,1),\\
0,&\text{otherwise}.
\end{cases}
\]
\end{proposition}
\begin{proof}
The case $(k,l)=(1,0)$ is true by definition. Otherwise, since the stable maps in $\mM_{k,l}(\b_0)$ are constant, we have
\[
evb_0=\cdots=evb_k,\qquad evi_1=\cdots=evi_l=i\circ evb_0,
\]
where $i:L\to X$ is the inclusion.
Thus, Proposition~\ref{Module-like behavior  proposition} implies
\begin{align*}
    \fq_{k,l}^{\b_0}(\a;\g)&=\rho(\b_0;\a,\g){evb_0}_*\left(\qor_{k,l}^{\b_0}\right)_*{evb_0}^*\left(\bigwedge_{j=1}^li^*\g_j\we\bigwedge_{j=1}^k\a_j\right)\\
    &=(-1)^{k(|\a|+|\g|)}\rho(\b_0;\a,\g)\left(\bigwedge_{j=1}^l\g_j|_L\we\bigwedge_{j=1}^k\a_j\right)\we \left({evb_0}_*\left(\qor_{k,l}^{\b_0}\right)_*\left(1^{\otimes k}\right)\right)
\end{align*}
However, $\dim {evb_0}=n-3+\mu(\b_0)+k+1+2l-n=k+2l-2$. Therefore, if $k+2l-2>0$, then \[{evb_0}_*\left(\qor_{k,l}^{\b_0}\right)_*\left(1\otimes \cdots \otimes 1\right)=0\] and thus $\fq_{k,l}^{\b_0}(\a;\g)=0$. In the case $k+2I=2$, the map ${evb_0}$ is a diffeomorphism.
By Proposition \ref{energy zero q theorem},
\[
\qor_{2,0}^{\b_0}=\expinv{\left(evb_0\right)}m,\qquad \qor_{0,1}^{\b_0}=\expinv{\left(evb_0\right)}1_L.
\]
Note that $\rho(\b_0;\a,\g)=(-1)^{\e(\a,\g)}$.

Then
\[
\fq_{2,0}^{\b_0}(\a_1,\a_2)=(-1)^{|\a_1|}\a_1\we\a_2,\qquad \fq_{0,1}^{\b_0}(\g_1)=-\g_1|_L.
\]
\end{proof}
\subsection{Divisors}
\begin{proposition}
\label{divisors} Assume $\g_1\in A(X,L;Q),$ and $d\g_1=0.$ Consider the map \[\int \g:\underline{H_2}(X,L;\Z)\to R\] given by $\b\mapsto \int_\b\g_1$, where the integral is performed over each $t\in\W$ separately. Assume $\int\g$ descends to $\Pi$. Then
\[
\fq_{k,l}^\b\left(\bigotimes_{j=1}^l\g_j;\bigotimes_{j=1}^k\a_j\right)=\left(\int_{\b}\g_1\right)\cdot\fq_{k,l-1}^\b\left(\bigotimes_{j=2}^l\g_j;\bigotimes_{j=1}^k\a_j\right)
\] for $k\geq-1$.
\end{proposition}
The proof requires the following result.

\begin{lemma}
Suppose $(k,l,\b)\notin\{(0,1,\b_0),(1,1,\b_0),(-1,2,\b_0)\}$. Recall the map
\[
Fi:\mM_{k+1,l}(\b)\to \mM_{k+1,l-1}(\b)
\]
that forgets the $l$th interior point and recall its orientation $\mO^{Fi}$. Denote by $evi_1$ the evaluation map at the first interior point for $\mM_{k+1,l}(\b)$. Let $\g\in A^*(X)$ be such that $\g|_L=0$, $|\g|=2$ and $d\g=0$. Assume the map $\underline{H_2}(X,L;\Z)\to R$ given by $\b\mapsto \int_{\b}\g$ descends to $\Pi$. Then, as currents, \[{Fi}_*{\phi^{\mO^{Fi}}}_*evi_1^*\g=\left(\int_\b\g\right)\cdot \phi(1).\]
That is,
\[
{Fi}_*{\phi^{\mO^{Fi}}}_*evi_1^*\g(\xi)=\left(\int_\b\g\right)\cdot \pi_*(\xi),\qquad\forall \xi\in A^{*}_c\left(\mM_{k+1,l-1}(\b),\pa^v \mM_{k+1,l-1}(\b)\right).
\]
\end{lemma}
\begin{proof}
The case where $\W=pt$ appears in
 \cite[Lemma 3.11]{Sara1}, noting that over a regular value of ${Fi}$, the relative orientation $\mO^{Fi}$ agrees with the orientation of the oriented real blow-up.
 
The general case is obtained from this special case as follows. Denote by $\a:=Fi_*\phi^{\mO^{Fi}}_*evi_1^*\g$. By Lemma~\ref{restriction and pushforward commute lemma} and the proof for $\W=pt$, we see that $\a|_t=\int_{\b|_t}\g|_t$ for all $t\in \W$. By Lemma~\ref{0 current with constant fibers is a function proposition} we obtain
\[
\a=\left(\int_\b\g\right)\cdot\phi(1).
\]
\end{proof}
\begin{proof}[Proof of Proposition~\ref{divisors}]
The proof is identical to the proof of \cite[Proposition 3.9]{Sara1}, recalling Proposition \ref{factorization q through forgetful theorem}.
\end{proof}
\subsection{Top degree}
Let $M$ be an orbifold with corners and $K$ a local system over $M$. Given $\a\in A(M;K)$ a homogeneous differential form, denote by $\deg^d(\a)$ the degree of the differential form, ignoring the grading of $K$. More generally, denote by $(\a)_j$ the part of $\a$ that has degree $j$ as a differential form, ignoring the grading of $R$. In particular, $\deg^d((\a)_j)=j$.
\begin{proposition}\label{top degree vanishes}
Suppose $(k,l,\b)\notin \{(1,0,\b_0),(0,1,\b_0),(2,0,\b_0)\}$. Then \[\left(i_t^*\left(\oqb{3}(\a;\g)\right)\right)_n=0\] for all lists $\a,\g$ and for all $t\in \W$, where $i_t:L_t\to L$ is the inclusion of the fiber over $t\in \W$.
\end{proposition}
\begin{proof}
Assume, without loss of generality, that $\a,\g$ are all homogeneous with respect to the grading $\deg^d$. Let $evb_j^{k+1},evi_j^{k+1}$, be the evaluation maps for $\mM_{k+1,l}(\b)$. Set
\begin{align*}
\xi:=&\bigwedge_{j=1}^l(evi_j^{k+1})^*\g_j\we\bigwedge_{j=1}^k(evb_j^{k+1})^*\a_j,\\
\xi':=&\bigwedge_{j=1}^l(evi_j^{k})^*\g_j\we\bigwedge_{j=1}^k(evb_{j-1}^{k})^*\a_j.
\end{align*}
Then $\oqb{3}(\a;\g)=\rho(\b;\a,\g)\left(evb_0^{k+1}\right)_*\left(\qor_{k,l}^\b\right)_*\xi.$ If $\deg^d\left(i_t^*\left(\oqb{3}(\a;\g)\right)\right)=n$, then
\[\deg^d\left(\left(i_t^{k+1}\right)^*\xi\right)=\dim \mM_{k+1,l}(\b)-\dim\W,\]
where we denote by $i^{k+1}_t:\mM_{k+1,l}(\b_t)\to \mM_{k+1,l}(\b)$ the inclusion of the fiber.

On the other hand, if $\pi:\mM_{k+1,l}(\b)\to \mM_{k,l}(\b)$ is the map that forgets the zeroth boundary point, then $\xi=\pi^*\xi'$. In particular, 
\[
\deg^d\left(\left(i^k_t\right)^*\xi'\right)=\deg^d\left(\left(i_t^{k+1}\right)^*\xi\right)=\dim\mM_{k+1,l}(\b)-\dim\W>\dim\mM_{k,l}(\b)-\dim\W.
\]
Therefore, $\left(i^k_t\right)^*\xi'=0$, and so $\left(i_t^{k+1}\right)^*\xi=0$. Therefore, $i_t^*\oqb{3}(\a;\g)=0.$
\end{proof}
\begin{proposition}
For all lists $\g=(\g_1,...,\g_l)$ we have
\[
\langle \fq_{0,l}(\g),1\rangle=\begin{cases}
0,&l\geq1,\\-\langle\g|_{L},1\rangle,&l=1.
\end{cases}
\]\begin{proof}
By Proposition~\ref{top degree vanishes}, the only contribution to $\langle \fq_{0,l}(\g),1\rangle$ is from $\fq_{0,1}^{\b_0}$. But $\fq_{0,1}^{\b_0}(\g_1)=-\g_1|_L$.
\end{proof}
\end{proposition}

\section{Conclusions}\label{Conclusions section}
Let $\target=(\W,X,\w,\pi^X,L,\fp,\underline\Upsilon,J)$ be a target.
Let $\g\in \mathcal I_{Q^\target}D^\target.$ Let $1^\target\in A^0(L)$ denote the constant function. Set
\[
\mathcal S^{\mathcal T,\g}:=(\fm^{\target,\g}_k,\langle,\rangle_{\text{odd}}^\target,1^\target).
\]

Theorem~\ref{A infinity algebra theorem} is the special case of the following theorem, in which $\W=\{*\}$. 
\begin{thm}[$A_\infty$ structure on $C$]\label{A-infty algebra conclusion thm}
$\mathcal S^{\target,\g}$ is a cyclic unital $n-1$ dimensional $A_\infty$-algebra structure on $C^\target$.
\end{thm}
\begin{proof}
Recall Definition~\ref{A infinity algebra definition}.
Properties~\ref{A infinity algebra definition: multilinearity},\ref{A infinity algebra definition: pairing bilinearity} follow from Proposition~\ref{lineraity}. Property~\ref{A infinity algebra definition: relations} follows from Proposition~\ref{A infty relations om proposition}. Properties~\ref{A infinity algebra definition: C valuation},\ref{A infinity algebra definition: R valuation} are immediate from the definitions. Properties~\ref{A infinity algebra definition: pairing antisymmetry},\ref{A infinity algebra definition: cyclic} follow from Propositions~\ref{antisymmetry of pairing} and~\ref{cyclic structure proposition}, respectively. Properties~\ref{A infinity algebra definition: unit k neq 0,2},\ref{A infinity algebra definition: unit k=2} follow from Proposition~\ref{unit of the algebra}. Property~\ref{A infinity algebra definition: unit k=2} follows from Proposition~\ref{zero energy}, Proposition~\ref{top degree vanishes} and because by assumption $\g|_L=0$.\end{proof}
\begin{remark}
In the case $\W=\{*\}$ and $L$ is oriented, let $\mO$ be a section of $\lort_{L}$, that is, an orientation for $L$.
Recall the local system $\rort_0 \subset\rort_L$ of even degrees, and set 
\[C_0^\target=A(L;\rort_0)\otimes \Lambda[[t_0,...,t_N]].\]
Set $\evenpairing{\cdot,\cdot}^\target=\oddpairing{\mO\cdot,\cdot}^\target$. The triple $\left(\{\fm_k^{\target,\g}\}_{k\geq 0},\evenpairing{,}^\target, 1^\target\right)$ is a cyclic unital $n$ dimensional $A_\infty$-algebra structure on $C_0^\target$. It is a scalar extension by $H^0(L;\rort_0)$ of the $A_\infty$-algebra constructed in \cite{Sara1}. The proof of the cyclic property remains the same, since the restriction to $C_0^\target$ implies all the signs in the calculations do not change.
\end{remark}

By Property (4), the maps $\fm_k$ descend to maps on the quotient
\[
\bar\fm^{\target,\g}_k:\overline{C^\target}^{\otimes k}\to \overline{C^\target}.
\]
Theorem~\ref{algebra deformation theorem} is the special case of the following theorem, in which $\W=\{*\}$. 
\begin{thm}
\label{algebra deformation conclusion theorem}
Suppose $\pa_{t_0}\g=1\in A^0(X,L)\otimes Q^\target$ and $\pa_{t_1}\g=\g_1\in A^2(X,L)\otimes Q$. Assume the map $H_2(X,L;\Z)\to Q^\target$ given by $\b\mapsto \int_\b\g_1$ descends to $\Pi^\target$. Then the operations $\fm_k^{\target,\g}$ satisfy the following properties.
\begin{enumerate}
    \item \label{algebra deformation conclusion theorem: fundamental class}(Fundamental class) $\pa_{t_0}\fm_k^{\target,\g}=-1\cdot\d_{0,k}.$
    \item \label{algebra deformation conclusion theorem: divisor}(Divisor) $\pa_{t_1}\fm_k^{\target,\g,\b}=\int_\b\g_1\cdot\fm_k^{\target,\g,\b}.$
    \item \label{algebra deformation conclusion theorem: energy zero}(Energy zero) The operations $\fm_k^{\target,\g}$ are deformations of the usual differential graded algebra structure on differential forms. That is,
    \[
    \bar\fm_1^{\target,\g}(\a)=d\a,\qquad \bar\fm_2^{\target,\g}(\a_1,\a_2)=(-1)^{|\a_1|}\a_1\we\a_2,\qquad \bar \fm_k^{\target,\g}=0,\quad k\neq 1,2.
    \]
\end{enumerate}
\end{thm}
\begin{proof}
Properties~\ref{algebra deformation conclusion theorem: fundamental class},\ref{algebra deformation conclusion theorem: divisor} and~\ref{algebra deformation conclusion theorem: energy zero} follow from Propositions~\ref{fundamental class proposition},~\ref{divisors} and~\ref{zero energy}, respectively.
\end{proof}

For $M\in\{\W,X,L\}$, let $\pi_M:M\times[0,1]\to M$ denote the projection, and for $t\in[0,1]$, let $j_t:M\to M\times[0,1]$ denote the inclusion $j_t(p)=(p,t)$.
Set
\begin{gather*}
   \mathfrak R^\target=A^*\left(\W\times [0,1];\pi_\W^*\efield_L\right)\otimes\tilde\La[[t_0,\ldots,t_N]],\\
\mathfrak C^\target=A^*\left(L\times [0,1];\pi_L^*\rort_L\right)\otimes\tilde\La[[t_0,\ldots,t_N]],\\
\mathfrak D^\target =A^*(X\times[0,1];Q).
\end{gather*}

The valuation $\nu^\target$ extends to valuations on $\mathfrak R^\target,\mathfrak C^\target$ and $\mathfrak D^\target$, and to valuations on their tensor products, which we also denote by $\nu^\target.$

\begin{definition}
Let $\mathcal S_1=(\fm,\prec,\succ,\mathbf{e})$ and $\mathcal S_2=(\fm',\prec,\succ',\mathbf{e}')$ be cyclic unital $A_\infty$ structures on $C^\target$. A cyclic unital \textbf{pseudoisotopy} from $\mathcal S_1$ to $\mathcal S_2$ is a cyclic unital $A_\infty$ structure $(\tilde\fm, {\pprec,\ssucc},\tilde{\mathbf{e}})$ on the $\mathfrak R^{\target}$-module $\mathfrak C^{\target}$ such that for all $\tilde \a_j\in \mathfrak C^{\target}$ and all $k\geq 0$,
\begin{align*}
j_0^*\tilde\fm_k(\tilde \a_1,\ldots,\tilde \a_k)=&\fm_k(j_0^*\tilde \a_1,\ldots,j_0^*\tilde \a_k),
\\
j_1^*\tilde\fm_k(\tilde \a_1,\ldots,\tilde \a_k)=&\fm'_k(j_1^*\tilde \a_1,\ldots,j_0^*\tilde \a_k),
\end{align*}
and
\begin{align*}
j_0^* {\pprec\tilde \a_1,\tilde \a_2\ssucc}=\prec j_0^*\tilde\a_1,j_0^*\tilde\a_2\succ,\qquad j_0^*\tilde{ \mathbf e}=\mathbf e,\\
j_1^* {\pprec\tilde \a_1,\tilde \a_2\ssucc}=\prec j_1^*\tilde\a_1,j_1^*\tilde\a_2\succ',\qquad j_1^*\tilde{ \mathbf e}=\mathbf e'.
\end{align*}
\end{definition}
Let $J'$ be another $\W$-tame vertical almost complex structure on $X$, and define the target
\[
\target'=(\W,X,\w,\pi^X,L,\fp,\underline\Upsilon,J').
\]
Let $\g,\g'\in \mathcal I_{Q^{\target}}D^{\target}$ be closed with $|\g|=|\g'|=2$.
Theorem~\ref{pseudoisotopy thm introduction} is the special case of the following theorem, in which $\W=\{*\}$. 
\begin{thm}\label{pseudoisotopy thm conclusion}
If $[\g]=[\g']\in \hat H^*(X,L;Q^{ \target})$, then there exists a cyclic unital pseudoisotopy from $S^{\target,\g}$ to $S^{\target',\g'}$.
\end{thm}
The following proof was inspired by that of~\cite[Theorem 2]{Sara1}.
\begin{proof}
Let $\mathcal J=\{J_t\}_{t\in[0,1]}$ be a family of $\w$-tame vertical almost complex structures on $X$ such that $J_0=J,J_1=J'$. Such $\mathcal J$ exists since the space of $\w$-tame vertical almost complex structures is contractible.
Set
\[
\mathfrak T=(\W\times[0,1],X\times[0,1],\pi_X^*\w,\pi^X\times\Id_{[0,1]},L\times [0,1],\pi_L^*\fp,\pi_\W^*\underline\Upsilon, \mathcal J).
\]
The octuple $\mathfrak T$ is a target over $\W\times[0,1]$. It satisfies $\target=j_0^*(\mathfrak T)$ and $\target'=j_1^*(\mathfrak T)$.

There is a canonical isomorphism
\[
\tilde \La^{\mathfrak T}\simeq\tilde \La^\target.
\]
Moreover, under the positive orientation of $[0,1]$, there is a canonical isomorphism
\[
\rort_{L\times[0,1]}\simeq \pi_L^*\rort_{L}.
\]
These isomorphisms induce canonical isomorphisms
\[
R^{\mathfrak T}\simeq \mathfrak R^\target,\qquad C^{\mathfrak T}\simeq \mathfrak C^\target. 
\]
Moreover, $D^{\mathfrak T}\simeq \mathfrak D^\target.$
The valuation $\nu^\target$ agrees with $\nu^{\mathfrak T}$.
Choose $\eta\in D^{\target}$ with $|\eta|=1$ such that $\g'-\g=d\eta.$ Take
\[
\tilde \g:=\g+t(\g'-\g)+dt\we\eta\in D^{\mathfrak T}.
\]
Then $|\tilde\g|=2$ and
\begin{gather*}
d\tilde \g=dt\we(\g'-\g)-dt\we d\eta=0,\\
j_0^*\tilde \g=\g,\qquad j_1^*\tilde\g=\g'.
\end{gather*}
From Propositions~\ref{naturality of q operators families} and~\ref{naturality of pairing families}, it follows that
$\mathcal S^{\mathfrak T}$
is a cyclic unital pseudoisotopy from $\mathcal S^\target$ to $\mathcal S^{\target'}$. 
\end{proof}


\newpage
\bibliographystyle{amsabbrvcnobysame1.bst}

\bibliography{bibliography.bib}
\end{document}